\documentclass[11pt,reqno, a4paper]{amsart}
\usepackage[english]{babel}
\usepackage{amsthm}
\usepackage[T1]{fontenc}
\usepackage[utf8]{inputenc}
\usepackage[english]{babel}
\usepackage{graphicx, color}
\usepackage{amssymb}
\usepackage{amsmath, xfrac}
\usepackage{latexsym}
\usepackage{caption, subcaption}
\usepackage{lineno}
\usepackage{mathtools}  
\usepackage{enumitem} 
\usepackage{hyperref}
\usepackage{cite}
\usepackage{complexity}
\usepackage{comment}

\newcounter{contSect} \numberwithin{contSect}{section}
 \numberwithin{contSub}{subsection}

\newtheorem{theorem}[contSect]{Theorem}
\newtheorem{corollary}[contSect]{Corollary}
\newtheorem{lemma}[contSect]{Lemma}
\newtheorem{proposition}[contSect]{Proposition}
\newtheorem{conjecture}[contSect]{Conjecture}
\newtheorem{claim}[contSect]{Claim}

\newtheorem{problem}[contSect]{Problem}

\usepackage[letterpaper,top=2cm,bottom=2cm,left=2.5cm,right=2.5cm,marginparwidth=1.75cm]{geometry}

\title{Minor-excluded graphs and soficity}
\author{Oriol Solé-Pi}
\thanks{Department of Mathematics, Massachusetts Institute of Technology. \textit{Email}: oriolsp@mit.edu}

\date{}                     

\begin{document}

\begin{abstract}
    A random rooted graph is said to be \textit{sofic} if it is the Benjamini-Schramm limit of a sequence of finite graphs. Given any finite graph $H$, we prove that every one-ended, unimodular random rooted graph that does not have $H$ as a minor must be sofic. The hypothesis regarding the number of ends can be dropped under the additional assumption that the graph is quasi-transitive. 
\end{abstract}

\small\maketitle

\section{Introduction}\label{sec:intro}

Since the pioneering works of Benjamini and Schramm\cite{benjamini2011recurrence} and Aldous and Steele~\cite{aldous2004objective}, the study of locally (i.e., Benjamini-Schramm) convergent sequences of finite sparse graphs, and their corresponding limit objects, has become an important area of study in probability, combinatorics, and theoretical computer science. In this setting, a sequence $\{G_n\}_{n\geq 1}$ of finite graphs is said to converge if, for every positive integer $r$ and every finite rooted graph $F_r$ of radius $r$, the probability that the $r$-neighborhood of a uniformly chosen vertex of $G_n$ is isomorphic to $F_r$ converges as $n\rightarrow\infty$. To each convergent sequence of finite graphs we can associate a limit object in the form of a random (possibly infinite) rooted graph.

The theory of local convergence provides a powerful framework to study several properties and parameters of sparse graphs, including the asymptotic number of spanning trees~\cite{lyons2005asymptotic} and matchings~\cite{bordenave2013matchings}, the critical percolation probability~\cite{benjamini2011critical}, PageRank~\cite{Garavaglia2018LocalWC}, and even the spread of epidemics~\cite{alimohammadi2022algorithms}. Moreover, there has been an important amount of work devoted towards trying to understand the limiting behavior of graphs drawn from multiple random graph models~\cite{bollobas2007phase,van2023local, alimohammadi2025local}. Local convergence has also been investigated for other structures, such as manifolds~\cite{abert2022unimodular} and permutations~\cite{borga2020local}.

Other prominent notions of convergence for sequences of large dense graphs and other discrete structures such as posets~\cite{janson2011poset} and permutations~\cite{hoppen2013limits} have also been studied. In~\cite{denselimits}, some of these notions have even been unified as part of a broader framework that makes it possible to treat convergence of dense discrete structures in great generality. One of the main successes in this broad venue of research has been the development of a comprehensive and powerful theory of convergence for sequences of dense graphs~\cite{lovasz2006limits,borgs2008convergent,borgs2012convergent2}. 

Just as the study of locally convergent sequences of graphs, this theory is deeply intertwined with the tasks of property testing and parameter estimation~\cite{borgs2008convergent,lovasz2010testing}. We refer the reader to the book of Lov{\'a}sz~\cite{lovaszlimits} for a careful treatment of the theory of graph limits in both the dense and the sparse settings, as well as extensive discussions on the parallels between these two.

One of the most striking differences between the theories of local convergence of sparse graphs and convergence of dense graphs is that our understanding of the local limit objects is rather poor. As mentioned earlier, every locally convergent sequence of finite graphs has as its limit a random rooted graph. However, not all random rooted graphs can arise as local limits of finite graphs: Let $G$ be a connected graph with no non-trivial automorphisms, and $\rho$ be an arbitrary vertex of $G$. Then, the random rooted graph which takes value $(G,\rho)$ with probability one cannot be obtained in this way. Those random rooted graphs that can be obtained via local limits are called \textit{sofic}. It is not hard to see that every sofic random rooted graph must enjoy a certain structural property known as \textit{unimodularity}---which, in particular, rules out the example we just mentioned. At an intuitive level, one can think of unimodularity as saying that each vertex of the graph has the same probability of appearing as the root (although this does not really make sense as stated); the precise definition will be given in Section~\ref{subsec:unimodular-sofic}. Aldous and Lyons~\cite{aldous2007processes} asked whether the classes of unimodular and sofic random rooted graphs coincide.

\begin{problem}[Aldous-Lyons]\label{Aldou-Lyons}
    Is every unimodular random rooted graph sofic?
\end{problem}

A negative answer to this question has been provided in a series of two papers by Bowen, Chapman, Lubotzky, Vidick~\cite{bowen2024aldous} and Bowen, Chapman, Vidick~\cite{bowen2024aldous2}. At the core of these contributions lies a modification of the $\MIP^*=\RE$ reduction from complexity theory first proven by Ji et al.~\cite{ji2021mip}. Let us remark that the solution to the Aldous-Lyons problem in~\cite{bowen2024aldous,bowen2024aldous2} is non-constructive, and finding a unimodular non-sofic random rooted graph remains a major open problem.

Despite these breakthroughs, our understanding of the family of sofic random rooted graphs remains far from complete. The purpose of this paper is to make some further progress in this direction by giving a positive answer to Problem~\ref{Aldou-Lyons} under additional structural assumptions on the random rooted graph. More precisely, we study the soficity of unimodular random rooted graphs which do not have some arbitrary graph as a minor.
The class of graphs which exclude some finite graph as a minor includes all trees and, more generally, all graphs that can be properly embedded on a surface of finite genus. Further background regarding the theory of graph minors is given in Section~\ref{subsec:structural}. Below are our two main results.

\begin{theorem}\label{thm:main}
    Fix any finite graph $H$. Let $(G,\rho)$ be a unimodular random rooted graph such that $G$ almost surely is one-ended and does not have $H$ as a minor. Then, $(G,\rho)$ is sofic.
\end{theorem}

A connected graph is \textit{one-ended} if and only if it is infinite and cannot be split into two (or more) infinite components by removing a finite set of vertices. The number of ends of a graph is defined more carefully in Section~\ref{subsec:ends}. It is well-known that unimodular random rooted graphs have zero, one, two, or infinitely many ends almost surely~\cite{aldous2007processes}. Moreover, graphs with zero ends are finite, and unimodular random graphs with two ends and bounded average degree are known to be hyperfinite (see Theorem~\ref{thm:ends}, as well as Section~\ref{subsec:expansion-amenable}), and thus sofic. Hence, whenever $(G,\rho)$ has finite average degree, the one-endedness assumption can be substituted by the weaker condition that the number of ends is a.s.\footnote{Throughout the paper, we use \textit{a.s.}\ as a shorthand for "almost surely".} finite. Using some recent results of Esperet, Giocanti, Legrand-Duchesne~\cite{esperet2024structure} and Jard{\'o}n-S{\'a}nchez~\cite{jardon2023applications}, we are able to remove this assumption altogether under the additional hypothesis that the graph is (vertex) quasi-transitive.

\begin{theorem}\label{thm:main-transitive}
    Fix any countable graph $H$. Let $(G,\rho)$ be a unimodular random rooted graph such that $G$ a.s.\ is quasi-transitive and does not have $H$ as a minor. Then, $(G,\rho)$ is treeable, and hence sofic. 
\end{theorem}

A unimodular random rooted graph is \textit{treeable} if it admits a unimodular random rooted spanning tree; all treeable unimodular random rooted graphs are known to be sofic (see Section~\ref{subsec:previous-work}). Perhaps surprisingly, our proof of this second result is considerably shorter than the proof of Theorem~\ref{thm:main}. This is because the quasi-transitivity allows us to use as a black box some results from~\cite{esperet2024structure} regarding the existence of certain highly-structured tree decompositions of the graph (see Section~\ref{sec:transitive} for details). As such, the majority of this paper is devoted to the proof of the first of the two theorems above, which we consider to be our main contribution.  Our proofs can be adapted to imply that any unimodular random rooted graph satisfying the hypotheses of either of the two theorems is strongly sofic, as defined in~\cite{angel2018hyperbolic} (see Section~\ref{subsec:markings}). 

\medskip\noindent\textit{Remark.} It is also important to mention that a sequence of finite graphs which converges in the Benjamini-Schramm sense to a unimodular random rooted graph with no $H$-minor may consist of graphs which are themselves very far from being $H$-minor-free. For example, for any $d\geq 3$, a sequence of uniform random $d$-regular graphs of increasing orders will converge a.s.\ to the rooted infinite $d$-regular tree. The infinite tree is clearly $H$-minor-free for every graph $H$ which contains at least one cycle; yet, for any such $H$, a.s.\ all sufficiently large graphs in the sequence will have $H$ as a minor---this is implied by the results in~\cite{kleinberg1996short}, but can also be proven through simpler methods. In fact, unimodular random rooted graphs which arise as limits of finite $H$-minor-free graphs are not particularly interesting in our setting: they are precisely those unimodular random rooted graphs which are $H$-minor-free and hyperfinite (this is a simple consequence of the fact that finite graphs with an excluded minor admit small separators~\cite{alon1990separator} in conjunction with the results in~\cite{schramm2011hyperfinite}).

\subsection{Previous work}\label{subsec:previous-work} 
As we have mentioned, the solution to the Aldous-Lyons problem in~\cite{bowen2024aldous,bowen2024aldous2} expands upon previous work of Ji et al.~\cite{ji2021mip}. Already in~\cite{ji2021mip}, a similar strategy was used in order to refute the Connes embedding conjecture, thus solving a major open problem in the theory of von Neumann algebras. We refer the reader to~\cite{goldbring2022connes} for a friendly introduction to the complexity classes $\MIP^*$ and $\RE$, as well as to the Connes embedding problem and its solution.

In the positive direction, there have been multiple works studying properties of unimodular random rooted graphs which imply soficity. One of the foundational results in this line of research tells us that if $(G,\rho)$ is a unimodular random rooted graphs such that $G$ is a.s.\ a tree, then it is sofic; this is due to Bowen~\cite{bowen2003periodicity}, Elek~\cite{elek2010limit}, and Benjamini, Lyons, Schramm~\cite{benjamini2015unimodular}. 
This result was later extended from trees to treeable unimodular random rooted graphs by Elek and Lipner~\cite{elek2010sofic}. This constitutes one of the most powerful tools available for proving soficity. Yet another very general property which is known to imply soficity is hyperfiniteness (see Section~\ref{subsec:expansion-amenable} for the definition). Hyperfinite unimodular random rooted graphs include Cayley graphs of amenable groups, as well as unimodular graphs of polynomial growth. Note that trees are precisely those graphs that do not have the cycle of length $3$ as a minor. Thus, Theorem~\ref{thm:main-transitive} is a generalization of the fact that unimodular, quasi-transitive, random rooted trees are sofic.

In a beautiful paper, Angel, Hutchcroft, Nachmias and Ray~\cite{angel2018hyperbolic} studied the geometry of unimodular random rooted maps\footnote{A \textit{unimodular random rooted map} consists of a unimodular random rooted graph where each vertex carries some additional information in the form of a cyclic permutation of its neighbors. The permutation encodes the order in which the edges appear around this vertex in some specific map, and the collection of all these cyclic permutations uniquely determines the combinatorial structure of the map. See~\cite{angel2018hyperbolic} for details.} and, among many other results, they showed that every unimodular, one-ended planar map of bounded average degree is treeable, and thus sofic. This was subsequently extended to all unimodular one-ended planar graphs of bounded average degree by Timár~\cite{timar2023unimodular},\footnote{In~\cite{timar2023unimodular}, this result is stated without any assumptions on the average degree. However, the proof of this stronger version is flawed: the paper claims that it is sufficient to prove the result for unimodular random graphs of bounded maximal degree, since every unimodular random graph arises as a limit of a sequence of such graphs; unfortunately, the one-endedness might not be preserved after passing to this sequence.} and later to even more general classes of planar graphs by Jardón-Sánchez\cite{jardon2023applications}. All these results rely on powerful tools that are available only when working with planar graphs, such as duality. Since planar graphs do not have $K_5$ nor $K_{3,3}$ as a minor, Theorem~\ref{thm:main} constitutes a generalization of the fact that unimodular one-ended planar graphs of bounded average degree are sofic, both in that it allows for any arbitrary forbidden minor, and in that it removes the assumption about the average degree. 

Recently, it has also been shown by Jardón-Sánchez~\cite{jardon2023applications} and Chen, Poulin, Tao, and Tserunyan~\cite{chen2023tree} that every locally-finite Borel graph all of whose connected components have bounded treewidth must be treeable (the definition of treewidth will be given in Section~\ref{subsec:structural}); in fact, it is shown in~\cite{chen2023tree} that the conclusion also holds whenever the components are quasi-isometric to trees. In our setting,\footnote{We refer the reader to Section~\ref{sec:transitive} or~\cite[Chapter 18]{lovaszlimits} for a discussion on the connection between the settings of Borel graphs/graphings and random rooted graphs.} the first of these facts implies that every unimodular random rooted graph whose treewidth is a.s.\ bounded must be treeable. By a well-known result from structural graph theory\cite{robertson1986graph}, we conclude that for every finite planar graph $H$, a unimodular random rooted graph which a.s.\ does not have $H$ as a minor must be treeable. A "streamlined" version of the proof in~\cite{chen2023tree} has been given in~\cite{zhai2024quassi-treeings}.

Recently~\cite{lyons2025note}, the results from~\cite{bowen2024aldous,bowen2024aldous2} have been used to address a certain finite analogue of Problem~\ref{Aldou-Lyons} due to Lovász~\cite[Chapter 19]{lovaszlimits}.

\subsection{Sofic groups}\label{subsec:groups} Soficity was originally introduced in the realm of group theory by Gromov~\cite{gromov1999endomorphisms} and Weiss~\cite{weiss2000sofic} as a common generalization of amenability and residual finiteness. Although we will work only in the setting of unimodular random rooted graphs, we also include the group-theoretic definition of soficity below, in order to remark on the implications of our main results in this setting.

Let $S(n)$ denote the permutation group on $n$ elements and let $e_{S(n)}$ be its identity element. For any $\sigma_1,\sigma_2\in S(n)$, define $d_{S(n)}(\sigma_1,\sigma_2):=|\{i\in[n]:\sigma_1(i)\neq\sigma_2(i)\}|/n$. A discrete group $\Gamma$ is said to be \textit{sofic} if, for any finite set $F\subseteq\Gamma$ containing the identity element $e_\Gamma$ and any $\varepsilon>0$, there exist an $n\in \mathbb N$ and a map $\phi:\Gamma\rightarrow S(n)$ with the following properties: 
\begin{itemize}
    \item $\phi$ maps the identity in $\Gamma$ to the identity in $S(n)$;

    \item $d_{S(n)}(e_{S(n)},\phi(g))=1$ for all $g\in F\backslash\{e_\Gamma\}$;

    \item $d_{S(n)}(\phi(gh),\phi(g)\phi(h))\leq \varepsilon$ for all $g,f\in \Gamma$ such that $gh\in F$. 
\end{itemize} Intuitively, these properties can be interpreted as saying that $\phi$ behaves as an "approximate non-trivial homomorphism" when restricted to $F$. Clearly, a group is sofic if and only if all of its finitely generated subgroups are sofic. 

Let us briefly touch as well on the connection between the two notions of soficity we have seen thus far. We begin by introducing Cayley graphs of groups. Given a finitely generated group $\Gamma$ and a set of generators $S$, the \textit{(right) Cayley graph} $\operatorname{Cay}(\Gamma,S)$ associated to the pair $(\Gamma,S)$ is the edge-labeled graph whose vertex set is $\Gamma$ and where two vertices $x$ and $y$ are connected by an edge with label $x^{-1}y$ if and only if $x^{-1}y\in S$. By rooting $\operatorname{Cay}(\Gamma,S)$ at $e_\Gamma$, every locally finite Cayley graph gives rise to a unimodular random rooted graph (i.e., the random rooted graph which takes value $(\operatorname{Cay}(G,S),e_\Gamma)$ with probability $1$). It is well-known that a finitely generated group is sofic if and only if it admits a Cayley graph which is sofic (in the sense that the unimodular random rooted graph it induces is sofic), and that this occurs if and only if all of its Cayley graphs are sofic. In the group theory literature, sofic Cayley graphs are often referred to as being \textit{initially subamenable}. Our result that quasi-transitive minor-excluded unimodular random rooted graphs are strongly sofic has the following consequence.

\begin{corollary}\label{cor:groups}
    Every finitely generated group admitting a minor-excluded Cayley graph is sofic.
\end{corollary}

\noindent\textit{Remark.} As observed in~\cite{ostrovskii2015metric}, minor-exclusion need not be a property of the group itself, in the sense that there are groups---such as the additive group $\mathbb{Z}^2$---which admit Cayley graphs that are minor-excluded while also admitting Cayley graphs containing every countable graph as a minor.

Gromov~\cite{gromov1999endomorphisms} asked the following question, which remains open to this day. 

\begin{problem}[Gromov]\label{sofic-groups}
    Is every group sofic?
\end{problem}

It is widely believed that the answer should be negative. A possitive answer to the Aldous-Lyons problem would have also implied a positive answer to Problem~\ref{sofic-groups}. The interest in sofic groups has increased drastically since their introduction, mainly due to the growing number of important properties that we now know they satisfy. For example, it has been shown that the Gottschalk surjunctivity conjecture~\cite{kerr2011entropy} and the Kaplansky conjecture~\cite{elek2003sofic} hold for all sofic groups (see, also,~\cite{luck2002l2,pestov2008hyperlinear,thom2008sofic,bowen2010measure}). We refer the reader to~\cite{capraro2015introduction} for a more careful treatment of the rich theory of sofic groups.

 While solving Problem~\ref{sofic-groups} has proven to be extremely challenging, some partial progress has been made towards a negative answer. In~\cite{de2020stability} and~\cite{lubotzky2020non}, some groups were shown to not be approximable by finite-dimensional unitary groups $U(n)$ under two different kinds of approximations (both of these notions of approximability are similar in spirit to soficity; we will not discuss the precise definitions here). Even more recently, a promising strategy to prove the existence of non-sofic groups has been put forward in~\cite{chapman2025stability} and~\cite{gohla2024high}. 

Just as in the case of unimodular graphs, there are several properties of groups which are known to imply soficity. As we have already mentioned, two of the most important such properties are amenability and residual finiteness. We will discuss graph amenability in Section~\ref{subsec:expansion-amenable}. A group $\Gamma$ is \textit{residually finite} if, for every nontrivial element $g$, there exists a  homomorphism from $\Gamma$ to a finite group which does not map $g$ to the identity. Notably, the class of residually finite groups is an extension of the class of free groups. The fact that free groups are sofic has been known for a long time, and also follows from the fact that unimodular random rooted trees are sofic. It is known that soficity is preserved under several group operations, including cartesian products, free products and amalgamated free products over amenable subgroups~\cite{elek2011sofic}, wreath products~\cite{hayes2018metric}, and graph products~\cite{ciobanu2014sofic}. Other recent developments surrounding soficity of groups include~\cite{lazarovich2024surface,gao2024soficity,gao2024all,gao2024sofic}.  

Some aspects of the interplay between group structure and minor-exclusion have also been studied by multiple authors. For instance, it was shown in~\cite{ostrovskii2015metric} that every finitely generated, one-ended group admits a Cayley graph that is not minor-excluded, and some time later it was proven that a finitely generated group is virtually free if and only if each of its Cayley graphs is minor-excluded~\cite{khukhro2023characterisation} (see~\cite{hamann2024minor} for an extension of this result to the setting of quasi-transitive graphs). Particularly related to our paper are the results from~\cite{conley2021one}, which, among other things, imply that every finitely generated group admitting a planar Cayley is treeable. Note that Corollary~\ref{cor:groups} generalizes this result. The proof in~\cite{conley2021one} relies on certain algebraic structure (which goes by the name of a $2$-basis) that appears to not be available for more general classes of minor-excluded Cayley graphs.

\subsection{About our proofs}\label{subsec:about-proofs} The proofs of theorems~\ref{thm:main} and~\ref{thm:main-transitive} rely heavily on the notions of treewidth and tree decompositions, both of which have been extremely fruitful within the field graph theory since their reinvention and popularization by Robertson and Seymour~\cite{robertson1984graph}. We are not the first to study the interplay between treewidth and soficity. Indeed, as we already mentioned, it has been shown in~\cite{jardon2023applications} and~\cite{chen2023tree} that every unimodular random rooted graph of bounded treewidth is treeable, and hence also sofic. This result will play a crucial role in our proofs.

Let us provide a very informal description of the strategy that we will follow in order to prove Theorem~\ref{thm:main} (a more detailed sketch will be given in Section~\ref{sec:strategy}). For this presentation, we restrict our attention to the case of planar graphs where every finite set of vertices expands by at least a positive constant factor, meaning that every finite set of vertices $S$ contains $\Omega(|S|)$ elements which are adjacent to some vertex outside $S$. The planar setting allows us to visualize some of the main ideas in the proof, and the expansion assumption turns out to be non-essential. We remark that the strategy we follow is very different from the one used in any of the works we have already mentioned. We hope that this sketch will make it easier for the reader to follow the rest of the paper. 

Consider an infinite, one-ended, expanding planar graph $G$ sampled according to a unimodular measure on rooted graphs. Throughout this graph, we will select several finite, pairwise disjoint sets of vertices, which will be called \textit{filaments}. The picture one should have in mind is that each of these filaments spans a finite and tree-like connected subgraph of $G$, which is large but rather sparse. These sets of vertices will be constructed by means of a randomized local algorithm, which decides whether a vertex will belong to a filament by observing a ball of bounded radius around it. We then delete from the graph each vertex that belongs to one of these filaments, obtaining a new graph $G_{\operatorname{thin}}$. If the filaments are chosen appropriately, then $G_{\operatorname{thin}}$ will have locally-small treewidth; this is exemplified in Figure~\ref{fig:filaments_intuition}. More precisely, we will construct the filaments so that there exist two positive integers $t$ and $r$, with $r$ much larger than $t$, such that every ball of radius less than $r$ in $G_{\operatorname{thin}}$ has treewidth at most $t$. Intuitively, this occurs because the portions of the graph that lie in between the filaments resemble "thin corridors", and the macroscopic cycles formed by these corridors will show up only within balls of $G_{\operatorname{thin}}$ which are large enough to contain an entire filament in their interior. We claim that any such ball must be very large. Indeed, suppose that $C$ is a cycle in $G_{\operatorname{thin}}$ which encloses at least one filament, and let $\overline{C}$ denote the set of vertices of $G$ which either belong to $C$ or lie in the finite region enclosed by this cycle. Since $G$ is one ended, the set $\overline{C}$ must be finite. On the other hand, since each of the filaments contains many vertices, $\overline{C}$ must also contain many vertices. Given that $G$ is an expander, the set of elements of $\overline{C}$ which are adjacent to some vertex outside $\overline{C}$ must have size at least some constant times $|\overline{C}|$. All of these elements are actually in $C$, so $C$ must be a very long cycle, as desired. In this interpretation, $t$ corresponds to the width of the corridors, and $r$ is the smallest radius of a ball in $G_{\operatorname{thin}}$ which encloses at least one of the filaments.  

\begin{figure}[ht]
    \centering
    \includegraphics[width=.97\linewidth]{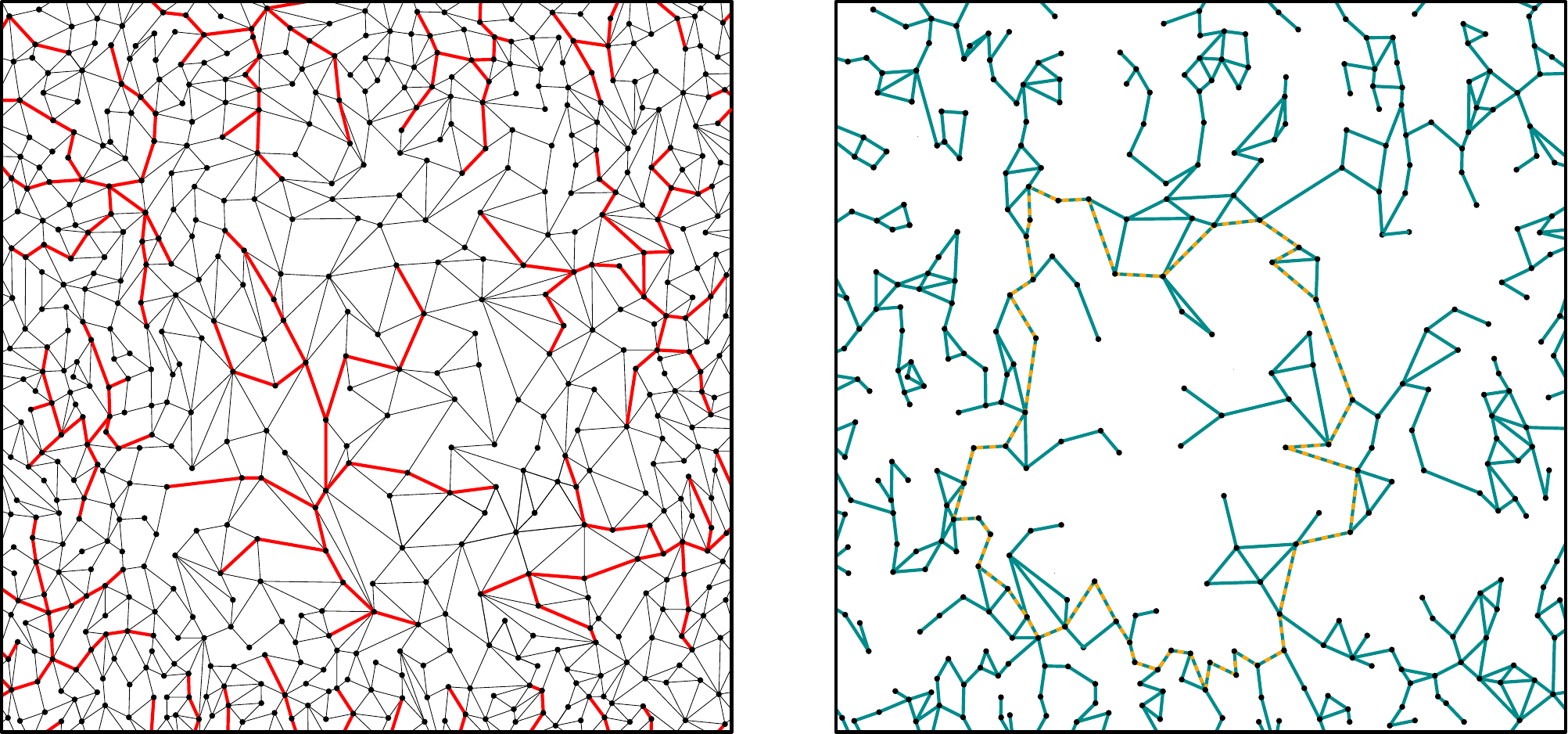}
    \caption{On the left, we have depicted a portion of an infinite, one-ended, expanding, planar graph $G$. Some large but finite tree-like structures, called filaments, have been selected throughout the graph, and are highlighted in red. After removing these filaments from the graph, we obtain a new graph $G_{\operatorname{thin}}$, a portion of which is shown on the right. If the filaments are chosen appropriately, then $G_{\operatorname{thin}}$ will have locally-small treewidth, in the sense that any ball in this graph whose radius is not too large will have small treewidth. In essence, the only structures within $G_{\operatorname{thin}}$ which are stopping the graph from having small treewidth are the macroscopic cycles which enclose at least one of the filaments; one such cycle has been marked with dotted yellow lines. Since $G$ is an expander and the filaments are large, any such cycle must also be large.}
    \label{fig:filaments_intuition}
\end{figure}

The fact that $G$ is an expander also suggests that most of the vertices of $G$ should lie somewhat far (say, at distance at least $t/10$) from the filaments. More specifically, the expansion guarantees that the set of vertices of $G$ which lie at distance at most $t/10$ from some individual filament is larger than the filament itself by at least a multiplicative factor which depends only on $t$ and can be made arbitrarily large by making $t$ large. This phenomenon will essentially allow us to make the sizes of the filaments arbitrarily large, while keeping fixed both $t$ and the fraction of vertices that belong to a filament. Making the filaments larger amounts to letting $r$ go to infinity. Thus, by taking the limit of this construction, one can arrive at a new family of \textit{super-filaments}, which are infinite and whose deletion results in a graph of bounded treewidth. 

The next and final step consists of iterating the above construction while letting the width of the corridors (i.e., $t$) go to infinity. Crucially, while doing this, the fraction of vertices that belong to one of the filaments will converge to $0$. As the density of the filaments tends towards $0$, the graphs that arise after deleting the super-filaments will converge in the Benjamini-Schramm sense towards the original unimodular random rooted graph. Furthermore, each graph in this sequence has bounded treewidth---although the treewidth gets larger the further down we go in this sequence---and is thus sofic (as shown in~\cite{jardon2023applications} and~\cite{chen2023tree}). It follows that the initial unimodular measure must itself be sofic. A large portion of our paper is devoted to making the construction of the filaments and the properties of $G_{\operatorname{thin}}$ precise, and we consider this to be our main contribution at a technical level.

A more detailed sketch of this proof will be given in Section~\ref{sec:strategy}.

\subsection{Subsequent work}\label{subsec:subsequent-work} The most natural question left unanswered by our work is whether the one-endedness assumption can be removed from the statement of Theorem~\ref{thm:main}; we believe this to be the case.

\begin{conjecture}\label{conj:many-ends}
    For any finite graph $H$, every unimodular random rooted graph $(G,\rho)$ such that a.s.\ $G$ does not have $H$ as a minor is sofic.
\end{conjecture}

A less ambitious goal would be to prove this statement under the additional assumption that $G$ is a.s.\ an accessible graph, as defined in Section~\ref{thm:ends}. It is also possible that every unimodular random rooted graph as above is actually treeable. 

\begin{problem}
   Let $H$ be a finite graph. Is is true that a unimodular random rooted graph $(G,\rho)$ such that $G$ a.s.\ does not have $H$ as a minor must be treeable?
\end{problem}

Anush Tserunyan has independently considered this question, as well as some generalizations of it~\cite{tserunyan_communication}. We remark that this problem remains open even in the one-ended case (as mentioned earlier, it is known to hold for every planar, one-ended, unimodular graph of bounded average degree~\cite{timar2023unimodular}). It is possible that one can reduce Conjecture~\ref{conj:many-ends} to Theorem~\ref{thm:main} by constructing some sort of decomposition of $(G,\rho)$ into one-ended graphs, perhaps in the style of Stalling's theorem~\cite{stallings1971group} for groups with more than one end. Multiple decomposition results along these lines have been obtained for graphs~\cite{tserunyan2018descriptive,hamann2022stallings,ishikura2025decomposition}, but it is not clear to us whether, in their current form, they can be directly applied to our setting.

In some sense, minor-excluded graphs can be thought of as being two-dimensional. This intuition is supported by the remarkable graph structure theorem of Robertson and Seymour, as well as a more recent result by Bonamy et al.~\cite{bonamy2023asymptotic} regarding the asymptotic dimension of minor-excluded graph classes. It was suggested some time ago by Lovász~\cite{lovasz2006graph} that higher-dimensional analogues of minor exclusion should be investigated. Wagner~\cite{wagner2013minors} also asked whether there is a good way of extending the study of minor-excluded graphs to higher-dimensional simplicial complexes. These research directions have seen some success; see~\cite{carmesin2023embedding} and the references therein. It seems possible that some extension of our techniques (in particular, the construction of the filaments mentioned in the previous section) could be used to reduce the "intrinsic dimension" of either graphs or simplicial complexes by deleting only an arbitrarily small fraction of their vertices. Regardless of whether this is correct or not, we believe that the interplay between dimension and soficity deserves to be studied further, and we pose the following question.
\begin{problem}
    Let $(G,\rho)$ be a unimodular random rooted graph such that $G$ can a.s.\ be represented as the skeleton of some $3$-dimensional simplicial complex which is homeomorphic to $\mathbb R^{3}$ (and which might depend on $G$). Is it true that $(G,\rho)$ must be sofic?
\end{problem}

Note that if $(G,\rho)$ satisfies the properties in the above statement, then $G$ must a.s.\ be one-ended. Since every unimodular random rooted graph which is a.s.\ planar and one-ended must also be sofic, the answer to the two-dimensional version of this question is positive.

              To conclude this section, we step away from the study of soficity in order to briefly discuss some potential algorithmic applications of our work. In recent years, the interplay between descriptive combinatorics and distributed algorithm has become apparent; see~\cite{bernshteyn2022descriptive} for a detailed treatment of the topic. At a high level, this connection stems from the fact that most constructions in descriptive combinatorics can be phrased in terms of local distributed algorithms. The construction of the filaments that we will present in Section~\ref{sec:Voronoi and Starfish} falls within this class of local algorithms. Given that several algorithmic tasks are known to be easier to perform on graphs of bounded treewidth (see Courcelle's Theorem~\cite{courcelle1990monadic,conley2021one}, as well as some more recent developments~\cite{alon2010solving,peters2016graphical}), it is possible that our proof techniques could be applied to design novel distributed algorithms for graphs with no shallow $H$-minor (intuitively, a graph has no shallow $H$-minor if every small ball in the graph is $H$-minor-free; the precise definition is given in Section~\ref{subsec:structural}). 

\subsection*{Acknowledgments}
I wish to thank Tom Hutchcroft, \'Adám Timár, Michael Chapman, Anush Tserunyan, and Antoine Poulin for multiple insightful conversations. I am also very grateful to Nike Sun and Jacopo Borga, whose comments helped improve the presentation of this paper significantly.

\section{Proof overview}\label{sec:strategy}

Here, we present a more detailed overview of the proof of Theorem~\ref{thm:main}, which we believe captures most of the important ideas.

Fix a finite graph $H$ and let $(G,\rho)$ denote a unimodular random rooted graph such that $G$ is a.s.\ one-ended and does not have $H$ as a minor. By deleting an arbitrarily small fraction of the vertices of $G$, we may assume that the maximum degree of $G$ is uniformly bounded from above by some integer $D\geq 2$. This operation could cause $G$ to no longer be one-ended a.s., but this turns out to not be a serious issue, since all that we will need is that the initial graph is one-ended. Our end goal consists of showing that there exists a sequence of treeable unimodular random rooted graphs which converge to $(G,\rho)$ in the Benjamini-Schramm sense. 
As we mentioned during the introduction, all treeable unimodular random rooted graphs are sofic. Since limits of sequences of sofic graphs must themselves be sofic, this would yield the desired conclusion. Below, we outline the main steps involved in the construction of a sequence of treeable graphs that approximate $(G,\rho)$.

\subsection{Cutting out non-expanding sets} 

Fix some $\alpha>0$. First, we show that by deleting a (possibly random) set of vertices $\eta_\alpha$ with $\mathbb P_{(G,\rho),\eta_\alpha}[\rho\in\eta_\alpha]\leq \alpha$, $G$ can be partitioned into connected components which are either finite or have (outer) vertex expansion at least $\alpha$, meaning that every finite set of vertices $S$ within one of these components is connected to at least $\alpha |S|$ vertices which belong to the same connected component but not to $S$. For technical reasons, during the actual proof of this result (Theorem~\ref{thm:non-amenable}) we will instead work with inner vertex expansion, but we remark that both of these notions are equivalent up to a factor of $D$. Here, one should interpret $\mathbb P_{(G,\rho),\eta_\alpha}[\rho\in\eta_\alpha]$ as the fraction of vertices that belong to $\eta_\alpha$. The existence of $\eta_\alpha$ is probably knwon to several experts in the field, but we have not seen it explicitly stated anywhere. We remark that, for amenable groups, it is well known that all the components can actually be made finite by removing an arbitrarily small fraction of the vertices (i.e., amenable groups are hyperfinite); this immediately implies that such groups are sofic. 

At an intuitive level, if $G$ does not directly satisfy the desired expansion property, then it must contain some large but finite sets of vertices whose vertex boundaries are---comparatively---very small, and we should be able to separate these poorly-expanding sets from the rest of the graph by means of a percolation that deletes their boundaries. This procedure can then be repeated until we are left with some expanding infinite connected components and some residual finite ones. If the fraction of vertices that have been deleted is small, then this new graph must be close to $(G,\rho)$ with respect to the local topology. 
It then remains to prove that the fraction of deleted vertices, and thus also the approximation error, can be made arbitrarily small. Making all of these arguments precise requires some care---and the mass transport principle, which is discussed in Section~\ref{subsec:unimodular-sofic}. One of the main challenges consists of ensuring that the poorly-expanding sets that are being cut off during a single step do not overlap, as otherwise the union of their boundaries could turn out to contain far too many vertices.

\subsection{Voronoi cells and filaments}\label{subsec:filaments}

Suppose now that $(G,\rho)$ is a unimodular random rooted graph such that $G$ a.s.\ is one-ended and does not have $H$ as a minor, and that $\eta_\alpha$ is such that the graph $G_\alpha=G\backslash\eta_\alpha$ has expansion at least $\alpha$ (we ignore the finite connected components for the remainder of this exposition, as they are easy to handle). Fix a parameter $\varepsilon\in (0,1)$ along with some large positive integer $R$ and let $S$ be a maximal $R$\textit{-separated} set of vertices within $G_\alpha$, meaning that any two of its elements are at distance at least $R$ in $G_\alpha$. This construction of $S$ can be carried out in a measurable way, in the sense that the probability (with respect to the choice of $(G,\rho)$, $\eta_\alpha$, and $S$) that the root $\rho$ belongs to $S$ is well defined. Now, we consider the \textit{Voronoi partition} of $G_\alpha$ with respect to $S$. That is, we assign each vertex of $G$ to the closest element of $S$, where ties are broken at random. For each $s\in S$, let the \textit{Voronoi cell} of $s$, $V_s$, be the set of vertices that have been assigned to $s$. The subgraph $G_\alpha[V_s]$ of $G_\alpha$ induced by $V_s$ is finite and connected, contains the ball $B(s,\lfloor R/2\rfloor)$ of radius $\lfloor R/2\rfloor$ around $s$, and is contained in the ball $B(s, R)$. Since every finite set of vertices in $G_\alpha$ expands by a factor of at least $\alpha$, we get that $|V_s|\geq|B(s,\lfloor R/2\rfloor)|\geq(1+\alpha)^{\lfloor R/2\rfloor}$.

Choose some positive integer $h$ with $h(1+\alpha)^{\lfloor h/2\rfloor}\leq\varepsilon$; one should think of $h$ as being much smaller than $R$. We will construct one filament within each of the Voronoi cells. We start by selecting a maximal $h$-separated subset $W_s\subset V_s$. Essentially, the filament $\Psi_s$ within $V_s$ will consist of those vertices that belong to a tree of minimal size in $G[V_s]$ which contains every element of $W_s$ (during the actual proof, the filaments will be constructed in a slightly different manner in order to avoid running into some issues related to the boundary of $V_s$). Note that each filament is connected by definition. Furthermore, the fact that every vertex of $G_\alpha$ has degree at most $D$ implies that the set of vertices in $V_s$ which are at distance at most $h$ from any individual element of $W_s$ is less than $D^{h+1}$. By the maximality of $W_s$, this yields $|W_s|\geq |V_s|/D^{h+1}$, which in turn implies $|\Psi_s|\geq |V_s|/D^{h+1}$. On the other hand, the balls of radius $\lfloor h/2\rfloor$ around the elements of $W_s$ are pairwise disjoint. Using once again the fact that $G_\alpha$ is an expander (and ignoring for now the fact that some of these balls might not be fully contained within $V_s$), we deduce that $|W_s|\leq |V_s|/(1+\alpha)^{\lfloor h/2\rfloor}$. Lastly, using the maximality of $W_s$ one more time, one can conclude that \[|\Psi_s|\leq 2h(|W_s|-1)+1<2h|V_s|/(1+\alpha)^{ \lfloor h/2\rfloor}+1\leq 2\varepsilon|V_s|\,,\] where the last inequality follows from the choice of $h$. Most importantly, we have obtained the two following bounds:
\begin{enumerate}[label=(\roman*)]
    \item $|\Psi_s|\geq|V_s|/D^{h+1}\geq(1+\alpha)^{\lfloor R/2\rfloor}/D^{h+1}$;
    \item $|\Psi_s|\leq 2\varepsilon|V_s|$.
\end{enumerate}

Next, let $\omega:=V\backslash \bigcup_{s\in S}\Psi_s$ denote the set of vertices that do not belong to any of the filaments, and write $G_\alpha[\omega]$ to denote the subgraph of $G_\alpha$ induced by $\omega$. Using the mass transport principle, together with property (ii) above, one can easily deduce that $\mathbb P[\rho\not\in \omega]\leq 2\varepsilon$. Intuitively, if within each Voronoi cell we are deleting at most a $2\varepsilon$ fraction of the vertices, then the same should be true globally. Note that $\varepsilon$ and $h$ can be kept constant while letting $R$ (and thus also $|\Psi_s|$) go to infinity. In fact, it will be helpful to think of $\varepsilon$ and $h$ as being fixed for the time being.

A key property of our construction, which should become clear later on, is that both the filaments and the set $\omega$ can be encoded by means of a marking, as defined in Section~\ref{subsec:markings}.

\subsection{Locally-bounded treewidth}\label{subsec:locally_treelike}

As mentioned in Section~\ref{subsec:about-proofs}, we shall prove that $G_\alpha[\omega]$ is \textit{locally-thin}, in the sense that there exist some positive integers $t$ and $r$, with $t$ depending only on $h$ and $r$ growing with $R$, such that every ball of radius $r$ in $G_{\alpha}[\omega]$ has treewidth at most $t$. The main tool that will be used to control the treewidth of balls within $G_\alpha[\omega]$ is the grid-minor theorem of Robertson and Seymour~\cite{robertson1986graph} (see Section~\ref{subsec:structural} and, in particular, Theorem~\ref{thm:treewidth-walls}). This foundational result tells us that a there exists a function $g$ of $t$ such that any graph which does not contain a subdivision of the $g(t)\times g(t)$-wall, $W_{g(t)\times g(t)}$, has treewidth at most $t$ (the $7\times 7$-wall has been depicted in Figure~\ref{fig:wall} for the reader's convenience). With this result in hand, our goal is to show that any subdivision of $W_{g(t)\times g(t)}$ within $G_\alpha[\omega]$ must contain so many vertices that it cannot possibly fit within a ball of radius $r$. 
\begin{figure}[ht]
    \centering
    \includegraphics[width=0.37\linewidth]{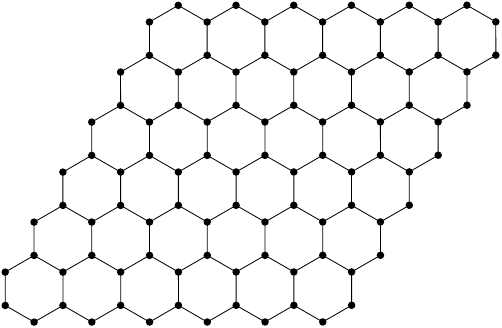}
    \caption{The $k\times \ell$-wall, $W_{k\times\ell}$, is simply the skeleton of a hexagonal grid with $k$ rows and $\ell$ columns. The $7\times 7$ wall is shown above.}
    \label{fig:wall}
\end{figure}

Choose $t$ so that $g(t)$ is far larger than $h$, and suppose that there is some subdivision $W$ of the $g(t)\times g(t)$-wall in $G_\alpha[\omega]$. A path $P$ in $G$ which connects two vertices of $W$ will be called a \textit{leap} if it satisfies the following properties:
\begin{itemize}
    \item the only vertices of $P$ that belong to $W$ are its two endpoints;

    \item the graph $P\cup W$ is not planar.
\end{itemize}
We will prove that, in order for $W$ to be contained within a ball of radius at most $r$ in $G_{\alpha}[\omega]$, there must be many vertices of $W$ which are endpoints of at least one leap. Let us begin by considering a large collection $\{W_1,\dots,W_k\}$ of vertex-disjoint subdivisions of the $5h\times 5h$-wall within $W$ which are arranged in a $\sqrt{k}\times \sqrt{k}$-grid-like manner and such that any two of them are separated by at least one row or column of $W$; note that $k$ can be taken to be of order $\Omega((g(t)/h)^2)$. For each $1\leq i\leq k$, let $B_i$ denote the set of vertices on the outer-most layer of $W_i$. For each such $i$, we pick some vertex $u_i$ which belongs to $W_i$ and so that any path connecting $u_i$ to $B_i$ within $W_i$ must contain at least $2h$ vertices that have degree $3$ in $W_i$ (almost equivalently, we can assume that $u_i$ is at distance at least $2h$ from $B_i$ in the copy of $W_{5h\times 5h}$ obtained from $W$ by undoing all edge subdivisions). Suppose that $s\in S$ is such that $u_i\in V_s$, and note that $u_i$ must be at distance at most $h$ from $\Psi_s$ (with respect to $G_\alpha$). Thus, there is a path $P_i$ in $G_\alpha$ which has length at most $h$ and connects $u_i$ to $\Psi_s$. There are two possibilities: either every vertex of $P_i\cap W$ belongs to $W_i$, or there is some section of $P_i$ which constitutes a leap with at least one endpoint inside $W_i$ (see Figure~\ref{fig:leap_to_void}). In the second scenario, we have found a leap with an endpoint in $W_i$, and we can move on to the next index between $1$ and $k$. Assume that there are at least two indices $i$ and $j$ with the property that there exists no leap with an endpoint within $W_i$ or $W_j$. Then, the first scenario must occur for both $u_i$ and $u_j$.
\begin{figure}[ht]
    \centering
    \includegraphics[width=0.93\linewidth]{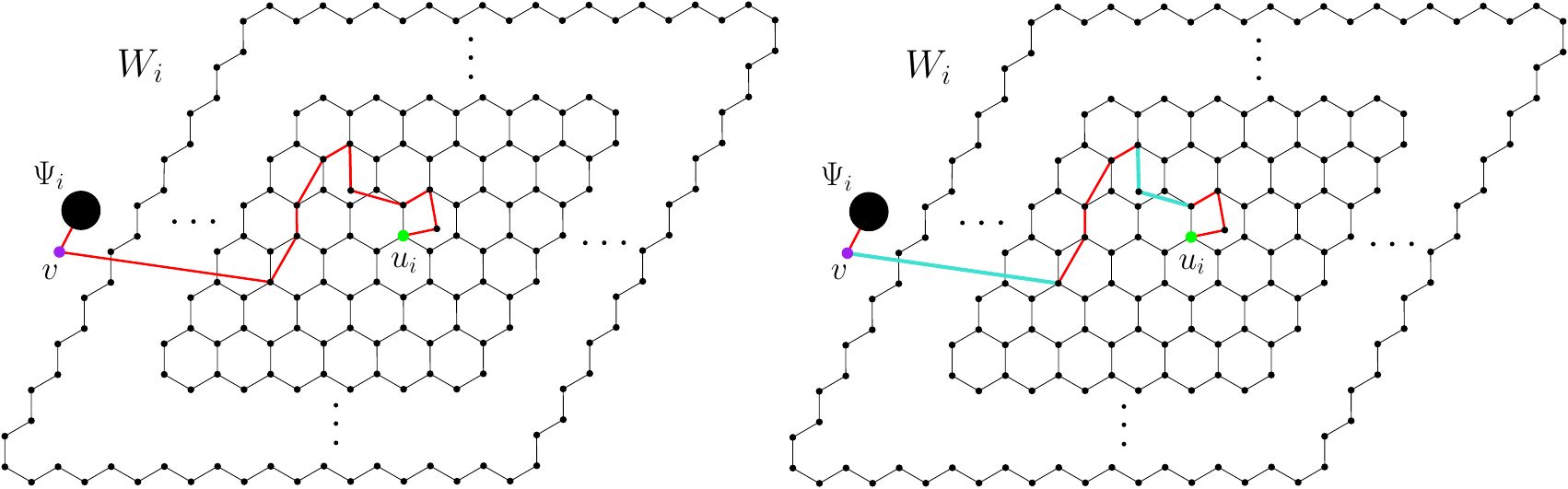}
    \caption{Both images depict $W_i$, which is a subdivision of the $5h\times 5h$-wall (some of its edges may be subdivided, even if it is not shown in the picture). On the left, we see a path of length at most $h$ in $G_{\alpha}$ (red) connecting $u_i$ (green) to the filament $\Psi_i$. This path goes through a vertex $v$ (purple) of $W$ that lies outside $W_i$. On the right, two subsections of this path have been highlighted (turquoise). Both of these sections constitute leaps with an endpoint inside $W_i$. If it were the case that no such section exists, then the path joining $u_i$ to $\Psi_i$ would not be long enough to reach a vertex of $W$ outside $W_i$.}
    \label{fig:leap_to_void}
\end{figure}

Within $G$, the set of vertices of $W_{\backslash i,j}=W\backslash(W_i\cup W_j)$ must separate $u_i$ from $u_j$ (in the sense that $u_i$ and $u_j$ belong to different connected components of $G\backslash W_{\backslash i,j}$, or else there would be a leap with one endpoint in $W_i$ and the other in $W_j$. Since the first of the two scenarios described in the previous paragraph occurs for both indices, we also know that the connected components of $u_i$ and $u_j$ in $G_\alpha\backslash W_{\backslash i,j}$ must contain the filaments $\Psi_i$ and $\Psi_j$, respectively. Thus, property (i) near the end Section~\ref{subsec:filaments} tells us that each of these connected components contain at least $(1+\alpha)^{\lfloor R/2\rfloor}/D^{h+1}$ vertices. Since $G$ is one-ended and $W_{\backslash i,j}$ is finite, one can conclude that at least one of these two connected components must be finite too. The fact that every finite set of vertices within $G_\alpha$ has expansion at least $\alpha$ can be restated as saying that in order to separate some finite set of $m$ vertices from the rest of graph we need to remove at least $\alpha m$ vertices from the graph. This implies that \[|W_{\backslash i,j}|\geq \alpha(1+\alpha)^{\lfloor R/2\rfloor}/D^{h+1}\,.\] As every vertex of $G$ has degree at most $D$, each ball of radius of radius $r$ contains less than $D^{r+1}$ vertices. Hence, if $R$ and $r$ satisfy \[D^{r+1}\leq\alpha(1+\alpha)^{\lfloor R/2\rfloor}/D^{h+1}\,,\] it is impossible for $W_{\backslash i,j}$---and thus also $W$---to lie entirely within a ball of radius $r$ in $G_\alpha[\omega]$, and we are done. 

We have shown that, without loss of generality, it is enough to restrict our attention to the case where, for all indices $1\leq i\leq k-1$, there is a leap $P_{i,\operatorname{leap}}$ with an endpoint in $W_i$. We then move on to show that if $k$ is sufficiently large with respect to the order of $H$, then the graph $W\cup P_{1,\operatorname{leap}}\cup\dots\cup P_{k-1,\operatorname{leap}}$ must contain $H$ as a minor, contradicting the initial assumption that $G$ is $H$-minor-free. We will not go into much detail here regarding this final part of the argument, but we remark that this statement should not be too surprising. Indeed, it has been known since the work of Robertson and Seymour that any planar graph is contained as a minor within any sufficiently large wall. Of course, a wall does not admit any non-planar minors, but the leaps will allow us to bypass the planarity. In essence, the proof proceeds by taking a drawing of $H$ on the plane (which might have some pairs of crossings edges) and attempting to reproduce this drawing within $W$. For any pair of crossing edges in this drawing, we will utilize some leap as a sort of bridge that allows us to draw one of the two edges without intersecting the other one. It might be helpful for the reader to compare this to the well-known fact that any graph $H$ is contained as a minor within any sufficiently large $k\times k$-grid to which we have added some edges connecting every pair of diagonally opposite vertices which lie in the same cell (see Figure~\ref{fig:crossed-grid}); in fact, we will use this result during our proof.

Overall, we have shown that in order for $W$ to be contained within a ball of radius $r$ in $G_\alpha[\omega]$, there must be many leaps whose endpoints are spread out throughout $W$. However, the graph formed by taking the union of $W$ and each of these leaps must have $H$ as a minor. This contradiction concludes the argument.

\subsection{Taking limits} 

Now, we keep $\varepsilon$ and $h$ fixed and let $R,r\rightarrow\infty$. In the limit---and after passing to a subsequence---this produces a set of vertices $\omega_\varepsilon$ with $\mathbb P[\rho\not\in\omega_\varepsilon]\leq 2\varepsilon$, and such that $G_\alpha[\omega_\varepsilon]$ has treewidth at most $t$. As shown in~\cite{jardon2023applications} and~\cite{chen2023tree}, this graph will be sofic. Finally, we can let $\varepsilon\rightarrow 0$, and then $\alpha\rightarrow 0$, in order to obtain a sequence of sofic graphs which converge in the Benjamini-Schramm sense to $(G,\rho)$, as desired.

\subsection*{Organization of the paper} We include the relevant preliminaries in Section~\ref{sec:prelim}. In Section~\ref{sec:expansion}, we discuss how to efficiently cut out non-expanding sets of vertices. In Section~\ref{sec:Voronoi and Starfish}, we describe the construction of the filaments, and then in Section~\ref{sec:local-treewidth} we show that the graph induced by their complement is locally-thin. All these ingredients are then combined in Section~\ref{sec:limits} in order to complete the proof of Theorem~\ref{thm:main}. Lastly, Section~\ref{sec:transitive} contains the proof of Theorem~\ref{thm:main-transitive}.

\section{Preliminaries}\label{sec:prelim}
We use $[k]$ to denote the set $\{1,\dots,k\}$ equipped with the discrete topology. In fact, every finite topological space that we consider will be equipped with the discrete topology. When working with a graph $G$, we use $V$ and $E$ to denote its set of vertices and edges, respectively. For a subset $U\subseteq V$, we let $G[U]$ be the subgraph of $G$ induced by $U$. We will often use $|G|$ to denote $|V|$. Throughout the paper, all distances between vertices of a graph are taken with respect to the shortest-path metric. All the graphs considered in this paper are assumed to be locally finite. Some of the observations we make, both in this and later sections, rely on this assumption, even if it is not explicitly mentioned. We will often omit the term almost surely when it can be inferred. 

An important fraction of our notation and preliminary results is borrowed from~\cite{aldous2007processes} and~\cite{angel2018hyperbolic}.

\subsection{Rooted graphs, local topology, and marks}\label{subsec:rooted-graphs}
By a \textit{rooted graph} we mean a pair $(G,\rho)$ such that $G$ is a locally finite, connected graph, and $\rho$ is a distinguished vertex, which we call the \textit{root}. An \textit{isomorphism} of rooted graphs is simply a graph isomorphism satisfying the additional property that it preserves the root. Let $\mathcal G_*$ denote the space of isomorphism classes of rooted graphs.
For any graph $G$, any vertex $v$ and any $r\in\mathbb N$, let the $r$\textit{-neighborhood} of $v$ in $G$, $B_G(v,r)$, be the rooted subgraph of $G$ which has $v$ as its root and is induced by the set of vertices that are at distance at most $r$ from $v$. The \textit{local topology} is the topology on $\mathcal G_*$ induced by the metric \[d_\text{loc}((G,\rho),(G',\rho')):=2^{-R}\,,\] where \[R:=\sup\{r\in \mathbb N \mid B_G(\rho,r)\text{ and  } B_{G'}(\rho',r) \text{ are isomorphic rooted graphs}\}\,.\] This topology makes $\mathcal G_*$ into a complete and separable space. 

Similarly, by a \textit{doubly rooted graph} we mean a locally finite, connected graph with two distinguished vertices. Isomorphisms of doubly rooted graphs must preserve both roots. We let $\mathcal G_{**}$ be the space of isomorphism classes of doubly rooted graphs, and extend the definition of the local topology to this space in the natural way. We will often write $(G,\rho,o)$ to refer to an element of $\mathcal G_{**}$.

It will be important for us to consider a further generalization of the spaces $\mathcal G_*$ and $\mathcal G_{**}$ where one allows marks (i.e., labels) on the vertices and edges. A \textit{marked graph} consists of a connected, locally finite graph $G=(V,E)$ along with a function $m:V\cup E\rightarrow \mathbb X$ which associates to each vertex and edge of $G$ a \textit{mark} in some polish space $\mathbb X$, whose subjacent metric we denote by $d_{\mathbb X}$. The local topology on the space of marked rooted graphs with marks in $\mathbb X$ is the topology induced by the distance function \[d_{\operatorname{loc}}((G_,\rho,m),(G',\rho',m')):=2^{-R}\,,\] where $R$ now denotes the supremum of all those $r\in\mathbb R$ such that there exists an isomorphism of rooted graphs $\phi$ between $B_{G}(\rho,r)$ and $B_{G'}(\rho',r)$ such that $d_{\mathbb X}(m(v),m(\phi(v)))\leq 1/r$ for all $v\in B_G(\rho,r)$. Once again, this topology induces a complete and separable metric space. The topology on the space of marked doubly rooted graphs is defined analogously. We denote the spaces of marked rooted and doubly rooted graphs with marks in $\mathbb X$ as $\mathcal G_*^{\mathbb X}$ and $\mathcal G_{**}^{\mathbb X}$. 

A \textit{random rooted graph} is a random variable taking values in the space $\mathcal G_*$ equipped with the local topology. Similarly, a \textit{random rooted} $\mathbb X$\textit{-marked graph} is a random variable taking values in $\mathcal G_*^{\mathbb X}$. \textit{Random doubly rooted graphs} are defined analogously.

\subsection{Convergence, soficity, and unimodularity}\label{subsec:unimodular-sofic}
Consider a sequence $\{G_n\}_{n\geq 1}$ of finite, connected (and possibly marked) graphs whose orders tend to infinity. Each $G_i$ gives rise to a random rooted graph $(G_i,\rho_i)$ where the underlying graph is always equal to $G_i$ and $\rho_i$ is chosen uniformly at random from the set of vertices. We say that the sequence $\{G_n\}_{n\geq 1}$ is \textit{Benjamini-Schramm convergent} if the sequence $\{(G_i,\rho_i)\}_{n\geq 1}$ converges in the weak sense (with respect to the local topology on $\mathcal G_*$). The statement that the sequence converges can be restated as saying that, for every positive integer $r$, if we sample a ball of radius $r$ from $G_i$ by choosing the root uniformly at random, then the probability distribution that arises on balls of radius $r$ converges weakly as $n\rightarrow \infty$. This notion of convergence was first introduced and studied by Benjamini and Schramm~\cite{benjamini2011recurrence} and Aldous and Steele~\cite{aldous2004objective}. The fact that $\mathcal G_*$ is complete can be seen to imply that any convergent sequence of finite graphs has as its limit a random rooted graph $(G,\rho)$, which we call the \textit{Benjamini-Schramm limit} of the sequence. A random rooted graph will be called \textit{sofic} if it can be obtained as the Benjamini-Schramm limit of a sequence of finite graphs. Benjamini-Schramm convergence is often called \textit{local convergence} or simply \textit{weak convergence}.

A natural question now is whether every random rooted graph is sofic. Not surprisingly, the answer is no. Indeed, random rooted graphs that arise as limits of finite graphs can easily be seen to enjoy an additional property, which goes by the name of unimodularity. There are multiple equivalent definitions of unimodular random rooted graphs; the one we give below is perhaps the most useful one, even if not the most intuitive.\footnote{Another such definition, which might be easier to parse, is as follows: the random rooted graph $(G,\rho)$ is \textit{unimodular} if the random rooted graph obtained by biasing $(G,\rho)$ according the degree of $\rho$ is precisely the stationary measure of the Markov chain that we get from letting the root of $(G,\rho)$ perform a simple random walk on $G$. One can compare this to the analogous well-known statement for finite graphs, which tells us that the stationary distribution of a simple random walk on a finite graph is proportional to the degree distribution. See~\cite{aldous2007processes} for details.} We refer the reader to~\cite{aldous2007processes} for a much more in-depth treatment of unimodularity.

A random rooted graph $(G,\rho)$ is said to satisfy the \textit{mass transport principle} if, for every Borel function $f:\mathcal G_{**}\rightarrow \mathbb R_+$, we have that 
\begin{equation}\label{eq:MTP}
    \mathbb E_{(G,\rho)}\left[\sum_{o\in V(G)}f(G,\rho,o)\right]=\mathbb E_{(G,\rho)}\left[\sum_{o\in V(G)}f(G,o,\rho)\right]\,.
\end{equation}
    The name of this property is inspired by its physical interpretation: Assume we have some amount of mass distributed among the vertices of $G$, and that for each ordered pair of vertices $(\rho,o)$, $f(G,\rho,o)$ units of mass are sent from $\rho$ to $o$. The mass transport principle then tells us that the expected amount of mass going out of the root is the same as the expected amount of mass coming in. The mass transport principle was first introduced by Häggström~\cite{haggstrom1997infinite} as a tool to study percolation on Cayley graphs, which are known to always satisfy the said principle. A Borel function from $\mathcal G_{**}$ to $\mathbb R_+$ is usually referred to as a \textit{mass transport}. A random rooted graph $(G,\rho)$ is \textit{unimodular} if it satisfies the \textit{mass transport principle}. Soficity and unimodularity are defined analogously for marked graphs. We will often say that the law of some random rooted graph is sofic (resp. unimodular); to us, this means exactly the same as saying that the random rooted graph is sofic (resp. unimodular). 
    
    Unimodular graphs satisfy the following very useful property.
\begin{lemma}[Everything shows at the root, Lemma 2.3,~\cite{aldous2007processes}]\label{lem:everything_root}
    Let $(G,\rho,m)$ be a unimodular random rooted $\mathbb X$-marked graph, and fix some $\gamma\in \mathbb X$. If the mark at the root is a.s.\ $\gamma$, then a.s.\ the mark at every vertex of $G$ is also $\gamma$.
\end{lemma}
In some sense, unimodularity can be interpreted as saying that each vertex of the graph has the same probability of appearing as the root, and the above lemma seems to support this intuition. Still, this interpretation should not be pushed too far.

Clearly, finite graphs (with a uniformly chosen root) are unimodular. Furthermore, unimodularity is preserved under taking Benjamini-Schramm limits, so any random rooted graph $(G,\rho)$ which arises as a weak limit of finite graphs must be unimodular too. Aldous and Lyons asked whether the converse is true. As mentioned in the introduction, it has recently been shown that the answer to this question is negative~\cite{bowen2024aldous,bowen2024aldous2}.

The next elementary result tells us that unimodularity is preserved under a wide range of measurable transformations that can be applied to the graph.

\begin{lemma}[Lemma 2.2,~\cite{angel2018hyperbolic}]\label{lem:local-changes}
    Let $\mathbb X_1,\mathbb X_2$ be polish spaces. Let $(G,\rho,m)$ be a random rooted $\mathbb X_1$-marked graph and $(G',\rho,m')$ be a random rooted $\mathbb X_2$-marked graph such that $G$ and $G'$ have the same vertex set. Furthermore, assume that for any two vertices $u,v$ of $G$, the conditional distribution of $(G',u,v,m')$ given $(G,\rho,m)$ is a.s.\ given by some measurable function of the (doubly rooted) isomorphism class of $(G,u,v,m)$. Then, $(G',\rho,m')$ is also unimodular.
\end{lemma}

We say that a probability measure $\mathbb P$ on $\mathcal G_*$ is \textit{ergodic} if, for every measurable event $\mathcal A\subseteq \mathcal G_*$ that is invariant under changing the root of the graph, we have that $\mathbb P[\mathcal A]\in\{0,1\}$. A random rooted graph is ergodic if the corresponding probability measure is ergodic. As observed in~\cite{aldous2007processes}, every unimodular random rooted graph is a mixture of ergodic unimodular random rooted graphs, in the sense that it can be sampled from by first selecting an ergodic unimodular measure on rooted graphs according to some distribution, and then sampling a rooted graph from this random measure. It will be helpful to note that soficity is preserved under taking mixtures.
 
\subsection{Markings, colorings, and percolation}\label{subsec:markings}

Given a unimodular random rooted graph $(G,\rho)$ and a polish space $\mathbb X$, an $\mathbb X$\textit{-marking} of $(G,\rho)$ consists of a random function $m:E\cup V\rightarrow \mathbb X$ (possibly defined on a larger probability space) assigning a mark to each vertex and edge of $G$ in such a way that $(G,\rho,m)$ is unimodular. One can also define $\mathbb X$-markings of marked graphs in the same way.

Suppose that we are given a random rooted graph $(G,\rho)$ which we wish to augment with an $\mathbb X$-marking $m$. One way in which we can make sure that $m$ will be a marking, as defined above, is to select the mark of each vertex and edge by means of a \textit{local algorithm}: for each element $u$ of $V\cup E$, the mark assigned to this element is decided by an algorithm which takes as input only the ball of radius $r$ around $u$ plus (possibly) some additional randomness. Here, both $r$ and the local algorithm are fixed (i.e., independent of $u$), and the algorithm must not take the root into account.
There are two particular kinds of markings that will play an important role in our proofs: proper colorings and percolations. We proceed to define them and discuss some of their basic properties.

A marking $m:V\rightarrow [k]$ of a unimodular random rooted graph $(G,\rho)$ is said to be a \textit{proper} $k$\textit{-coloring} if any two adjacent vertices of $G$ receive different colors.
A landmark result from descriptive combinatorics\cite{kechris1999borel}---which is usually stated in terms of Borel measurable graphs, as defined in Section~\ref{sec:transitive}---tells us that every random rooted graph with degrees bounded from above by $D$ admits a proper $(D+1)$-coloring (note that the analogous statement for finite graphs is a simple exercise).

Given a unimodular random rooted graph $(G,\rho)$, a \textit{percolation} of $(G,\rho)$ is simply a $\{0,1\}$-marking $\omega$. We further assume that if $e=(u,v)\in E$ satisfies $\omega(e)=1$, then $\omega(u)=\omega(v)=1$. An element $x\in V\cup E$ is said to be \textit{open} (with respect to $\omega$) if $\omega(x)=1$, and it is said to be \textit{closed} if $\omega(x)=0$. We can think of a percolation as an encoding of a random subgraph of $G$ (which consists of all those edges and vertices which are open), and we sometimes refer to closed vertices and edges simply as having been \textit{deleted} or \textit{removed}. Let $G[\omega]$ denote the subgraph of $G$ consisting of all the open vertices and edges and, for every open vertex $v$, define the \textit{cluster} $K_\omega(v)$ of $v$ to be the connected component of $G[\omega]$ that contains $v$. A percolation $\omega$ is a \textit{site percolation} if $\omega (e)=1$ for every $e\in E$ such that both of its endpoints are open. By the \textit{rate} of a percolation $\omega$, we mean $\mathbb P_{(G,\rho),\omega}[\omega(\rho)=0]$; one should think of this quantity as the fraction of vertices which are being deleted. By a slight abuse of notation, for a site percolation $\omega$ we will often use $\omega$ also to denote the set of open vertices. If $\mu$ is the law of $(G,\rho)$, then we denote by $\mu_\omega$ the conditional law of $(K_\omega(\rho),\rho)$ given that $\omega(\rho)=1$; note that this is well defined as soon as $\mathbb P[\omega(\rho)=1]>0$. All of these definitions extend naturally to marked graphs. Most percolations considered in this text will be site percolations.

The following lemma states that percolations preserve unimodularity.
\begin{lemma}[Lemma 3.1,~\cite{angel2018hyperbolic}]
    Let $\mu$ be the law of a unimodular random rooted graph. Then, for any site percolation $\omega$ with positive rate of this random rooted graph, $\mu_\omega$ is also unimodular.
\end{lemma}

We will need one more lemma from~\cite{angel2018hyperbolic}, which allows us to combine multiple markings into a single one. 
\begin{lemma}[Lemma 3.2,~\cite{angel2018hyperbolic}]\label{lem:product_markings}
    Let $(G,\rho)$ be unimodular random rooted graph, and $\{m_i\}_{i\in I}$ be a countable collection of markings of $(G,\rho)$ with marks in the spaces $\{\mathbb X_i\}_{i\in I}$. Then, there exists a $\prod_{i\in I}\mathbb X_i$-marking $m$ of $(G,\rho)$ so that $(G,\rho,\pi\circ m)$ and $(G,\rho,m_i)$ have the same distribution for all $i\in I$, were $\pi_i$ is the projection of $\prod_{i\in I}\mathbb X_i$ onto $\mathbb X_i$.
\end{lemma}

Next, we show that taking a percolation of small rate on a unimodular random rooted graph of bounded degree will have a small effect on the random rooted graph (in the Benjamini-Schramm sense).

\begin{proposition}\label{prop:percolation-stability}
    Let $D$ be a positive integer, and $(G,\rho)$ a unimodular random rooted graph with law $\mu$ such the degree of any vertex in $G$ is at most $D$. Suppose that $\{\omega_n\}_{n\geq 1}$ is a sequence of site percolations of $(G,\rho)$ whose rates are going to $0$ as $n\rightarrow \infty$. Then, the sequence $\{\mu_{\omega_n}\}_{n\geq1}$ converges to $\mu$ in the local topology.

\end{proposition}

\begin{proof}
For starters, we may pass to a subsequence $\{\omega_{i(n)}\}_{n\geq 1}$ such that the rate of $\omega_{i(n)}$ is no more than $2^{-n}$. By Lemma~\ref{lem:product_markings}, there exists a $\prod_{n\geq1}\{0,1\}$-marking $m$ of $(G,\rho)$ which encodes all percolations in this sequence. Since $\sum_{n=1}^{\infty}\mathbb P[w_{i(n)}(\rho)=0]< \infty$, the Borel–Cantelli lemma tells us that a.s.\ $\rho$ is closed with respect to only finitely many percolations of the form $\omega_{i(n)}$. By Lemma~\ref{lem:everything_root}, we get that this must also be true a.s.\ for every vertex of $G$. This allows us to define an $\mathbb N$-marking $m'$ of $(G,\rho,m)$ by letting \[m'(v):=\max\{n\in \mathbb N\mid \omega_{i(n)}(v)=0\}\] for every vertex $v$ of $G$.

Now, for every positive integer $r$, consider the random variable $m'_r$ defined as \[m'_{r}:=\max\{m(v)\mid v\in B_G(\rho,r)\}\,.\] For every $\varepsilon>0$, there exists some positive integer $N(\varepsilon)$ so that $\mathbb P[m'_r\geq N(\varepsilon)]\leq \varepsilon$. In other words, with probability at least $1-\varepsilon$, all the vertices inside $B_G(\rho,r)$ are open with respect to $\omega_{i(n)}$ for all $n\geq N(\varepsilon)$. Letting $\varepsilon\rightarrow 0$, and then $r\rightarrow\infty$, the result follows.
\end{proof}

In particular, Proposition~\ref{prop:percolation-stability} tells us that in order to prove that some unimodular random rooted graph $(G,\rho)$ with law $\mu$ and bounded maximum degree is sofic, it suffices to show that there exists a sequence $\{\omega_n\}_{n\geq 1}$ of percolations of $(G,\rho)$ whose rates are tending to $0$ and such that $\mu_{\omega_n}$ is sofic for each $n$.

A random rooted graph $(G,\rho)$ is said to be \textit{strongly sofic} if for any polish space $\mathbb X$ and any $\mathbb X$-marking $m$ of $(G,\rho)$, $(G,\rho,m)$ is sofic as a marked random rooted graph. Strong soficity clearly implies soficity (although the currently available tools do not allow us to discard the possibility that the two properties are actually equivalent).

\subsection{Fundamentals of structural graph theory}\label{subsec:structural}

Here, we recall the notions of graph minors, graph subdivisions, tree decompositions, and walls.

Given two graphs $G=(V,E)$ and $H=(V_H,E_H)$, we say that say that $H$ is a \textit{minor} of $G$, and write $H\preceq G$, if there exists a family $\{U_v\}_{v\in V_H}$ of pairwise disjoint subsets of $V$ such that $G[U_v]$ is connected for every $v\in V_H$ and, for every edge $(u,v)\in E_H$, there is an edge of $G$ between $U_{u}$ and $U_v$.\footnote{Furthermore, we say that $H$ is an $r$\textit{-shallow minor} of $G$ if the $U_v$'s can be chosen so that $G[U_v]$ has radius at most $r$ for every $v\in V_H$.} In other words, $H$ is a minor of $G$ if it can be obtained from $G$ by a sequence of edge contractions, vertex deletions, and edge deletions. Whenever these properties hold, we will say that the sets $\{U_v\}_{v\in V(H)}$ \textit{act as witnesses} to the fact that $H\preceq G.$ A graph $G$ is said to be $H$\textit{-minor-free} if it does not have $H$ as a minor. A graph is \textit{minor-excluded} if it is $H$-minor-free for some finite graph $H$. Similarly, a class of graphs $\mathcal C$ is \textit{minor-excluded} if there exists a finite graph $H$ such that every graph in $\mathcal C$ is $H$-minor-free. 

The process of \textit{subdividing} an edge of a graph consists of erasing the edge and reconnecting its endpoints by means of a path. By a \textit{subdivision} of a graph $H$, we mean any graph that can be obtained from $H$ by subdividing some of its edges. 

The importance of the study of graph minors and subdivisions is hard to overstate, and these notions have played a central role in many developments, both within graph theory and computer science. Graph minors were first studied due to the role they play in the characterization of planar graphs: Wagner~\cite{wagner1937eigenschaft} showed that a graph is planar if and only if it has neither the complete graph $K_{5}$ nor the complete bipartite graph $K_{3,3}$ as a minor. It was later shown that graphs which can be embedded in any fixed surface are characterized by a finite family of excluded minors~\cite{archdeacon1989kuratowski,robertson1990graph}.

We move on to discuss tree decompositions. A \textit{tree decomposition} of a graph $G=(V,E)$ is a pair $(T,\mathcal V)$, where $T$ is a tree with vertex set $V(T)$ and $\mathcal V$ is a family $\{V_t\}_{t\in V(T)}$ of subsets of $V$, called \textit{bags}, indexed by the set of vertices of $T$, such that the following properties hold:

\begin{enumerate}
    \item $V=\cup_{t\in V(T)}V_t$;

    \item for every edge $e=(u,v)$ of $G$, there exists some $t\in V(T)$ with $u,v\in V_t$;

    \item if $t_1,t_2,t_3\in V(T)$ are such that $t_2$ belongs to the path connecting $t_1$ to $t_3$ in $T$, then $V_{t_2}\supseteq V_{t_1}\cap V_{t_3}.$
\end{enumerate}

The \textit{width} of a tree decomposition $(T,\mathcal V)$ is defined as $\max_{t\in V(T)}|V_t|-1$, and the \textit{treewidth} of $G$ is the least width among all tree decompositions of $G$. Tree decompositions and treewidth play an important role in understanding the structure of minor-excluded graphs (see, for example,~\cite{robertson2003graph}). The \textit{torsos} of a tree decomposition $(T,\mathcal V)$ are the graphs $G_t$ (where $t\in V(T)$) obtained by adding to the induced subgraph $G[V_t]$ an edge between any two vertices $u,v$ such that $u,v\in V_t\cap V_{t'}$ for some neighbor $t'$ of $t$ in $T$. The sets $V_{t,t'}:=V_t\cap V_{t'}$, where $(t,t')$ ranges over all edges of $T$, are called the \textit{adhesion sets} of the tree decomposition. The \textit{adhesion} of $(T,\mathcal V)$ is the supremum of the sizes of the adhesion sets over all edges of $T$. 

For each two positive integers $k$ and $\ell$, the $k\times \ell$\textit{-grid} $G_{k\times \ell}$ is the graph that arises as the Cartesian product of a path with $k$ vertices and a path with $\ell$ vertices. The $k\times \ell$\textit{-wall} $W_{k\times \ell}$ is the skeleton of a $k\times \ell$ hexagonal grid on the plane. The $7\times 7$-grid and the $7\times 7$-wall are depicted in Figure~\ref{fig:grid-wall}.

\begin{figure}[ht]
    \centering
    \includegraphics[width=0.6\linewidth]{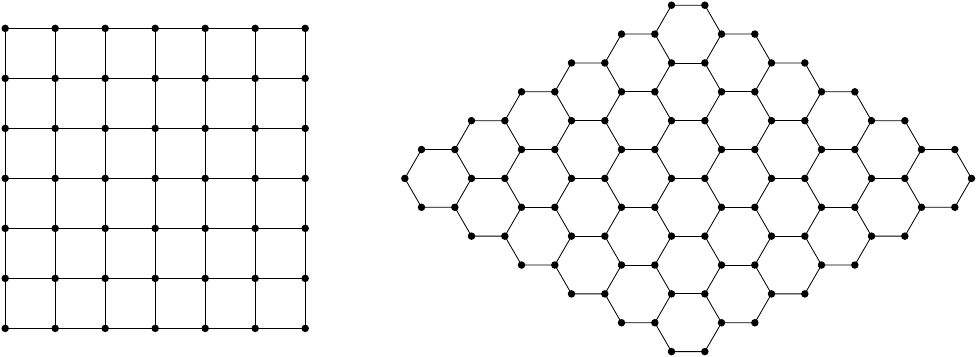}
    \caption{On the left, the grid $G_{7\times 7}$. On the right, the wall $W_{7\times 7}$.}
    \label{fig:grid-wall}
\end{figure}

The following result of Robertson and Seymour is one of the cornerstones of structural graph theory, and will also play a crucial role in our proofs.

\begin{theorem}[Grid-Minor Theorem~\cite{robertson1986graph}]\label{thm:treewidth-walls}
    For each positive integer $t$, there exists an integer $g(t)$ such that every graph of treewidth at least $t$ contains $G_{g(t)\times g(t)}$ as a minor, and contains a subgraph isomorphic to some subdivision of $W_{g(t)\times g(t)}$.
\end{theorem}

Clearly, as $t\rightarrow \infty$, we also have that $g\rightarrow \infty$. In fact, in the original proof of this theorem, the function $g$ that was used grew extremely quickly. Several years after the discovery of this result, it was shown in~\cite{chekuri2016polynomial} that the function $g(t)$ in the above theorem can in fact be taken to be a polynomial.

\subsection{Expansion, amenability, and hyperfiniteness}\label{subsec:expansion-amenable}

Given a simple graph $G=(V,E)$ and a set of vertices $K\subset V$, the \textit{inner boundary} of $K$ in $G$ is defined as \[\partial_{V,\operatorname{in}}^GK:=\{u\in K \mid u\ \text{is adjacent to some vertex not in}\ K \}\,,\] and the \textit{outer boundary} of $K$ is  \[\partial_{V,\operatorname{out}}^GK:=\{u\not\in K \mid u\ \text{is adjacent to some vertex in}\ K \}\,.\] Similarly, the \textit{edge boundary} of $K$ in $G$ is \[\partial_{E}^GK:=\{e\in E\mid \text{exactly one vertex of }e\text{ belongs to $K$}\}\,.\] We will omit the superscript $G$ whenever the graph is clear. Whenever $K$ is finite, its \textit{(inner) vertex-expansion} and \textit{edge-expansion} are defined as \[\Phi_V^G(K):=\frac{|\partial_{V,\operatorname{in}}K|}{|K|}\,\quad\text{and}\quad \Phi_E^G(K):=\frac{|\partial_EK|}{|K|}\,,\] respectively. Now, the \textit{(inner) vertex-expansion constant} and \textit{edge-expansion constant} of $G$ are respectively given by \[\Phi_V(G):=\inf\left\{\frac{|\partial_{V,\operatorname{in}}K|}{|K|}\mid \emptyset\neq K\subset V\text{ is finite}\right\}\,,\quad \Phi_E(G):=\inf\left\{\frac{|\partial_EK|}{|K|}\mid \emptyset\neq K\subset V\text{ is finite}\right\}\,.\] A graph is said to be \textit{vertex}(resp. \textit{edge)-amenable} if $\Phi_V(G)=0$ (resp. $\Phi_E(G)=0)$. For graphs of bounded degree, $G$ being vertex-amenable is equivalent to $G$ being edge-amenable.

If $(G,\rho)$ is a unimodular random rooted graph, we say that a percolation $\omega$ of $(G,\rho)$ is \textit{finitary} if every connected component of $G[\omega]$ is finite. A unimodular random rooted graph $(G,\rho)$ is said to be \textit{hyperfinite} if there exists a sequence $\{\omega_n\}_{n\geq 1}$ of finitary site percolations of $(G,\rho)$ such that $\bigcup_{n=1}^\infty\omega_n=G$ and $\omega_n\subseteq \omega_{n+1}$ for all $n$. Every hyperfinite, unimodular random rooted graph is sofic.

\subsection{Ends and accessibility}\label{subsec:ends}

Let $G$ be an infinite connected graph. The \textit{number of ends} of $G$ is the supremum of the number of connected components of the graph $G\backslash X$, as $X$ ranges over all finite sets of vertices. A \textit{ray} of $G$ is a simple path starting at some vertex of $G$ and extending infinitely in one direction. Two rays $\xi_1$ and $\xi_2$ are said to be equivalent if there exists a third ray which contains infinitely many vertices from both $\xi_1$ and $\xi_2$. This defines an equivalence relation on the set of rays, and its equivalence classes are known as \textit{ends}. This is consistent with the definition of the number of ends given at the beginning of the paragraph. The number of ends of infinite graphs is a widely studied parameter (see, for example,~\cite{diestel2003graph} and the references therein). In the unimodular setting, we have the following powerful result, essentially due to Aldous and Lyons~\cite{aldous2007processes} (see also~\cite{angel2016unimodular,angel2018hyperbolic}). 

\begin{theorem}[Theorem 3.7~\cite{angel2018hyperbolic}]\label{thm:ends}
    Let $(G,\rho)$ be a unimodular random rooted graph. Then, the number of ends of $G$ a.s.\ belongs to the set $\{0,1,2,\infty\}$. Moreover, if the number of ends of $G$ is a.s.\ $2$ and $\mathbb E[\operatorname{deg}\rho]\leq \infty$, then $(G,\rho)$ is hyperfinite.
\end{theorem}

A finite set of vertices $S$ is said to \textit{separate} two ends $\xi_1$ and $\xi_2$ of an infinite graph $G$ if the removal of $X$ results in at least two connected components, and every time we select a ray $r_1$ in $\xi_1$ and a ray $r_2$ in $\xi_2$, we can pick (infinite) subrays $r_1'$ and $r_2'$ which are completely contained within different such connected components. A graph is said to be \textit{accessible} if there exists some positive integer $k$ such that any two ends of $G$ can be separated by a set of at most $k$ vertices. 

\subsection{Treeability}\label{subsec:treeable}

A unimodular random rooted graph $(G,\rho)$ is said to be \textit{treeable} if there exists a random graph $T$ on the vertex set of $G$ such that $(T,\rho)$ is unimodular and $T$ is a.s.\ a spanning tree of $G$ ($T$ is allowed to be random even after conditioning on a realization of $(G,\rho)$). As we mentioned during the introduction, it was shown by Elek and Lipner~\cite{elek2010sofic} that treeable unimodular random rooted graphs are sofic. 

Proving that a certain random rooted graph is treeable can often be very challenging. Recently, the following result has been obtained independently by Jard{\'o}n-S{\'a}nchez~\cite{jardon2023applications} and Chen, Poulin, Tao, Tserunyan~\cite{chen2023tree}.

\begin{theorem}[Bounded treewidth implies treeability, Theorem 1.4,~\cite{chen2023tree}]\label{thm:treewidth-treeable}
Let $(G,\rho)$ be a unimodular random rooted graph such that $G$ a.s.\ has boumded treewidth. Then, $(G,\rho)$ is treeable.
\end{theorem}

This theorem will be another one of the cornerstones of the proof of Theorem~\ref{thm:main}.

\medskip\noindent\textit{Remark.} Both in~\cite{jardon2023applications} and~\cite{chen2023tree}, this result is stated in the language of Borel measurable graphs and graphings. We refer the reader to either Section~\ref{sec:transitive} or~\cite[Chapter 18]{lovaszlimits} for a discussion on the connection between graphings and unimodular random rooted graphs.

\section{Removing sets with poor expansion}\label{sec:expansion}

Fix some finite graph $H$. Let $(G,\rho)$ be a unimodular random rooted graph such that $G$ is one-ended and $H$-minor-free, and let $\mu$ be its law. For every poitive integer $D$, let $\xi_D$ denote the site percolation of $(G,\rho)$ such that, for every vertex $v$, $\xi_D(v)=1$ if and only if $v$ has degree at most $D$ in $G$. We denote $G[\xi_D]$---the subgraph induced by all the open vertices---simply as $G_D$. Clearly, every vertex in $G_D$ has degree at most $D$ and, as $D\rightarrow \infty$, $\mu_{\xi_{D}}$ converges to $\mu$ in the Benjamini-Schramm sense. We remark that the connected components of the graph $G_D$ are not necessarily one-ended. 

The goal of this section is to show that there is exists a percolation of small rate which splits the graph $G_D$ into connected components which are either non-amenable or finite. 

\begin{theorem}\label{thm:non-amenable}
    Let $(G,\rho)$ be a unimodular random rooted graph and $\xi$ be a percolation such that every vertex of the induced subgraph $G_D:=G[\xi]$ has degree at most $D$, where $D$ is some positive integer. Then, for every $\alpha>0$, $(G,\rho)$ admits a site percolation $\omega_\alpha$ of rate at most $\alpha$ such that a.s.\ every infinite connected component of $G_D[\omega_\alpha]$ 
    has inner vertex-expansion constant at least $\alpha$. 
\end{theorem}

\begin{proof}
Fix some $\alpha>0$ and consider a positive integer $M$. We will repeatedly get rid of portions of the graph $G_D$ which are connected, contain at most $M$ vertices, and have vertex-expansion no more than $\alpha$. We then let $M$ go to infinity and take the limit of these constructions. Let us say that a set of vertices of a graph is an $(M,\alpha)$-\textit{island} if it is minimal---with respect to the containment relationship---with these two properties (i.e., it has at most $M$ elements and its vertex-expansion is at most $\alpha$). At a high level, our strategy consists of cutting out $(M,\alpha)$-islands by deleting the vertices on their inner boundaries. This requires some care, since we wish to remove disjoint islands at every step, and everything must be done in a measurable way.

We begin by describing a site percolation $\omega_{M,1}$ of $(G,\rho)$. Start by assigning uniformly and independently a random weight from $[0,1]$ to each vertex of $G_D$. For each vertex $v$ of $G_D$, we make use of the following local algorithm that decides whether the vertex should be deleted or not: If $\xi(v)=0$, then $\omega_{M,1}(v)=1$. Otherwise, we look at the ball of radius $3M$ around $v$ in $G_D$, $B_{G_D}(v,3M)$. Create an auxiliary graph $G_{v,\operatorname{isl}}$ whose vertices are the $(M,\alpha)$-island that are completely contained in this ball (i.e., $G_{v,\operatorname{isl}}$ has one vertex for every such set), and where two vertices are adjacent if and only if the two corresponding islands share at least one element. For each $(M,\alpha)$-island $U$ corresponding to a vertex of $G_{v,\operatorname{isl}}$, compute the average weight of its elements, and denote it by $w(U).$ We say that $U$ is \textit{heavy} if $w(U)>w(W)$ for every $W$ that is adjacent to $U$ in $G_{v,\operatorname{isl}}$. Set $\omega_{M,1}(v)=0$ (i.e., delete $v$) if and only if it belongs to $\partial_{V,\operatorname{in}}^{G_D}U$ for some $(M,\varepsilon)$-island $U$ that is heavy.

The point of this construction is that if some $(M,\alpha)$-island $U$ contains $v$ and is heavy when looking at $B_{G_D}(v,3M)$, then it will be heavy when looking at any other ball of radius $3M$ that contains it. This allows us to talk about the family of heavy $(M,\alpha)$-islands without needing to specify the vertex $v$, and it implies that all elements in $\partial_{V,\operatorname{in}}^{G_D}U$ will be deleted, thus cutting off $U$ from the rest of the graph. Furthermore, no other elements in $U$ will be deleted, since that would imply that there is some $(M,\alpha)$-island which is adjacent to $U$ in $G_{v,\operatorname{isl}}$ and has larger labels on average. Thus, after the deletions, we are left with multiple finite residual connected components of order less than $M$, and some (possibly $0$) infinite components. We use $\mathcal I_{M,\alpha}$ to denote the set of heavy $(M,\alpha)$-islands in $G_D$, which is well defined by our discussion above.

Let $\theta:=\mathbb P_{(G,\rho,\xi),\omega_{M,1}}[\omega_{M,1}(\rho)=0]$ denote the rate of the percolation $\omega_{M,1}$. Also, define \[\tau:=\mathbb E_{(G,\rho,\xi),\omega_{M,1}}[\rho\in U\text{ for some }U\in \mathcal I_{M,\alpha}]\,.\] (Note that we are abusing the notation slightly by using the subscript $\omega_{M,1}$ to denote the randomness coming from the weights that were assigned to the vertices.) We claim that $\theta\leq\alpha\tau$. Let us describe a Borel-measurable function $f$ to which we shall apply the mass transport principle~\eqref{eq:MTP}. First, if $\omega_{M,1}(\rho)=1$ then we set $f(G,\rho, o)=0$.\footnote{To be entirely precise, $f$ is also a function of $\xi$ and the $[0,1]$-marking which corresponds to the vertex weights. We ignore this in our notation in order to avoid excessive clutter. Some other imprecisions of this nature will come up during the rest of the paper.} Instead, if $\omega_{M,1}(\rho)=0$, we let $U_\rho$ be the unique element of $\mathcal I_{M,\alpha}$ such that $\rho\in\partial_{V,\operatorname{in}}^{G_D}U$, and then set $f(G,\rho,o)=0$ if $o\not\in U_\rho$, and $f(G,\rho,o)=1/|U_\rho|$ otherwise. On one hand, \[\mathbb E_{(G,\rho,\xi),\omega_{M,1}}\left[\sum_{o\in V(G)}f(G,\rho,o)\right]=\mathbb E_{(G,\rho,\xi),\omega_{M,1}}\left[\mathbf 1[\omega_{M,1}(\rho)=0]\cdot\left(|U_\rho|\cdot\frac{1}{|U_\rho|}\right)\right]=\mathbb P[\omega_{M,1}(\rho)=0]=\theta\,,\] where $\mathbf 1[\omega_{M,1}(\rho)=0]$ denotes the indicator function of the event that $\omega_{M,1}(\rho)=0$. Note that $U_\rho$ is only well defined whenever $\omega_{M,1}(\rho)=0$, but the expression written above makes sense nonetheless, as it takes the value $0$ when $\omega_{M,1}(\rho)=1$. In the other direction, we have \[\mathbb E_{(G,\rho,\xi),\omega_{M,1}}\left[\sum_{o\in V(G)}f(G,o,\rho)\right]\leq\mathbb P_{(G,\rho,\xi),\omega_{M,1}}[\rho\in U\text{ for some }U\in \mathcal I_{M,\alpha}]\cdot\max_{K\in \mathcal I_{M,\alpha}}\frac{|\partial_VK|}{|K|}\leq\alpha\tau\,.\] Hence, $\theta\leq\alpha\tau$, as desired.

The above process is now repeated within the subgraph of $G$ induced by all those vertices which are open with respect to both $\xi$ and $\omega_{M,1}$ and which do not belong to any of the heavy $(M,\alpha)$-islands in $\mathcal I_{M,\alpha}$ (i.e., those $(M,\alpha)$-islands that have already been separated from the rest of the graph), thus obtaining a new percolation $\omega_{M,2}$. We keep going in this way, obtaining a sequence of $\{\omega_{M,n}\}_{n\geq 1}$ of site percolations of $(G,\rho)$. More precisely, when constructing $\omega_{M,n}$ for $n>1$, we consider the $(M,\alpha)$-islands within $G_D[\omega_{M,n-1}]$---only those which are not a subset of one of the heavy $(M,\alpha)$-islands that were separated during previous steps---and draw new i.i.d. $[0,1]$-marks encoding the new weights of the vertices of $G_D[\omega_{M,n-1}]$, which are then used to define the new heavy $(M,\alpha)$-ilsands (all previously assigned weights will be ignored from now on). A vertex $v$ will be closed with respect to $\omega_{M,n}$ if either $\omega_{M,n-1}(v)=0$ or if $v$ belongs to the inner vertex boundary of a heavy $(M,\alpha)$-island considered in this last step of the process. In particular, note that the set of vertices which are open with respect to $\omega_{M,n-1}$ is a subset of the set of vertices which are open with respect to $\omega_{M,n}$ (in other words, $\omega_{M,n-1}\supseteq\omega_{M,n}$). By iterating the argument from the previous paragraph, we see that the rate of each $\omega_{M,n}$ is upper bounded by $\alpha$ times the probability that the root belongs to a heavy $(M,\alpha)$-island from one of the first $n$ iterations of the process. 

Recall that $\omega_{M,n-1}\supseteq\omega_{M,n}$ for all $n\geq 2$ and consider the percolation $\omega_M:=\bigcap_{n=1}^{\infty}{\omega_{M,n}}$ of $(G,\rho)$. That is, a vertex will be closed with respect to $\omega_m$ if and only if it is closed with respect to least one of the percolations $\omega_{M,n}$ with $n\in \mathbb N$. 
The rate of $\omega_M$ is no more than $\alpha$ times the probability that the root belongs to one of the heavy $(M,\alpha)$-islands that were separated from the rest of the graph during the construction of one of the percolations in the sequence $\{\omega_{M,n}\}_{n\geq 1}$. 

We will argue that, a.s., there are no $(M,\alpha)$-islands in any of the infinite connected components of $G_D[\omega_M]$. By Lemma~\ref{lem:everything_root}, it suffices to show that the probability that $\rho$ simultaneously belongs to an $(M,\alpha)$-island and an infinite connected component of $G_D[\omega_M]$ is $0$. Towards this goal, we fix some $G_D$ and $\rho$ and consider a set $U$ of at most $M$ vertices such that $\rho\in U$ and $G_D[U]$ is connected. In order for $U$ to be an $(M,\alpha)$-island in an infinite connected component of $G_D[\omega_M]$, it must be an $(M,\alpha)$-island in an infinite connected component of $G_D[\omega_{M,n}]$ for all sufficiently large values of $n$, and thus it must have been considered during each of these steps. However, for each such island in $G_D[\omega_{M,n}]$, the probability that this island will turn out to be heavy during the construction of $\omega_{M,n+1}$ is bounded from below by some $p=p(D,M,\alpha)>0$ that depends only on $D$, $M$, and $\alpha$. If $M$ is a heavy $(M,\alpha)$-island at this step, then $\rho$ will not belong to an infinite connected component of $G_D[\omega_{M,n+1}]$. Hence, the probability that $U$ is an $(M,\alpha)$-island in an infinite component of $G_D[\omega_{M,n}]$ for at least $k$ distinct values of $n$ tends to $0$ as $k\rightarrow \infty$. The set of candidates for $U$ is finite, as they correspond to induced connected subgraphs of $G_D$ with order at most $M$ which contain $\rho$. Whence, by the union bound, the probability---with respect to the randomness in the construction of $\omega_M$---that $\rho$ belongs to an $(M,\alpha)$-island within an infinite connected component of $G_{D}[\omega_M]$ is $0$. As a consequence, this must also hold over the randomness in the choices of $G$, $\xi$, $\rho$, and $\omega_M$, as desired.

Our next step will be to increase the value of $M$ and repeat the above construction, but this time on the graph $G_{D}[\omega_M]$ instead of $G_D$. This will work in essentially the same way as the construction of the sequence $\{\omega_{M,n}\}_{n\geq 1}$. Suppose we have already constructed the percolation $\omega_{M'-1}$ for some positive integer $M'> M$. Then, we look at the graph $G_D[\omega_{M'-1}]$ and consider all $(M',\alpha)$-islands within it. Again, assign random $[0,1]$ labels to every vertex of this graph and define the heavy islands by looking at balls of radius $3M'$ within $G_D[\omega_{M'-1}]$. A vertex will be closed with respect to $\omega_{M',1}$ if either it is closed with respect to $\omega_{M'-1}$ or it belongs to the boundary of one of these heavy $(M',\alpha)$-islands. We keep going in this way, producing a sequence $\{\omega_{M',n}\}_{n\geq 1}$ of percolations of $(G,\rho)$ with $\omega_{M',n-1}\supseteq\omega_{M',n}$ for all $n$, and then we define $\omega_{M'}:=\bigcap_{n=1}^{\infty}\omega_{M',n}$. A simple inductive argument reveals that the rate of $\omega_{M'}$ is at most $\alpha$ times the probability that $\rho$ belongs to one of the heavy islands that have been separated from $G_D$ thus far. Note that $\omega_{M'-1}\supseteq\omega_{M'}$ for all $M'>M$. 

Finally, we set $\omega_{\alpha}:=\bigcap_{M'=M}^\infty\omega_M'$ and prove that it satisfies the desired properties. For starters, an inductive argument once again shows that the rate of $\omega_\alpha$ is upper bounded by $\alpha$ times the probability that the root belongs to one of the heavy islands that were separated from the graph at any step of the construction. This, in particular, implies that $\mathbb P[\omega_{\alpha}(\rho)=0]\leq \alpha$. Next up, we show that a.s.\ each infinite component of $G_D[\omega_\alpha]$ has vertex-expansion constant at least $\alpha$. If this were not the case, then there would be an infinite connected component of $G_{D}[\omega_{\alpha}]$ which contains a $(k,\alpha)$-island for some positive integer $k$. By our choice of $\omega_\alpha$, any such island must also be a $(k,\alpha)$-island within some infinite connected component of $G_D[\omega_{M'}]$ for all sufficiently large $M'$. Yet, for any $M'\geq k$, we know that a.s\ no infinite connected component of $G_{D}[\omega_\alpha]$ contains a $(k,\alpha)$-island. 
This observation concludes the proof.\end{proof}

Given $(G,\rho)$ and $\xi_D$ as in the beginning of this section and some percolation $\omega_\alpha$ as in the statement of Theorem~\ref{thm:non-amenable}, we write $G_{D,\alpha}:=G_D[\omega_\alpha]$.

\section{Voronoi cells and Filaments}\label{sec:Voronoi and Starfish}

The goal of this section is to make precise the construction of the filaments described in Section~\ref{subsec:filaments}. Our starting point will be the percolated graph $G_{D,\alpha}$. Namely, we prove the following.

\begin{theorem}\label{thm:filaments}
    Consider some $\alpha>0$ and some integer $D\geq 2$. Let $(G,\rho)$ be a unimodular random rooted graph and $\xi$ be a percolation such that the degree of every vertex in $G_{D,\alpha}:=G[\xi]$ is at most $D$, and so that each infinite connected component of $G_{D,\alpha}$ has inner vertex-expansion constant at least $\alpha$. Then, for every $\varepsilon\in(0,1/2)$ and every positive integer $N$, there exist a positive integer $h=h(\varepsilon)$ (which does not depend on $N$) and a site percolation $\omega$ of $(G,\rho)$ of rate at most $\varepsilon$ with (a.s.) the two following properties:
    \begin{enumerate}[label=(\Roman*)]
        \item each vertex of $G_{D,\alpha}[\omega]$ which belongs to an infinite connected component of $G_{D,\alpha}$ is at distance at most $h$ in $G_{D,\alpha}$ from some vertex (also in $G_{D,\alpha}$) which is closed with respect to $\omega$;
        
        \item for every vertex $v$ in $G_{D,\alpha}$ with $\omega(v)=0$, there is some set $\Psi$ of at least $N$ vertices from $G_{D,\alpha}$ which are closed with respect to $\omega$, and such that $G_{G,\alpha}[\Psi]$ is connected.
    \end{enumerate}
\end{theorem}

The site percolation $\omega$ produced by the above theorem corresponds to the \textit{filaments} described in sections~\ref{subsec:about-proofs} and~\ref{subsec:filaments}. One should think of the set $\Psi$ in property (II) above as being the filament containing the vertex $v$.

Before proving Theorem~\ref{thm:filaments}, we pave the way by setting up a couple of Voronoi partitions of $G_{D,\alpha}$ at different scales, which will be encoded using some markings. The following standard lemma shall allow us to choose appropriate sets of vertices to act as the centers of the regions in these partitions. We will mark with the symbol $\star$ every vertex which is to be a center, and the mark $\bullet $ will be assigned to all other vertices.

\begin{lemma}\label{lem:maximal_k-independent_set}
    Let $(G,\rho)$ be a unimodular random rooted graph admitting a percolation $\xi$ such that the degree of every vertex in $G[\xi]$ is uniformly bounded from above by a constant. Then, for every positive integer $k$, there exists a $\{\bullet, \star\}$-marking $m:V\rightarrow \{\bullet, \star\}$ of $(G,\rho)$ such that, a.s., any two vertices marked with $\star$ are at distance greater than $k$, and any vertex of $G[\xi]$ marked with $\bullet$ is at distance at most $k$ in $G[\xi]$ from some vertex marked with $\star$ (which also belongs to $G[\xi]$).
\end{lemma}

\begin{proof}
Consider $(G,\rho)$ and $\xi$ as in the statement of the theorem and suppose that every vertex in $G[\xi]$ has degree at most $D$. Let $G^k$ be the graph obtained from $G$ by adding an edge between any two vertices of $G[\xi]$ which lie at distance at most $k$ within $G[\xi]$ and are not already adjacent. By Proposition~\ref{lem:local-changes}, $(G^k,\rho)$ is again unimodular. Moreover, the degree of every vertex in the induced subgraph $G^k[\xi]$ is bounded from above by \[1+D+D(D-1)+\dots+D(D-1)^{R-1}=:D'\,.\] The next step is to show that there exists a marking encoding a maximal independent set of vertices in $G^k[\xi]$.  

By the results in~\cite{kechris1999borel} (see also Section~\ref{subsec:markings} and~\cite[Theorem 18.3]{lovaszlimits}), there exists a $[D'+1]$-marking $c$ of $(G^k,\rho)$ which behaves as a proper $(D'+1)$-coloring when restricted to $G^k[\xi]$. We modify this coloring $c$ in a series of steps. Starting with $i=2$ and then going through $i=3,4,\dots,D'+1$, we do the following: simultaneously for every vertex $v$ of $G[\xi]$ with $c(v)=i$, if no neighbor of $v$ in $G^k[\xi]$ has color $1$, then we change the color of $v$ to be $1$ (that is, we set $c(v)=1$); otherwise, we leave $c(v)$ unchanged. It is not difficult to see that after each step of this process we will have obtained a new $[D'+1]$-marking which is still a proper $(D'+1)$-coloring when restricted to $G^{k}[\xi]$. Moreover, at the end of the process, every vertex which has not been assigned color $1$ must be adjacent (in $G^k[\xi]$) to some vertex of color $1$. In other words, the set of vertices of color $1$  constitutes a maximal independent set in $G^k[\xi]$. Hence, any two vertices of color $1$ in $G[\xi]$ are at distance at least $k$ in $G[\xi]$, and any vertex of $G[\xi]$ which is not of color $1$ is at distance at most $k$ in $G[\xi]$ from some vertex of color $1$. Now, we can simply define $m$ by letting
\begin{align*}
    m(v) :=
    \begin{cases*}
      \star & if $c(v)=1$,\\
      \bullet        & otherwise.
    \end{cases*}
  \end{align*}

\end{proof}

Fix some $\varepsilon>0$ and let $h=h(\varepsilon)$ be an even positive integer which is large enough that \[2h(1+\alpha D^{-1})^{-h/2}\leq \varepsilon\,.\] Consider a unimodular graph $(G,\rho)$ and a percolation $\xi$ as in the statement of Theorem~\ref{thm:filaments}. Apply Lemma~\ref{lem:maximal_k-independent_set} with parameter $k=h$ to obtain an $\{\bullet,\star\}$-marking $m_1$, and let $S_1$ denote the set of vertices of $G_{D,\alpha}$ which are marked with $\star$. Let $d_{D,\alpha}$ denote the shortest path metric in $G_{D,\alpha}$. Assign random labels in $[0,1/2]$ uniformly and independently to all elements of $S_1$ and let $\ell_1(s)$ be the label of $s$ for each $s\in S_1$. These labels will serve two purposes: they will act as the identifiers of the elements of $S$, and they will allow us to break ties in an organized manner when constructing Voronoi cells. Now, for every vertex $v$ of $G_{D,\alpha}$, let \[\mathcal V_1(v):=\arg\min_{s\in S_1}\{d_{D,\alpha}(v,s)+\ell_1(s)\}\,,\] which is well defined a.s.. Note that, by the properties of $m_1$, the function $\mathcal V_1(v)$ can be computed by means of a local algorithm. Indeed, $v$ must be at distance at most $h$ from some element in $S_1$, so it suffices to look at $B_{G_{D,\alpha}}(v,h)$ when computing $\mathcal V_1(v)$. For each $s\in S_1$, the set of vertices $v$ of $G_{D,\alpha}$ such that $\mathcal V_1(v)=s$ is the \textit{Voronoi cell} of $s$, and will be denoted by $V_{1,s}$. We consider a marking $m_{\ell_1}$ which assigns to each vertex $v$ of $G_{D,\alpha}$ the mark $\ell_1(\mathcal V_1(v))$ (and marks the rest of vertices with, say, $-1$). Since the labels of the vertices in $S_1$ are pairwise distinct with probability $1$, the Voronoi cells can a.s.\ be read directly from the marking $m_{\ell_1}$. The point of defining $m_{\ell_1}$ is simply to make it clear that the Voronoi partition can be encoded by means of a marking.

\begin{claim}\label{claim:connected_Voronoi}
    For each $s\in S_1$, the induced subgraph $G_{D,\alpha}[V_{1,s}]$ is connected. Moreover, this subgraph contains the ball $B_{G_{D,\alpha}}(s,h/2)$ and is contained in the ball $B_{G_{D,\alpha}}(s,h)$.
\end{claim}

\medskip\noindent\textit{Remarks.}    This is a standard fact about Voronoi cells with well-separated centers, but we include the proof here since it is very simple.
 In this statement, we are implicitly assuming that $h$ is even. However, the claim still holds for any $h$ after substituting $h/2$ for $\lfloor h/2\rfloor$.

\begin{proof}
    In order to prove the connectedness, it suffices to show that for every vertex $v\in V_{1,s}$, every path $P$ of minimum length in $G_{D,\alpha}$ which connects $s$ to $v$ must be fully contained in $G_{D,\alpha}[V_{1,s}]$. Choose any such $P$ and number its vertices as $s=u_0,u_1,\dots,u_d=v$, in the order that they appear when traversing the path from $s$ to $v$. Suppose, for the sake of contraction, that there is some $i$ with $0\leq i<d$ such that $u_i\not\in V_s$, and assume that $i$ is actually the largest index with these properties. Then, there exists some $s'\in S_1$ with $s'\neq s$ such that \[d_{D,\alpha}(u_i,s')+\ell_1 (s')<d_{D,\alpha}(u_i,s)+\ell_1(s)=i+\ell_1(s)\,.\] Moreover, the vertex $u_{i+1}$ must belong to $V_{1,s}$ so, in particular, we have \[d_{D,\alpha}(u_{i+1},s')+\ell_1 (s')>d_{D,\alpha}(u_{i+1},s)+\ell_1(s)=i+1+\ell_1(s)\,.\] However, since $u_i$ and $u_{i+1}$ are adjacent in $G_{D,\alpha}$, $d_{D,\alpha}(u_{i+1},s')\leq d_{D,\alpha}(u_i,s')+1$. These three inequalities cannot hold simultaneously, so $P$ must be fully contained in $V_{1,s}$.

    Now, we show that any vertex $v\in B_{G_{D,\alpha}}(s,h/2)$ must also belong to $V_{1,s}$. Again, we proceed by contradiction. Assume that $v\in V_{1,s'}$ for some $s'\in S_1$ with $s'\neq s$. Then, we know that \[d_{D,\alpha}(u,s')+\ell_1 (s')<d_{D,\alpha}(u,s)+\ell_1(s)\leq h/2+1/2\,.\] From this, we can actually deduce that $d_{D,\alpha}(u,s')\leq h/2$, as $d_{D,\alpha}(u,s')$ is an integer and $h$ is even. Hence, by the triangle inequality, \[d_{D,\alpha}(s,s')\leq d_{D,\alpha}(s,u)+d_{D,\alpha}(u,s')\leq h/2+h/2=h\,.\] This contradicts the fact---guaranteed by Lemma~\ref{lem:maximal_k-independent_set}---that any two elements of $S_1$ are at distance at greater than $h$.

    Lastly, we show that $V_{1,s}$ is contained in $B_{G_{D,\alpha}}(s,h)$. For any vertex $v$ of $G_{D,\alpha}$, there exists some $s'\in S_1$ such that $d_{D,\alpha}(v,s')\leq h$. If $v\not\in B_{G_{D,\alpha}}(s,h)$, then $d_{D,\alpha}(v,s)\geq h+1$, and so $s'\neq s$. This also immediately implies that $s$ cannot belong to $V_{1,s}$.
\end{proof}

The next step consists of taking a Voronoi partition at a larger scale $R$, where $R$ is a positive integer to be specified later ($R$ shall depend on $\varepsilon$ and $N$, but for now one can simply think of it as being extremely large in comparison to $h$). More precisely, we begin by considering the auxiliary graph $G_{S_1}$, which is obtained by starting from $G$ and adding an edge between any two distinct elements $s$ and $s'$ of $S_1$ such that there is an edge in $G_{D,\alpha}$ with one endpoint in $V_{1,s}$ and the other in $V_{1,s'}$. By Lemma~\ref{lem:local-changes}, $(G_{S_1},\rho)$ is unimodular too. Clearly, every element of $S_1$ still has bounded degree in $G_{S_1}$. Consider now the percolation $\xi'$ of $(G,\rho)$ such that $\xi'(v)=1$ if and only if $\xi'\in S_1$ (equivalently, $\xi'(v)=1$ if and only if $m_1(v)=\star$). Applying Lemma~\ref{lem:maximal_k-independent_set} to $(G_{S_1},\rho)$ and $\xi'$ with parameter $R$, we obtain a new $\{\star,\bullet\}$-marking $m_2$.  Let $S_2\subseteq S_1$ denote the set of vertices in $S_1$ to which $m_2$ has assigned the mark $\star$. Note that any two elements of $S_2$ must be at distance at least $R+1$ not only in $G_{S_1}$, but also in $G_{D,\alpha}$. Similarly to what we did above, we let $\ell_2$ be a function assigning random, independent, uniform labels in $[0,1/2]$ to all elements of $S_2$, and we write $d_{S_1}$ to denote the shortest path metric on the graph $G_{S_1}$. Now, for every $v\in S_1$, we define\[\mathcal V_2(v):=\arg \min_{s\in S_2}\{d_{S_1}(v,s)+\ell_2(s)\}\,.\] Once again, $\mathcal V_2$ can be computed by a local algorithm. For every $s\in S_2$, the \textit{Voronoi macro-cell} of $s$, denoted by $V_{2,s}$, will consists of all those vertices $v$ in $G_{D,\alpha}$ such that $\mathcal V_2(\mathcal V_1(v))=s$. In other words, $V_{2,s}$ is the union of the Voronoi cells corresponding to the elements of $S_1$ which are mapped to $s$ under $\mathcal V_2$. Consider the marking $m_{\ell_2}$ which assigns to each vertex $v$ in $S_1$ the mark $\ell_2(\mathcal V_2(v))$, and to all other vertices $-1$. Then, the Voronoi macro-cells can a.s.\ be read directly from the markings $m_{\ell_1}$ and $m_{\ell_2}$.

By slightly modifying the strategy in the proof of Claim~\ref{claim:connected_Voronoi}, we obtain the following.

\begin{claim}\label{claim:connected_Voronoi_2}
    For each $s\in S_2$, the induced subgraph $G_{D,\alpha}[V_{2,s}]$ is connected. Moreover, this subgraph contains the ball $B_{G_{D,\alpha}}(s,\lfloor R/2\rfloor-h)$.
\end{claim}

\begin{proof}
    For each vertex $s'\in S_1$ which lies in $V_{2,s}$, the argument that was used to prove the first part of Claim~\ref{claim:connected_Voronoi} can be used directly to show that there is a path $P$ in $G_{S_1}$ which connects $s$ to $v$ and goes only through vertices in $S_1\cap V_{2,s}$. Consider an edge $(s_1,s_2)$ in this path and note that $s_1$ and $s_2$ belong to both $S_1$ and $V_{2,s}$, and that there is an edge connecting the two Voronoi cells $V_{1,s_1}$ and $V_{1,s_2}$. By Claim~\ref{claim:connected_Voronoi}, $G_{D,\alpha}[V_{1,s_1}]$ and $G_{D,\alpha}[V_{1,s_2}]$ are connected, so we can find a path completely contained in the graph $G_{D,\alpha}[V_{1,s_1}\cup V_{1,s_2}]$ which connects $s_1$ to $s_2$. Due to how the macro-cells are defined, these two Voronoi cells lie completely within $V_{2,s}$, and thus $G_{D,\alpha}[V_{1,s_1}\cup V_{1,s_2}]$ is a subgraph of $G_{D,\alpha}[V_{2,s}]$. We conclude that the edge $(s_1,s_2)$ of $P$ can be substituted by a path which connects its endpoints and lies completely within $G_{D,\alpha}[V_{2,s}]$; by the last part of the statement of Claim~\ref{claim:connected_Voronoi}, this path can be chosen to be of length at most $2h+1$ (this will be important during the proof of Theorem~\ref{thm:filaments}). Repeating this argument for every edge of $P$, we can construct a path $P'$ which connects $s$ to $s'$ and lies completely inside $G_{D,\alpha}[V_{2,s}]$.
    
    Now, suppose that $v$ is a vertex that belongs to both $V_{2,s}$ and $V_{1,s_v}$ for some $s_v\in S_1$. This implies that $s_v\in V_{2,s}$. The fact that $s_v\in V_{2,s}$ and $V_{1,s_v}\subseteq V_{2,s}$ now follows from the definition of $V_{2,s}$. Thus, taking $s_v$ in place of $s'$ during our discussion in the previous paragraph, we conclude that there exists some path in $G_{D,\alpha}[V_{2,s}]$ which connects $s$ to $s_v$. Using Claim~\ref{claim:connected_Voronoi} once again, we obtain that there exists some path connecting $s_v$ to $v$ which lies completely within $G_{D,\alpha}[V_{1,s_v}]$. Since $V_{1,s_v}\subseteq V_{2,s}$, this path is also contained in $G_{D,\alpha}[V_{2,s}]$. Hence, $G_{D,\alpha}[V_{2,s}]$ contains a path connecting $s$ to $s_v$, as well as a path connecting $s_v$ to $v$. This finishes the proof of the first statement in the claim.
    
    For the second part of the statement, we assume that there exists some vertex $v$ of $G_{D,\alpha}$ which belongs to $B_{G_{D,\alpha}}(s,\lfloor R/2\rfloor-h)$ but not $V_{2,s}$. Consider the vertex $s_v\in S_1$ satisfying $\mathcal V_1(v)=s_v$. Then, by the last part of the statement of Claim~\ref{claim:connected_Voronoi}, we get that $d_{G_{D,\alpha}}(v,s_v)\leq h$. The fact that $v\not\in V_{2,s}$ implies the existence of some $s'\in S_2$ such that \[d_{S_1}(s_v,s')+\ell_2(s')<d_{S_1}(s_v,s)+\ell_2(s)<d_{D,\alpha}(s_v,s)+\ell_2(s)<d_{D,\alpha}(s,v)+h+\ell_2(s)\,,\] where in the last step we have used the triangle inequality. This yields that $d_{S_1}(s_v,s')\leq d_{D,\alpha}(s,v)+h$. Since $v \in B_{G_{D,\alpha}}(s,\lfloor R/2\rfloor-h)$, we have that $d_{G_{D,\alpha}}(s,v)\leq\lfloor R/2\rfloor-h$. Putting things together and using the triangle inequality one more time, we conclude that \[d_{D,\alpha}(s,s')\leq d_{D,\alpha}(s,v)+d_{D,\alpha}(v,s_v)+d_{D,\alpha}(s_v,s')\leq  (\lfloor R/2\rfloor-h)+h+(\lfloor R/2\rfloor-h+h)\leq R\,.\] This contradicts the fact that any two elements of $S_2$ are at distance at least $R+1$ in $G_{D,\alpha}$.
\end{proof}

We are now ready to construct the percolation $\omega$ described in the statement of Theorem~\ref{thm:filaments} (i.e., the filaments).

\begin{proof}[Proof of Theorem~\ref{thm:filaments}]
    Note that it suffices to describe how $\omega$ behaves on the infinite connected components of $G_{D,\alpha}$. Indeed, we lose nothing by simply setting $\omega$ to be $1$ on all vertices of $G$ which do not belong to such a connected component. Let $S_1'$ and $S_2'$ denote the subsets of $S_1$ and $S_2$, respectively, which consist of those vertices that belong to an infinite connected component of $G_{D,\alpha}$. 
    
    We begin by discussing how to construct a filament $\Psi_s$ within $V_{2,s}$ for every $s\in S_2'$. Consider such a vertex $s$ and write $W_s=S_1\cap V_{2,s}$. Now, among all subgraphs of $G_{D,\alpha}[V_{2,s}]$ which are trees and contain every element of $W_s$, we select one with the least total number of vertices and call it $T_s$. The \textit{filament} $\Psi_s$ will consist of those vertices that belong to $T_s$. We remark that $T_s$ and $\Psi_s$ are well defined, since $G_{D,\alpha}[V_{2,s}]$ is connected. Furthermore, $G_{D,\alpha}[\Psi_s]$ is connected by definition, and the set $\Psi:=\bigcup_{s\in S_2'}\Psi_s$ contains every element of $S'_1$.

    We proceed to prove that, if $R$ is sufficiently large with respect to $h$, then the percolation $\omega$ defined by letting
  \begin{align*}
    \omega(v) :=
    \begin{cases*}
      0 & if $v\in \Psi$,\\
      1        & otherwise
    \end{cases*}
  \end{align*}
  for every vertex $v$ satisfies the properties in the statement of the theorem. 
  
    We begin by bounding the rate of $\omega$ from above. The main goal is to show that, for every $s\in S_2'$, $|\Psi_s|\leq \varepsilon|V_{2,s}|$. For starters, as we saw in the proof of Claim~\ref{claim:connected_Voronoi_2}, any two vertices $s_1,s_2\in W_s$ for which there is an edge joining $V_{1,s_1}$ to $V_{1,s_2}$ can be connected by means of a path of length at most $2h+1$ which lies entirely within $G_{D,\alpha}[V_{2,s}]$. This can be seen to imply that 
    \begin{equation}\label{eq:trees_upper_bound}
        |\Psi_s|\leq 2h(|W_s|-1)+1<2h|W_s|\,.
    \end{equation} 
    
    Next, we lower bound $|V_{2,s}|$ in terms of $|W_s|$. This is the first step of the proof at which the expansion of $G_{D,\alpha}$ comes into play. By Claim~\ref{claim:connected_Voronoi}, the ball $B_{G_{D,\alpha}}(s',h/2)$ is contained in $V_{1,s'}$ for every $s'\in W_s$. In particular, as $s'$ ranges over all elements of $W_s$, the balls $B_{G_{D,\alpha}}(s',h/2)$ are pairwise disjoint, and hence 
    \begin{equation}\label{eq:disjoint_balls}
        \sum_{s'\in W_s}|B_{G_{D,\alpha}}(s',h/2)|\leq |V_{2,s}|\,.
    \end{equation}
     However, we know that for any $s'\in W_s$ and any positive integer $k$, \[\frac{|\partial_{V,\operatorname{in}}^{G_{D,\alpha}} B_{G_{D,\alpha}}(s',k)|}{|B_{G_{D,\alpha}}(s',k)|}\geq\alpha\] (note that this is a slight abuse of notation, since $B_{G_{D,\alpha}}(s',k)$ is not a set of vertices, but a graph; still, the meaning should be clear). In particular, there are at least $\alpha\cdot|B_{G_{D,\alpha}}(s',k)|$ edges in $G_{D,\alpha}$ joining a vertex in $B_{G_{D,\alpha}}(s',k)$ to a vertex outside $B_{G_{D,\alpha}}(s',k)$. Since every vertex in $G_{D,\alpha}$ has degree at most $D$, this yields 
     \[|\partial_{V,\operatorname{out}}^{G_{D,\alpha}} B_{G_{D,\alpha}}(s',k)|\geq\alpha D^{-1}\cdot | B_{G_{D,\alpha}}(s',k)|\,.\]
     Iteratively applying this bound for $k=0,1,\dots h/2-1$, one obtains 
     \begin{equation}\label{eq:exp_growth}
         |B_{G_{D,\alpha}}(s',h/2)|\geq(1+\alpha D^{-1})^{h/2}\,.
     \end{equation} 
     Together with~\eqref{eq:disjoint_balls}, this gives us 
     \begin{equation}\label{eq:lower_bound_Vs}
         |V_{2,s}|\geq (1+\alpha D^{-1})^{h/2}\cdot|W_s|\,.
     \end{equation} 
     
     Finally,~\eqref{eq:trees_upper_bound} and~\eqref{eq:lower_bound_Vs} can be combined to obtain \[|\Psi_s|\leq 2h|W_s|\leq 2h\cdot |V_{2,s}|\cdot(1+\alpha D^{-1})^{-h/2}\leq \varepsilon|V_{2,s}|\,,\] where the last inequality follows from the choice of $h$.

    As in the proof of Theorem~\ref{thm:non-amenable}, we construct a simple mass transport $f$ to which we apply the mass transport principle~\eqref{eq:MTP}. This $f$ will be defined by\footnote{Note that, similarly as before, our notation does not reflect the dependency of $f$ on $\xi$, $\omega$, $m_{\ell_1}$ and $m_{\ell_2}$.}
    \begin{align*}
    f(G,\rho,o) :=
    \begin{cases*}
      1/|V_{2,s}| & if $\rho\in \Psi_s$ and $o\in V_{2,s}$ for some $s\in S_{2}'$,\\
      0        & otherwise.
    \end{cases*}
  \end{align*}
    If there exists some $s$ such that the first condition holds, then this $s$ is unique, so the function is well defined. 
    
    The expected amount of mass leaving the root is \[\mathbb E_{(G,\rho,\xi),\omega}\left[\sum_{o\in V(G)}f(G,\rho,o)\right]=\mathbb E\left[\sum_{s\in S_2'}\mathbf 1[\rho\in \Psi_s]\cdot\left(|V_{2,s}|\cdot\frac{1}{|V_{2,s}|}\right)\right]=\mathbb P[\rho\in \Psi]=\mathbb P[\omega(\rho)=0]\,.\] On the other hand, the expected amount of mass entering the root is \[\mathbb E_{(G,\rho,\xi),\omega}\left[\sum_{o\in V(G)}f(G,o,\rho)\right]=\mathbb E\left[\sum_{s\in S_2'}\mathbf 1[\rho\in V_{2,s}]\cdot\left(|\Psi_s|\cdot\frac{1}{|V_{2,s}|}\right)\right]\leq\mathbb E[1\cdot \varepsilon]= \varepsilon\,,\] where we have used the fact that $|\Psi_s|\leq\varepsilon|V_{2,s}|$ for all $s\in S_2'$, and then the pairwise disjointness of the Voronoi macro-cells for the last inequality. The mass transport now immediately tells us that $\mathbb P[\omega(\rho)=0]\leq \varepsilon$, as required.
    
    That Property (I) holds is almost immediate. Indeed, if $v$ is a vertex in some infinite connected component of $G_{D,\alpha}$, there is some $s_v\in S_1'$ such that $v\in V_{1,s_v}$. Clearly, $s_v$ is a vertex of $G_{D,\alpha}$ which satisfies $\omega(s_v)=0$. Now, Claim~\ref{claim:connected_Voronoi} implies that $v$ and $s_v$ are at distance at most $h$ with respect to $G_{D,\alpha}$, as desired.

    For Property (II), note that every vertex which is closed with respect to $\omega$ belongs to some filament $\Psi_s$. Hence, it suffices to show that $|\Psi_s|\geq N$ for every $s\in S_2'$. For every $s'\in S_1'$, since each vertex of $G_{D,\alpha}$ has degree at most $D$ and $G_{D,\alpha}[V_{1,s'}]$ is contained in $B_{G_{D,\alpha}}(s',h)$, \[|V_{1,s'}|\leq 1+D+D(D-1)+\dots+D(D-1)^{h-1}<D^{h+1}\,,\] where we have used that $D\geq2$. Given that $V_{2,s}=\bigcup_{s'\in W_s}V_{1,s'}$, we also have that
    \begin{equation}\label{eq:exp_growth2}
        |V_{2,s}|\leq \sum_{s'\in W_s}|V_{1,s'}|\leq|W_s|\cdot D^{h+1}\,.
    \end{equation}

    Next, we wish to lower bound $|V_{2,s}|$. This is where the second part of Claim~\ref{claim:connected_Voronoi_2} will be useful. Indeed, in the same way that~\eqref{eq:exp_growth} was obtained, one arrives at the inequality \[|V_{2,s}|\geq|B_{G_{D,\alpha}}(s',\lfloor R/2\rfloor-h)|\geq(1+\alpha D^{-1})^{\lfloor R/2\rfloor-h}\,.\] Alongside~\eqref{eq:exp_growth2}, this implies that $|W_s|\geq(1+\alpha D^{-1})^{\lfloor R/2\rfloor -h}\cdot D^{-h-1}.$ Since $W_s\subset \Psi_s$, it suffices to take $R$ large enough that \[(1+\alpha D^{-1})^{\lfloor R/2\rfloor -h}\cdot D^{-h-1}\geq N\,.\] 
    This concludes the proof.
\end{proof}

\section{Locally-small treewidth}\label{sec:local-treewidth}

For any two positive integers $t$ and $R$, a graph $G$ will be said to be $(t,R)$\textit{-locally-thin} if $B_G(v,R)$ has treewidth at most $t$ for every vertex $v$ of $G$.

In this section we show that if we start with a graph $(G,\rho)$ as in the statement of Theorem~\ref{thm:main} and then apply theorems~\ref{thm:non-amenable} and~\ref{thm:filaments} so as to obtain a graph $G_{\alpha,D}[\omega]$, then this new graph will be locally-thin. More precisely, we prove the following theorem.

\begin{theorem}\label{thm:locally_treelike}
    Let $\alpha$ be a positive real number and $D\geq 2$ be a positive integer. Then, for every $\varepsilon>0$ and every two positive integers $n$ and $r$, there exist two positive integers $t=t(\varepsilon,n)$ and $N=N(r)$, with $t$ independent of $r$ and $N$ depending only on $r$, such that the following holds: 
    
    Let $H$ be a graph on $n$ vertices and $(G,\rho)$ be a unimodular random rooted graph such that $G$ is a.s.\ one-ended and has no $H$-minor. Suppose that $\xi$ and $\omega$ are percolations of $(G,\rho)$ which satisfy the properties in the statement of Theorem~\ref{thm:filaments} with parameters $\varepsilon$ and $N$, and write $G_{D,\alpha}$ to denote $G[\xi]$. Then, every connected component of the graph $G_{D,\alpha}[\omega]$ which is contained in an infinite connected component of $G_{D,\alpha}$ is $(t,r)$-locally-thin.
\end{theorem}

As mentioned in Section~\ref{thm:locally_treelike}, the main tool that will be used to control the treewidth of balls of radius $r$ in $G_{D,\alpha}[\omega]$ is be the grid-minor theorem (Theorem~\ref{thm:treewidth-walls}). Let us prepare for the proof of the above statement with a few further preliminary results. 

By a \textit{crossed-}$k\times \ell$\textit{-grid}, we mean the graph obtained from a $k\times \ell$-grid by adding the two diagonals connecting the opposite vertices within every cell, as shown in Figure~\ref{fig:crossed-grid} for $k=\ell=7$.

\begin{figure}[ht]
    \centering
    \includegraphics[width=0.25\linewidth]{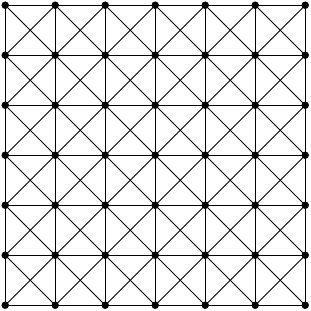}
    \caption{A crossed-$7\times 7$-grid.}
    \label{fig:crossed-grid}
\end{figure}

\begin{proposition}\label{prop:crossed-grid_minors}
    For every positive integer $n$, there exists another positive integer $\boxtimes (n)$ such that the crossed-$\boxtimes(n)\times \boxtimes(n)$-grid contains the complete graph $K_n$---and thus any other graph of order $n$---as a minor.
\end{proposition}

\begin{proof}[Proof sketch]
    At a high level, the proof proceeds as follows:
    
    First, we consider a drawing of $K_n$ on the plane with no three edges crossing at a single point. After adding a dummy vertex at each crossing between any two edges, we obtain a new drawing of some planar graph $K_n'$. Then, we use this drawing as a blueprint to replicate $K_n$ within some very large crossed grid. Essentially, we start by representing $K_n'$ as a minor within a sufficiently large grid, and then use the crossing diagonals within the cells of the corresponding crossed grid in order to properly handle all dummy vertices. A more detailed argument is given in the \hyperref[sec:Appendix]{appendix}, but can safely be skipped, as it is disconnected from the remainder of this section.
\end{proof}

Our next result states that a connected graph of bounded degree must contain many vertex-disjoint paths with endpoints in some specified set, so long as the set itself is large.

\begin{lemma}\label{lem:disjoint_leaps}
    Let $G$ be a connected graph of maximum degree at most $D$ for some positive integer $D$. Then, for every set $S$ of $n$ vertices of $G$, it is possible to construct at least \[\left\lceil\frac{n-1}{D}\right\rceil\,\] vertex-disjoint paths in $G$ all of whose endpoints belong to $S$. 
\end{lemma}

\begin{proof}
    We proceed by induction on $n$. The statement is clearly true whenever $n\leq D+1$. 
    
    We may assume, without loss of generality, that $G$ is a tree whose vertices have degree at most $D$. Fix an arbitrary vertex $\rho$ of $G$ to be the root. 
    Now, for every vertex $v$ of $G$, let $T_v$ denote the subtree of $G$ which consists of all those vertices $u$ such that any path in $G$ connecting $\rho$ to $u$ contains the vertex $v$. Note that $v$ itself is a vertex of $T_v$. 
    
    Chose a vertex $v$ which is at the maximum possible distance from $\rho$ while still having the property that $T_v$ contains at least two elements of $S$. Within $T_v$, we can find at least one path $P$ joining two vertices of $S$. Furthermore, since every vertex of $G$ has degree at most $D$, the choice of $v$ guarantees that $T_v$ contains at most $D$ elements of $S$---unless $v=\rho$, in which case $S$ has at most $D+1$ elements, and we are done. We may now add $P$ to our collection of paths and remove every vertex of $T_v$ from the graph to obtain a new graph $G'$, to which we apply the inductive hypothesis. The graph $G'$ contains at least $n-D$ elements of $S$, so we can find \[\left\lceil\frac{n-1-D}{D}\right\rceil=\left\lceil\frac{n-1}{D}\right\rceil-1\] vertex-disjoint paths within it whose endpoints belong to $S$. Adding $P$ to this collection, the result follows.
\end{proof}

We will also require the following related result, where we now have multiple specified sets of vertices and our goal is to find several disjoint paths connecting pairs of vertices which belong to the same set.

\begin{lemma}\label{lem:disjoint_leaps_2}
    Let $G=(V,E)$ be a connected graph of maximum degree at most $D$ and $m$ be a positive integer. Suppose that $S_1,\dots,S_m$ are pairwise disjoint subsets of $V$, each of size at least $2+(D-1)(m-1)$. Then, there exist $m$ vertex-disjoint paths $P_1,\dots,P_m$ in $G$ such that, for all $1\leq i\leq m$, the endpoints of $P_i$ belong to $S_i$ .
\end{lemma}

\begin{proof}
    Once more, we proceed by induction. The statement is clearly true for $m=1$. The argument below is very similar to the one in the proof of the previous lemma.

    Again, we may assume that $G$ is a tree all of whose vertices have degree at most $D$. Suppose that $S_1,\dots,S_m$ satisfy the hypothesis in the statement of the theorem. Root $G$ at some arbitrary vertex $\rho$ and, for every vertex $v\in V$, define $T_v$ as in the proof of Lemma~\ref{lem:disjoint_leaps}. Let $v$ be some vertex which is at the maximum possible distance from $\rho$ while still having the property that $T_v$ contains at least two elements of $S_i$ for at least one index $i$. If $v=\rho$, then for any of its at most $D$ neighbors $u$, the tree $T_u$ contains at most one element from each set $S_i$. If it were the case that $m\geq 2$, then there would be some set $S_i$ such that $\rho\not\in S_i$ and thus we would have that $|S_i|\leq D$, which is only possible if $m=1$. Hence, we may assume that $v\neq \rho$. 
    
    If there is some index $i$ such that $v$ belongs to $S_i$ and $T_v$ contains at least $2$ elements from $S_i$, then we let $P_i$ be some path which is fully contained in $T_v$ and has both endpoints in $S_i$. Otherwise, we arbitrarily choose some index $i$ with the property that $T_v$ contains at least two elements from $S_i$, and again let $P_i$ be a path in $T_v$ with both endpoints in $S_i$. In any case, for every index $j\neq i$, $T_v$ contains at most $D-1$ elements of $S_j$. Delete the subtree $T_v$ from $G$ to obtain a new graph $G'$, and let us forget about the set $S_i$. For each index $j\neq i$, the resulting graph contains at least $2+(D-1)(m-2)$ vertices from $S_j$. Thus, by the inductive hypothesis, there exist some vertex-disjoint paths $P_1,\dots,P_{i-1},P_{i+1},\dots,P_m$ within $G'$ such that $P_j$ has endpoints in $S_j$ for every $j\neq i$. Adding $P_i$ to this collection, we conclude the proof.
\end{proof}

We are now ready to move on to the proof of our main result in this section.

\begin{proof}[Proof of Theorem~\ref{thm:locally_treelike}]
    The values of $t$ and $N$ will be specified later, and the reader can think of them as being fixed for now. We will show that if some infinite connected component of $G_{D,\alpha}$ contains as a subgraph a subdivision of the wall $W_{g(t)\times g(t)}$ which consists only of vertices that are open with respect to $\omega$, then this subdivision must have so many vertices that it cannot possibly lie within a ball of radius $r$. 
    
    The first part of the proof essentially follows our discussion in Section~\ref{subsec:locally_treelike}. Suppose, for the sake of contradiction, that there is some subdivision $W$ of the wall $W_{g(t)\times g(t)}$ contained in some ball of radius $r$ in $G_{D,\alpha}[\omega]$ which, in turn, lies in an infinite connected component of $G_{D,\alpha}$. Given a path in $W$, we define its \textit{wall-length} as the number of vertices it contains, apart from its endpoints, which have degree $3$ in $W$ (that is, we ignore all subdivision vertices). The \textit{wall-distance} between two vertices $u$ and $v$ of $W$ is the least wall-length among all paths in $W$ with endpoints $u$ and $v$. A path $P$ in $G$ which has two vertices of $W$ as its endpoints will be called a \textit{leap} if it satisfies the following properties:
\begin{itemize}
    \item the only vertices of $P$ that belong to $W$ are its two endpoints;

    \item the graph $P\cup W$ is not planar.
\end{itemize}
    The first of our two main goals will be to show that we can find many vertex-disjoint leaps whose endpoints are somewhat spread out throughout $W$.
    
Let $h=h(\varepsilon)$ be as in the statement of Theorem~\ref{thm:filaments} and note that, without loss of generality, we may assume that $h\geq 4D$. Let $M=M(n)$ be some large positive integer to be specified later, and assume that $t$ is large enough that $g(t)\geq M\cdot7h$. Then, $W$ contains $M^2$ vertex-disjoint subdivisions of the $7h\times 7h$-wall which are arranged in an $M\times M$-grid-like fashion. An example of this has been depicted on the left of Figure~\ref{fig:walls_within_wall}. We label these subdivisions as $\overline W_1,\dots,\overline{W_{M^2}}$ in any order. For each $1\leq i\leq M^2$, let $W_i$ be the subdivision of the $5h\times 5h$-wall which is centered inside $\overline{W_i}$. For each such $i$, we also let $B_i$ denote the set of vertices on the outer-most layer of $W_i$, and then select a vertex $u_i$ from $W_i$ which lies at wall-distance at least $2h$ from every vertex in $B_i$. See the example on the right of Figure~\ref{fig:walls_within_wall}. 

\begin{figure}[ht]
    \centering
    \includegraphics[width=1\linewidth]{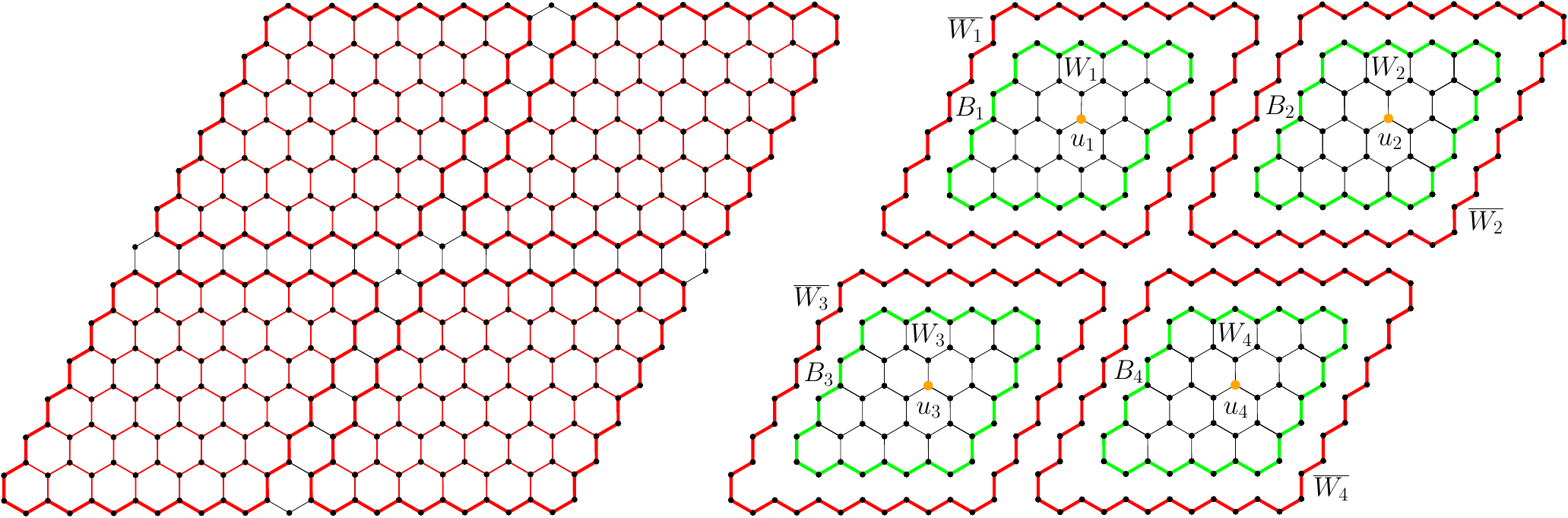}
    \caption{The picture on the left, which corresponds to the case $h=1$ and $M=2$, shows a $14\times 14$-wall where four $7\times 7$-walls have been highlighted in red. These $7\times 7$-walls are arranged to form a $2\times 2$-grid-like structure, and are labeled as $\overline{W_1},\dots,\overline {W_{4}}$. On the right, we see that within each $\overline{W_i}$ a $5\times 5$-wall $W_i$ has been chosen. We use $B_i$ to denote the outer-most layer of $W_i$, which has been highlighted in green. Lastly, within each $W_i$, we select a vertex $u_i$ (shown in orange) which lies at wall-distance at least $2$ from every vertex in $B_i$.}
    \label{fig:walls_within_wall}
\end{figure}

Since $u_i$ is a vertex of $G_{D,\alpha}[\omega]$ and its cluster with respect to $\xi$ is infinite, property (I) in the statement of Theorem~\ref{thm:filaments} guarantees the existence of a path $P_i$ of length at most $h$ which is contained in $G_{D,\alpha}$ and connects $u_i$ to some vertex $v_i$ that is closed with respect to $\omega$. Property (II) in the same theorem now ensures that $v_i$ belongs to some set $\Psi_i$ consisting of at least $N$ vertices from $G_{D,\alpha}$, all of which are closed with respect to $\omega$, and such that $G_{D,\alpha}[\Psi_i]$ is connected. As was essentially observed in Section~\ref{subsec:locally_treelike} (see Figure~\ref{fig:leap_to_void} and its description), the fact that $u_i$ is at wall-distance at least $2h$ from $B_i$ and $P_i$ has length at most $h$ implies that $P$ must satisfy at least one of the two following properties:
    \begin{enumerate}[label=(\alph*)]
        \item every vertex of $P_i$ which belongs to $W$ also belongs to $W_i$;
        
        \item there is at least one subsection of $P_i$ which constitutes a leap that is contained in $G_{D,\alpha}$ and has at least one endpoint in $W_i$.
    \end{enumerate}

   \begin{claim}\label{claim:many_leaps}
       There exists a collection of leaps $\mathcal L$  with the following properties:
    \begin{itemize}
        \item for all but at most one of the subdivisions $W_1,\dots,W_{M^2}$, there is at least one leap in $\mathcal L$ with at least one endpoint inside the subdivision;

        \item no vertex that is closed with respect to $\xi$ belongs to more than one of the leaps in $\mathcal L$.
    \end{itemize}
    \end{claim}

    \begin{proof}[Proof of Claim~\ref{claim:many_leaps}]
    As we show below, the collection $\mathcal L$ can be constructed greedily.
    
    Suppose, on the contrary, that at some point there exist two indices $i$ and $j$ with $1\leq i<j\leq M^2$ such that no leap with an endpoint within $W_i$ or $W_j$ can be added to $\mathcal L$ without violating the second of the two conditions above (this also encompasses the situation where there is simply no leap with an endpoint in one of the two subdivisions). Then, property (a) above must hold for both $i$ and $j$, or we would be able to add $P_i$ or $P_j$ to $\mathcal L$. Let $V_{\xi,\mathcal L}$ be the set of vertices of $G$ which are closed with respect to $\xi$ and belong to some leap in $\mathcal L$; this set is finite. Denote by $W_{\backslash i,j}$ the set of vertices of $W$ which do not belong to $W_i$ nor $W_j$. Note that $i$ and $j$ must belong to different connected components of the graph $G_{D,\alpha}\backslash W_{\backslash i,j}$, or else there would exist some leap which is contained in $G_{D,\alpha}$ and has one endpoint in $W_i$ and the other in $W_j$, and we could add it to $\mathcal L$. Now, since property (a) holds for both indices, the paths $P_i$ and $P_j$ are completely contained within $G_{D,\alpha}\backslash W_{\backslash i,j}$. Moreover, the sets $\Psi_i$ and $\Psi_j$ are also completely contained in $G_{D,\alpha}\backslash W_{\backslash i,j}$, since they consist only of vertices which are closed with respect to $\omega$ and $W$ is made out of vertices which are open. If follows that the connected components of $i$ and $j$ in $G_{D,\alpha}\backslash W_{\backslash i,j}$ contain $\Psi_i$ and $\Psi_j$, respectively. Suppose for a moment that both of these components are infinite. Then, the fact that $G$ is one-ended, together with the observation that both $V_{\xi,\mathcal L}$ and $W_{\backslash i,j}$ are finite, tells us that we can find some path in $G\backslash(W_{\backslash i,j}\cup V_{\xi,\mathcal L})$ which connects $W_i$ to $W_j$. This path can be added to $\mathcal L$, thus contradiction our initial assumption. Hence, we may assume that at least one of the two connected components is finite. 

    Suppose, without loss of generality, that the connected component of $i$ in $G_{D,\alpha}\backslash W_{\backslash i,j}$ is finite. Let $U$ denote the set of vertices in this component, and note that $W_{\backslash i,j}\supseteq\partial_{V,\operatorname{out}}^{G_{D,\alpha}}U$. Since $\Psi_i$ is a subset of $U$, we know that $|U|\geq|\Psi_i|\geq N$. Now, the fact that every infinite connected component of $G_{D,\alpha}$ has vertex-expansion constant at least $\alpha$, along with our assumption that $W$ (and thus also $U$) is fully contained in such a component, implies \[|\partial_{V,\operatorname{in}}^{G_{D,\alpha}}U|\geq\alpha|U|\geq\alpha N\,.\] Since $G_{D,\alpha}$ has maximum degree at most $D$, we also know that $|\partial_{V,\operatorname{out}}^{G_{D,\alpha}}U|\geq|\partial_{V,\operatorname{in}}^{G_{D,\alpha}}U|/D$. Putting everything together, one concludes that \[|W|> |W_{\backslash i,j}|\geq |\partial_{V,\operatorname{out}}^{G_{D,\alpha}}U|\geq |\partial_{V,\operatorname{in}}^{G_{D,\alpha}}U|/D\geq (\alpha D^{-1}) N\,.\] On the other hand, every ball of radius $r$ in $G_{D,\alpha}[\omega]$ contains less than $D^{r+1}$ vertices. Hence, as long as $N=N(r)$ is chosen such that \[(\alpha D^{-1})N\leq D^{r+1}\,,\] $W$ cannot be contained within a ball of radius $r$ in $G_{D,\alpha}[\omega]\,.$ This contradiction concludes the proof of the claim.
    \end{proof}

    Let $\mathcal L$ be a minimal set of leaps---with respect to inclusion---satisfying the properties in the statement of the above claim. Our next step consists of showing that if $M$ is large enough with respect to $n$ then the graph $W\cup\left(\bigcup_{P\in \mathcal L} P\right)$ must have $H$ as a minor. Towards this goal, we begin by finding, within this graph, many vertex-disjoint leaps with endpoints in different $\overline{W_i}$'s.

\begin{claim}
    There exists a family $\mathcal L'$ of pairwise vertex-disjoint leaps such that, for at least 
    \begin{equation}\label{eq:many_disjoint_leaps}
         \frac{M^2-1}{4D}\,
    \end{equation} distinct indices $1\leq i\leq M^2$, there is some leap in $\mathcal L'$ which has at least one endpoint within $\overline{W_i}$.
\end{claim}
    
    \begin{proof}
        By the minimality of $\mathcal L$, each leap $P\in \mathcal L$ has an endpoint $v_P$ in some subdivision $W_i$ such that $P$ is the only leap in $\mathcal L$ with an endpoint in $W_i$. Hence, we can select a set $S_\mathcal L$ of at least $M^2-1$ vertices from $\bigcup_{i=1}^{M^2}W_i$ such that: each of its elements is and endpoint of some leap in in $\mathcal L$, no two of its elements belong to the same $W_i$, and $v_P\in S_\mathcal L$ for each $P\in \mathcal L$. Now, let $G_{\mathcal L}$ be the graph $\bigcup_{P\in \mathcal L}P$ and define $S_\mathcal L'$ as the set of vertices of $G_\mathcal L$ which belong to $W\backslash S_\mathcal L$. Let $G_{\mathcal L,1},\dots,G_{\mathcal L,c}$ be the connected components of $G_\mathcal L$ and, for each $1\leq i\leq c$, write $S_{\mathcal L,i}$ (resp. $S_{\mathcal L,i}'$) to denote the set of vertices of $S_{\mathcal L}$ (resp. $S_\mathcal L'$) which belong to $G_{\mathcal L,i}$. By the properties of $S_\mathcal L$, it is not hard to see that $|S_{\mathcal L,i}'|\leq |S_{\mathcal L,i}|$ holds for all indices $i$.
        
        We will apply Lemma~\ref{lem:disjoint_leaps} within each connected component $G_{\mathcal L,i}$ with $|S_{\mathcal L,i}|\geq 2$. Each vertex which is closed with respect to $\xi$ shows up in at most one of the leaps in $\mathcal L$, and every other vertex of $G$ has degree at most $D$. Hence, every vertex of $G_{\mathcal L}$ has degree at most $\max\{2,D\}=D$. Thus, the lemma guarantees that, within each such connected component, there exists a set $\mathcal P_i$ of \[\left\lceil\frac{|S_{\mathcal L,i}|-1}{D}\right\rceil\geq\frac{|S_{\mathcal L,i}|}{2D}\] pairwise vertex-disjoint paths with endpoints in $S_{\mathcal L,i}$. Let $P\in \mathcal P_i$ be one these paths and suppose that one of its endpoints lies in $W_j$. Observe that, without loss of generality, we can assume that $P$ does not go through any vertex of $S_{\mathcal L}$ other than its endpoints. Note that $P$ is not necessarily a leap, as it could contain vertices from $S_{\mathcal L,i}'$ in its interior. However, for all $1\leq i\leq M^2$, every vertex in $W_i$ is at wall-distance at least $2h$ from every vertex of $W$ which lies outside $\overline{W_i}$. This tells us that $P$ satisfies at least one of the two following properties: either it contains at least $h$ elements of $S_{\mathcal L,i}'$, or some subsection of it constitutes a leap with an endpoint in $\overline {W_{i}}$---this is once again true by essentially the same argument that appeared in the caption of Figure~\ref{fig:leap_to_void}. 
        
        Since the paths in $\mathcal P_1$ are pairwise vertex-disjoint, there are at most $|S_{\mathcal L,i}'|/h\leq |S_{\mathcal L,i}|/h$ choices of $P$ for which the first of these two properties above is satisfied. Furthermore, as we go trough all paths in $\mathcal P_i$, the indices $j$ will all be distinct. As $h\geq 4D$, it follows that we can find, within $G_{\mathcal L,i}$, a collection $\mathcal L_i$ of  pairwise vertex-disjoint leaps such that the set of all their endpoints intersects at least \[\frac{|S_{\mathcal L,i}|}{2D}-\frac{|S_{\mathcal L,i}|}{h}\geq \frac{|S_{\mathcal L,i}|}{4D}\] distinct $\overline {W_{j}}$'s. 

        Adding over all connected components $G_{\mathcal L,i}$ with $|S_{\mathcal L,i}|\geq 2$, and then noting that every connected component with $|S_{\mathcal L,i}|=1$ must simply be a path that can be directly added to $\mathcal L'$, we arrive at the desired conclusion. Here, we have used the fact that $|S_{\mathcal L}|\geq M^2-1$.
    \end{proof}

Let $\mathcal L'$ be as in the above claim. For the sake of convenience, we shall relabel the subdivisions $\overline{W_1},\dots,\overline{W_{M^2}}$ as $\overline{W_{i,j}}$ ($1\leq i,j\leq M$), where $\overline{W_{i,j}}$ denotes the subdivision which lies in row $i$ and column $j$ within the $M\times M$-grid-like structure formed by the $\overline{W_i}$'s. Write \[Q=Q(n):=\left\lceil\sqrt{2+(D-1)\cdot(\boxtimes(n)^2-2\boxtimes(n))}\right\rceil+2\] and define \[K=K(n):=(\boxtimes(n)-1)\cdot\left[(\boxtimes (n)-1)\cdot Q+2\right]\,.\] We will say that a subdivision $\overline{W_{i,j}}$ is $\mathcal L'$\textit{-connected} if there is some leap in $\mathcal L'$ with an endpoint in $\overline{W_i}$. Taking $M\geq 5$ ensures that the expression in~\eqref{eq:many_disjoint_leaps} evaluates to at least $M^2/(5D)$, and thus at least $M^{2}/(5D)$ of the $\overline{W_{i,j}}$'s are $\mathcal L'$-connected. Now, the multidemnsional version of Szemerédi's theorem~\cite{furstenberg1978ergodic}\footnote{See the \hyperref[sec:Appendix]{appendix} and, in particular, Theorem~\ref{thm:generalized-szemeredi} for further details regarding this step.} tells us that, if $M$ is sufficiently large, then there must exist a pair of integers $(a,b)$ and some positive integer $q$ with $1\leq a,b\leq M-q(K-1)$ such that $\overline{W_{a+qk_1,b+qk_2}}$ is $\mathcal{L}'$-connected whenever $0\leq k_1,k_2\leq K-1$. In other words, there exists a family of $\mathcal L'$-connected subdivisions which are arranged in a $K\times K$-grid-like structure, although we have no control on the spacing between the rows/columns of this $K\times K$-grid. 

We now classify the subdivisions in this $K\times K$ arrangement into \textit{blocks}. The blocks will be indexed by ordered pairs $(i,j)$ of integers with $1\leq i,j\leq \boxtimes(n)-1$, and will be accordingly denoted as $Z_{i,j}$. The \textit{block} $Z_{i,j}$ consists of all those subdivisions $\overline{W_{a+qk_1,b+qk_2}}$ with \[(i-1)\frac{K}{\boxtimes(n)-1}\leq k_1<i\frac{K}{\boxtimes (n)-1}\quad\text{and}\quad (j-1)\frac{K}{\boxtimes(n)-1}\leq k_2<j\frac{K}{\boxtimes (n)-1}\,.\] This way, the blocks constitute a partition of the $K\times K$ arrangement of subdivisions, and are themselves arranged in a $(\boxtimes(n)-1)\times (\boxtimes(n)-1)$-grid-like manner. Moreover, each block itself consists of a $K'\times K'$ array of subdivisions of the form $\overline{W_{a+qk_1,b+qk_2}}$, where \[K':=\frac{K}{\boxtimes(n)-1}=(\boxtimes (n)-1)\cdot Q+2\,.\] 

We will say that a subdivision $\overline{W_{a+qk_1,b+qk_2}}$ is an \textit{inner} subdivision of the block $Z_{i,j}$ if \[(i-1)\frac{K}{\boxtimes(n)-1}+1\leq k_1<i\frac{K}{\boxtimes (n)-1}-1\quad\text{and}\quad (j-1)\frac{K}{\boxtimes(n)-1}+1\leq k_2<j\frac{K}{\boxtimes (n)-1}-1\,.\] In other words, the inner subdivisions are those which belong to the block but do not lie on its outer-most layer. A vertex $v$ of $W$ will be said to be \textit{enclosed} by $Z_{i,j}$ if either it belongs to an inner subdivision of this block, or the subdivision of the form $\overline{W_{a+qk_1,b+qk_2}}$ which is closest to $v$ in wall-distance is an inner subdivision of $Z_{i,j}$. Note that the vertices which are not enclosed by $Z_{i,j}$ span a connected subgraph of $W$ (this will be important later on). A block $Z_{i,j}$ will be called \textit{bridged} if there is some leap in $\mathcal L'$ both of whose endpoints are enclosed by $Z_{i,j}$; in this situation, we will also say that this leap of $\mathcal L'$ constitutes a \textit{bridge}. There are now two main cases to consider:

\noindent\textbf{Case 1:} Each of the blocks is bridged. 

\noindent In this case, the graph $G_{\mathcal L}$ must have the crossed-$\boxtimes(n)\times\boxtimes(n)$-grid as a minor. Indeed, each of the $(\boxtimes (n)-1)^2$ bridged blocks, together with its corresponding bridge, can be used to replicate one of the $(\boxtimes(n)-1)^2$ cells of the crossed grid, along with its two crossing diagonals. This step is what drove us to define the inner subdivisions and the enclosed vertices as we did above: for each bridged block, the vertices near its boundary do not interact with the corresponding bridge. Afterwards, the fact that each $\overline{W_{i,j}}$ is a subdivision of a wall of size at least $7\times 7$ leaves us with more than enough room to connect these bridged blocks using only edges and vertices of $W$ so that they form a mesh, thus obtaining a crossed-$\boxtimes(n)\times\boxtimes(n)$-grid.
 
\noindent\textbf{Case 2:} There exists at least one block $Z_{i,j}$ which is not bridged.

\noindent We will again conclude that the crossed-$\boxtimes(n)\times\boxtimes(n)$-grid is a minor of $G_{\mathcal L}$, but the argument is more complicated in this case.
Recall that, by our choice of $a$, $b$ and $q$, every subdivision $\overline{W_{i',j'}}$ in $Z_{i,j}$ is $\mathcal L'$-connected. Since the leaps in $\mathcal L'$ are pairwise vertex-disjoint and $Z_{i,j}$ is not bridged, we can choose a set of leaps $\mathcal L_Z\subseteq \mathcal L'$ such that:
\begin{itemize}
    \item no two leaps in $\mathcal L_Z$ share a vertex;

    \item for each leap in $\mathcal L_{Z}$, one of its endpoints belongs to some inner subdivision $\overline{W_{i',j'}}$ of $Z_{i,j}$, while its other endpoint is a vertex of $W$ not enclosed by $Z_{i,j}$;

    \item for every inner subdivision $\overline{W_{i',j'}}$ of $Z_{i,j}$, there is some leap in $\mathcal L_Z$ with an endpoint in $\overline{W_{i',j'}}$.
\end{itemize}

Next, we further classify the inner subdivisions of $Z_{i,j}$ into \textit{sub-blocks} $X_{i',j'}$, where $1\leq i',j'\leq  \boxtimes(n)-1$. The sub-block $X_{i',j'}$ consists of all those subdivisions $\overline{W_{a+qk_1,b+qk_2}}$ with \[\frac{(i-1)K}{\boxtimes(n)-1}+1+(i'-1)Q\leq k_1\leq\frac{(i-1)K}{\boxtimes(n)-1}+i'Q\] and \[\frac{(j-1)K}{\boxtimes(n)-1}+1+(j'-1)Q\leq k_1\leq\frac{(j-1)K}{\boxtimes(n)-1}+j'Q\,.\] The sub-blocks of $Z_{i,j}$ form a $(\boxtimes(n)-1)\times (\boxtimes(n)-1)$ array, and each such sub-block consists of a $Q\times Q$ arrangement of subdivisions of the form $\overline{W_{a+qk_1,b+qk_2}}$. We remark that these sub-blocks do not quite form a partition of $Z_{i,j}$, since the subdivisions on the outer-most layer of $Z_{i,j}$ do not belong to any of the sub-blocks. Similarly as before, we say that a subdivision $\overline{W_{a+qk_1,b+qk_2}}$ is an \textit{inner} subdivision of $X_{i','j}$ if it belongs to this sub-block but is not part of the the outer-most layer of subdivisions within $X_{i',j'}$; there are $(Q-2)^2$ such subdivisions. Let $\mathcal L_{Z,\operatorname{in}}\subseteq\mathcal L_Z$ denote the set of leaps in $\mathcal L_{Z}$ which have an endpoint in some inner subdivision of a sub-block of $Z_{i,j}$.

Consider the subgraph $W_{\neg Z}$ of $W$ spanned by all those vertices which are not enclosed by $Z_{i,j}$. As we mentioned earlier, this graph is connected. Furthermore, every leap in $\mathcal L_{Z,\operatorname{in}}$ has an endpoint in $W_{\neg Z}$. Each vertex of $W_{\neg Z}$ that is an endpoint of some leap $P\in\mathcal L_{Z,\operatorname{in}}$ will be colored with one of $C:=(\boxtimes(n)-1)^2$ colors, depending on which of the $C$ sub-blocks $X_{i',j'}$ of $Z_{i,j}$ contains the other endpoint of $P$. Note that the $C$ sets of vertices that have been colored with each of the $C$ colors are pairwise disjoint, and each of them contain at least $(Q-2)^2$ elements. Since \[(Q-2)^2\geq2+(D-1) (C-1)\,,\] Proposition~\ref{lem:disjoint_leaps_2} implies the existence of $C$ vertex-disjoint paths $P_1,\dots,P_C$ in $W_{\neg Z}$ such that both endpoints of $P_i$ are of color $i$ (for all $1\leq i\leq C$).\footnote{These paths are not to be confused with the paths $P_i$ $(1\leq i\leq M^2)$ which appeared earlier in the proof.} See the picture on the left of Figure~\ref{fig:blocks_bridges_colors}.
\begin{figure}[ht]
    \centering
    \includegraphics[width=.95\linewidth]{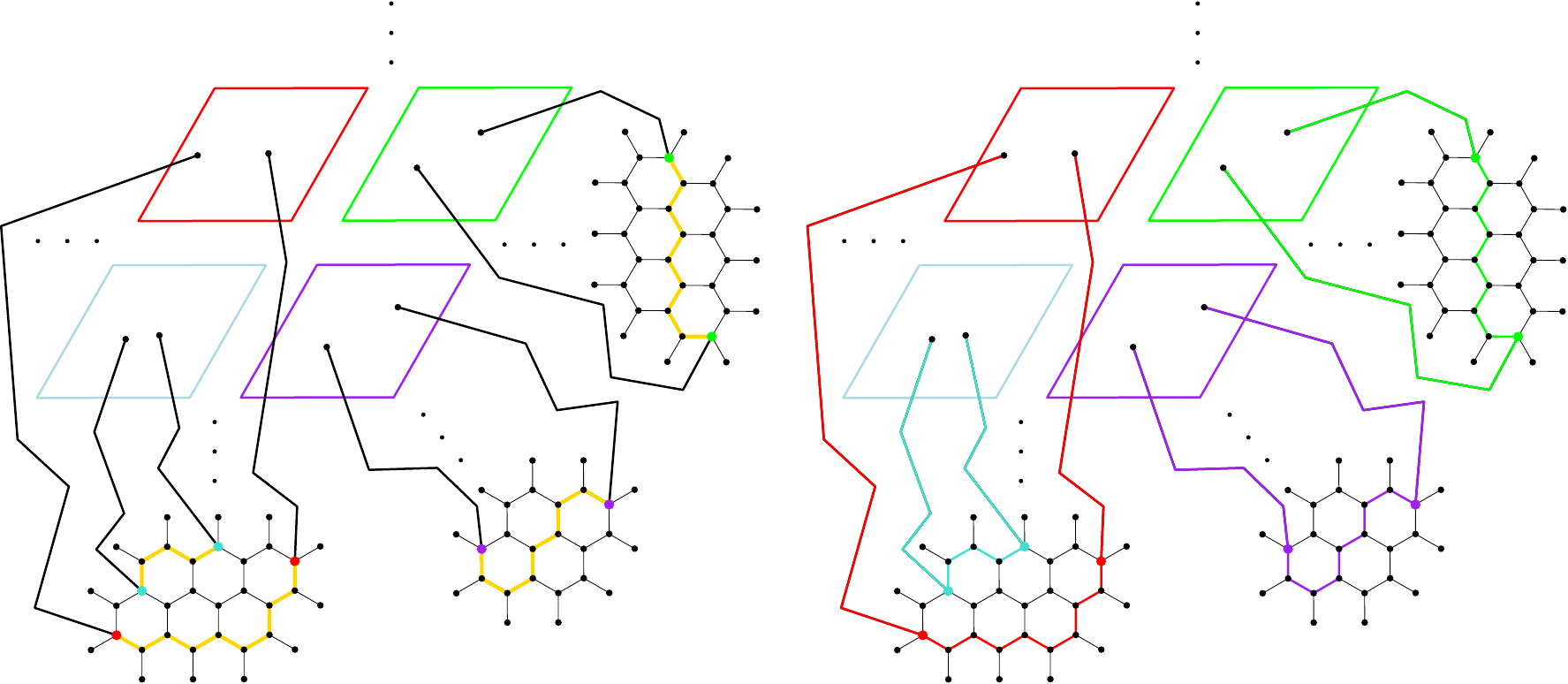}
    \caption{The picture on the left shows four sub-blocks (highlighted using different colors) which are connected to some vertices of $W_{\neg Z}$ via leaps from $\mathcal L_{Z,\operatorname{in}}$ (in black). The endpoints of these leaps that belong to $W_{\neg Z}$ have been colored according to which of the sub-blocks contains the other endpoint of the leap. Furthermore, some vertex-disjoint paths, which can be obtained by applying Proposition~\ref{lem:disjoint_leaps_2}, are shown in yellow. Note that each of these yellow paths connects two vertices of the same color. On the right, we have constructed a sort of bridge for each of the sub-blocks. Each of these bridges is obtained by concatenating three paths: a couple of leaps from $\mathcal L_{Z,\operatorname{in}}$, and the corresponding yellow path.}
    \label{fig:blocks_bridges_colors}
\end{figure}

For each sub-block $X_{i',j'}$ of $Z_{i,j}$, we let $c(i',j')$ denote the color corresponding to this block, and then construct a path $P_{i',j'}$ by concatenating three smaller paths: the path $P_{c(i',j')}$ and the two distinct leaps in $\mathcal L_{Z,\operatorname{in}}$ which connect the endpoints of $P_{c(i',j')}$ to distinct vertices within $X_{i',j'}$. The endpoints of the path $P_{i',j'}$ belong to some inner subdivisions of $X_{i',j'}$ and, apart from these two vertices, the path does not go through any vertex not enclosed by $Z_{i,j}$. Moreover, for any two distinct sub-blocks $X_{i',j'}$ and $X_{i'',j''}$ of $Z_{i,j}$, the paths $P_{i',j'}$ and $P_{i'',j''}$ are vertex-disjoint. Hence, we can think of the paths of the form $P_{i',j'}$ as disjoint bridges connecting pairs of points in the same sub-block; see the picture on the right of Figure~\ref{fig:blocks_bridges_colors}. Since there are $(\boxtimes(n)-1)^2$ sub-blocks, and they are arranged in a $(\boxtimes(n)-1)\times(\boxtimes(n)-1)$-grid-like manner, this is essentially the same situation that we encountered during Case 1. In particular, we can use the sub-blocks of $Z_{i,j}$ and these bridges to replicate the cells and diagonal edges of a crossed-$\boxtimes(n)\times\boxtimes(n)$-grid within $W$.

In either case, we conclude that the crossed-$\boxtimes(n)\times \boxtimes(n)$-grid is a minor of $G_\mathcal L$. Since $H$ is in turn a minor of this crossed grid (by Proposition~\ref{prop:crossed-grid_minors}), it must be the case that $H$ is a minor of $G_\mathcal L$. However, $G_{\mathcal L}$ is a subgraph of $G$, which is $H$-minor-free. This contradiction shows that the family of leaps $\mathcal L$ cannot exist, and thus $W$ cannot be contained within a ball of radius $r$ of $G_{D,\alpha}[\omega]$. This finishes the proof.
\end{proof}

\section{Proving soficity}\label{sec:limits}

In this section, we conclude the proof of Theorem~\ref{thm:main}. Let us briefly recapitulate what we have done up to this point.

At the start, we are given a unimodular random rooted graph $(G,\rho)$ with law $\mu$ which is a.s.\ one-ended and $H$-minor-free. Then, we took a site percolation $\xi_D$ which removes all vertices of degree larger than $D$ and considered some parameter $\alpha>0$. We applied Theorem~\ref{thm:non-amenable} to obtain another site percolation $\omega_\alpha$ of rate at most $\alpha$ which ensures that every connected component of the remaining graph, $G_{D,\alpha}=G_{\xi_D}[\omega_\alpha]$, is either finite or satisfies a certain isoperimetric inequality. We will refer to the percolation $\xi_D\cap\omega_{\alpha}$ simply as $\xi$ from now on. We chose two parameters $\varepsilon\in(0,1/2)$ and $N$ and, using Theorem~\ref{thm:filaments}, we constructed yet another site percolation, called $\omega$, of rate at most $\varepsilon$. We consider one last parameter $r$. By Theorem~\ref{thm:locally_treelike}, if $N$ is chosen appropriately (i.e., as long as $N\geq N(r)$), then every connected component of the graph $G_{D,\alpha}[\omega]$ which lies in an infinite component of $G_{D,\alpha}$ will be $(t(\varepsilon,n),r)$-locally-hin, where $n$ is the order of $H$ and $t$ is as given by the theorem. With this, we are ready to jump into the final argument in our proof.

\begin{proof}[Proof of Theorem~\ref{thm:main}]
Note, crucially, that $t(\varepsilon,n)$ is independent of $r$. This allows us to make $r$ (and $N$) as large as we want while keeping $\varepsilon$ and $t$ fixed. A standard compactness argument shows that, by letting $r\rightarrow\infty$ and passing to a subsequence, we can assume that the percolation $\omega$ converges to a limiting percolation $\omega_{\varepsilon,\infty}$. By this, we mean that the unimodular, random, rooted marked graph $(G,\rho,\xi,\omega)$ converges in the Benjamini-Schramm sense to $(G,\rho,\xi,\omega_{\varepsilon,\infty})$. Informally, this percolation $\omega_{\varepsilon,\infty}$ corresponds to the \textit{super-filaments} mentioned in Section~\ref{subsec:about-proofs}.

We claim that a.s.\ every connected component of $G_{D,\alpha}[\omega_{\varepsilon,\infty}]$ is either finite or has treewidth at most $t$. By theorems~\ref{lem:everything_root} and~\ref{thm:treewidth-walls}, it suffices to prove that the probability that the root $\rho$ belongs to a subdivision of a $g(t)\times g(t)$-wall which is in turn contained in an infinite connected component of $G_{D,\alpha}[\omega_{\varepsilon,\infty}]$ is $0$. Fix $G_{D,\alpha}$ and $\rho$ and note that, if this event were to occur, then $\rho$ would clearly also lie within an infinite connected component of $G_{D,\alpha}$. In this situation, the cluster of $\rho$ in $G_{D,\alpha}[\omega]$ must be $(t,r)$-locally-thin a.s.. Suppose there is some subgraph $W$ of $G_{D,\alpha}$ which contains $\rho$ and is isomorphic to a subdivision of $W_{g(t)\times g(t)}$. If $r$ is sufficiently large (say, larger than the order of $W$), then the probability---with respect to the choice of $\omega$---that $W$ is a subgraph of $G_{D,\alpha}[\omega]$ will be $0$, and thus the same must be true for $G_{D,\alpha}[\omega_{\varepsilon,\infty}]$. Taking a union bound over the (at most) countably many choices of $W$, we conclude that a.s.\ no such $W$ will be completely open with respect to $\omega_{\varepsilon,\infty}$. The claim follows.

Now, let $(G_\varepsilon,\rho_\varepsilon)$ denote the random rooted graph obtained from $(G,\rho,\xi,\omega_{\varepsilon,\infty})$ by conditioning on the event that $\rho$ is open with respect to both $\xi$ and $\omega_{\varepsilon,\infty}$, setting $\rho'=\rho$, and then letting $G_\varepsilon$ be the cluster $K_{\omega_{\varepsilon,\infty}}(\rho)$ of $\rho$ in $G_{D,\alpha}[\omega_{\varepsilon,\infty}]$. Denote the law of $(G_\varepsilon,\rho_\varepsilon)$ as $\mu_\varepsilon$. If $G'$ is finite a.s., then we immediately obtain that $(G',\rho')$ is sofic. Otherwise, we let $(G_\varepsilon',\rho_\varepsilon')$ be defined in the same way as $(G_\varepsilon,\rho_\varepsilon)$, except that we further condition on the event that the root belongs to an infinite cluster. By our above discussion, we have that $G_\varepsilon'$ a.s.\ has treewidth at most $t$. Hence, by Theorem~\ref{thm:treewidth-treeable}, $(G_\varepsilon',\rho_\varepsilon')$ must be treeable, and thus sofic. This also implies that $(G_\varepsilon,\rho_\varepsilon)$ is sofic. Indeed, we can write the law of $(G_\varepsilon,\rho_\varepsilon)$ as a mixture of two laws; one of these laws is sofic because it is supported on finite graphs, and the other one is simply the law of the treeable unimodular graph $(G_\varepsilon',\rho')$, which is also sofic. Soficity is preserved under taking mixtures.

Now  , we let $\varepsilon\rightarrow 0$. By Proposition~\ref{prop:percolation-stability}, the measure $\mu_{\varepsilon}$ will converge to $\mu_{\xi}$ in the the Benjamini-Schramm sense (recall that $\mu_\xi$ is the measure of the random rooted graph obtained from $(G,\rho)$ by conditioning on the event that $\rho$ is open with respect to $\xi$ and then looking at its cluster within this percolation). Since each $\mu_\varepsilon$ is sofic, the same must be true of $\mu_\xi$. Next, we let $\alpha\rightarrow 0$. Again by Proposition~\ref{prop:percolation-stability}, $\mu_\xi$ will converge to $\mu_{\xi_D}$, which must then also be sofic. Finally, we can let $D\rightarrow\infty$ to conclude that $\mu$ itself is sofic. In other words, the unimodular random rooted graph $(G,\rho)$ is sofic. 
\end{proof}

\section{Transitivity, tree decompositions, and soficity}\label{sec:transitive}

\subsection{Automorphisms, canonical tree decompositions, and separations}\label{subsec:transitive_definitions}
For a graph $G$, we denote by $\operatorname{Aut}(G)$ the automorphism group of $G$. A graph $G=(V,E)$ is said to be \textit{transitive} if, for any two vertices $x,y\in V$, there exists an automorphism of $G$ which maps $x$ to $y$. Similarly, $G$ is \textit{quasi-transitive} if the action of $\operatorname{Aut}(G)$ on $V$ has only finitely many orbits. Clearly, every Cayley graph is transitive, but there are transitive graphs which cannot be obtained as Cayley graphs of any group (see~\cite{eskin2012coarse} for a stronger result in this direction). Although this might initially appear surprising, there exist transitive graphs which are not unimodular (in the sense that the random rooted graph which a.s.\ equals this transitive graph rooted at an arbitrary vertex is not unimodular). We refer the reader to the book of Lyons and Peres~\cite[Chapter 8]{lyons2017probability} for a careful treatment of unimodularity in the setting of transitive graphs. 

A tree decomposition $(T,\mathcal V)$ of a quasi-transitive graph $G$ is said to be \textit{canonical} if $\operatorname{Aut}(G)$ induces a group action on $T$ such that, for every $\gamma\in \operatorname{Aut}(G)$, $V_t\cdot \gamma=V_{t\cdot \gamma}$. In other words, $(T,\mathcal V)$ is invariant under the action of $\gamma$, which maps bags of $(T,\mathcal V)$ to bags of the same decomposition. Note that $t\mapsto t\cdot \gamma$ must be an automorphism of $T$ for every $\gamma$ in $\operatorname{Aut}(G)$. 

A \textit{separation} of a connected graph $G=(V,E)$ is an ordered pair $(Y,Z)$ of subsets of $V$ such that $V=Y\cup Z$ and which satisfy the property that there is no edge of $G$ from $Y\backslash Z$ to $Z\backslash Y$. The set $S=Y\cap Z$ will be called the \textit{separator} of $(Y,Z)$. The separation is called \textit{proper} if at least one of the sets $Y$ and $Z$ is non-empty. Given a tree decomposition $(T,\mathcal V)$ of $G$, there is a proper separation of $G$ associated to every (oriented) edge $(t_1,t_2)$ of $T$. This separation is obtained by letting $T_1$ and $T_2$ be the connected components of $T-(t_1,t_2)$ which contain $t_1$ and $t_2$, respectively, and then taking $(Y_1,Y_2)$, where $Y_{i}=\bigcup_{t\in V(T_i)}V_{t}$ for $i\in\{1,2\}$. The separator corresponding to this separation will be $S=V_{t_1}\cap V_{t_2}$. Note that a tree decomposition can be recovered entirely from the collection of separations induced by all of its edges. 

Let us also describe a different way in which a separation can be encoded, which is actually the one we will use from now on. This alternative definition, which seems to have been introduced in~\cite{jardon2023applications} in order to study tree decompositions of Borel graphs (see the next section), has the advantage that it requires only local information. A \textit{separation} of $G=(V,E)$ can be encoded as a pair $(S,B)$, where $S$ is a finite subset of $V$, and $B\subset \partial_{E}^GS$ is a set of edges such that the following property holds: Suppose that $(u,v)\in \partial_{E}^G S$ and $(u',v')\subset B$ are edges with $u,u'\in S$ and $v,v'\not\in S$, and that $v$ and $v'$ lie in the same connected component of $G\backslash S$. Then, $(u,v)\in B$. Once again, $S$ will be called the \textit{separator} of $(S,B)$. It is straightforward to translate back and forth between the two definitions of a separation that we have given (see~\cite[Remark 2.2]{jardon2023applications} for details), and the separator $S$ is independent of the encoding being used. 

Suppose we are given a canonical tree decomposition $(T,\mathcal V)$ of some quasi-transitive, connected, locally finite graph $G$. Furthermore, assume that every vertex belongs only to finitely many bags from $\{V_t\}_{t\in V(T)}$. How can we encode $(T,\mathcal V)$ in a way that resembles a marking? It is tempting to do this simply by storing at every vertex $v$ of $G$ the collection $\mathcal S_v$ of all separations induced by $(T,\mathcal V)$ which have $v$ in their corresponding separator. Since every vertex belongs to finitely many bags, a finite amount of information must be stored at every vertex. While this encoding can be very complicated for general tree decompositions, it turns out to be quite nice in the canonical case. Indeed, the whole encoding can be recovered simply by knowing the information stored at one representative vertex from every orbit of vertices in $G$ (with respect to the action of $\operatorname{Aut}(G)$). To see this, consider two vertices $u$ and $v$ of $G$ for which there is some $\gamma\in \operatorname{Aut}(G)$ with $u\cdot \gamma=v$. For each bag $V_t$ that contains $u$, the bag $V_{t}\cdot \gamma=V_{t\cdot\gamma}$ contains $v$. Similarly, if some bag $V_t$ contains $v$, then the bag $V_t\cdot\gamma^{-1}=V_{t\cdot\gamma^{-1}}$ contains $u$. Hence, the set $\mathcal S_{v}$ can be obtained from $\mathcal S_u$ simply by letting $\gamma$ act from the right on each of the separations contained in $\mathcal S_u$.

\subsection{Borel graphs and graphings}\label{subsec:graphings}
A \textit{Borel graph} consists of a Borel-measurable subset $\mathcal G\subseteq X\times X$, where $X$ is some standard Borel space. In this definition, one can think of $X$ and $\mathcal G$ as the sets of vertices and edges of the Borel graphs, respectively. Each Borel graph induces an equivalence relation $\operatorname{Rel}(\mathcal G)$ on its vertex space $X$, where two vertices belong to the same equivalence class if and only if they belong to the same connected component with respect to $\mathcal G$.

A locally finite Borel graph $\mathcal G$ defined on a standard Borel space $X$ equipped with a probability measure $\lambda$ is said to be a \textit{graphing} if, for every two measurable sets $A,B\subseteq X$, 
\begin{equation}\label{eq:graphing}
    \int_A\operatorname{deg_B}(x)d\lambda(x)=\int_B\operatorname{deg_A}(x)d\lambda(x)\,,
\end{equation}
where $\operatorname{deg}_A(x)$ and $\operatorname{deg}_B(x)$ denote the number of edges joining $x$ to $A$ and $B$, respectively. It is not hard to see that both integrals above are well defined. Given a graphing $\mathcal G$ on $(X,\lambda)$, one can obtain a random rooted graph $(G,\rho)$ by letting the root be a random point $\rho\in X$ chosen according to the measure $\lambda$, and then taking the subgraph $G$ of $\mathcal G$ induced by the connected component of $\rho$. Under this correspondence, it turns out that~\eqref{eq:graphing} is actually equivalent to the mass transport principle~\eqref{eq:MTP}, so every random rooted graph produced in this manner is unimodular. Furthermore, every unimodular random rooted graph can be represented as a graphing (although not in a unique way). We refer the reader to the book of Lov{\'a}sz~\cite[Chapter 18 and, in particular, Theorem 18.37]{lovaszlimits} for the proofs of these statements and further details regarding the correspondence between unimodular graphs and graphings.

The equivalence relation $\operatorname{Rel}(\mathcal G)$ induced by a Borel graph $\mathcal G$ on a space $X$ is said to be \textit{Borel treeable} if there exists some acyclic Borel graph $\mathcal T$ with vertex set $X$ such that $\operatorname{Rel}(\mathcal G)=\operatorname{Rel}(\mathcal T)$. Furthermore, for a probability measure $\mu$ on $X$, $\operatorname{Rel}(\mathcal G)$ is said to be $\mu$\textit{-treeable} if there exists a set $X_0\subset X$ with $\mu(X_0)=1$ so that the restriction of $\operatorname{Rel}(\mathcal G)$ to $X_0$ is Borel treeable. We say that $\operatorname{Rel}(G)$ is \textit{measure treeable} if it is $\mu$-treeable for every Borel probability measure $\mu$ on $X$. Suppose that $(G,\rho)$ is a unimodular random rooted graph and $\mathcal G$ is a graphing which represents it on a probability space $(X,\lambda)$. Then, if $\operatorname{Rel}(\mathcal G)$ is $\lambda$-treeable, $(G,\rho)$ must be treeable too (as defined in Section~\ref{subsec:treeable}).

Given a standard Borel space $A$, we denote by $[A]^{\leq n}$ the standard Borel space of all subsets of $A$ with cardinality at most $n$, and we let $[A]^{<\infty}:=\bigcup_{n=1}^\infty [A]^{\leq n}$. A collection $\{(T_i,\mathcal V_i)\}_{i\in \mathcal I}$ containing one tree decomposition of each of the connected components of a Borel graph $\mathcal G$ on $X$ will be said to be Borel if the set of separations it induces is a Borel subset of $[X]^{<\infty}\times[\mathcal G]^{<\infty}$, (recall that we think of separations as pairs $(S,B)$). If the adhesion of each tree decomposition in this collection is bounded by a constant, then we can equip the set \[Z:=\left\{(x,t)\in X\times \bigcup_{i\in \mathcal I}V(T_i)\mid x\in V_t\right\}\] with a standard Borel space structure in a natural way (see Proposition 3.1 in~\cite{jardon2023applications}). In this setting, we define the \textit{levels equivalence relation} $Q$ of $\{(T_i,\mathcal V_i)\}_{i\in \mathcal I}$ as the equivalence relation on $Z$ with $(x,t)\sim_Q(y,t')$ if and only if $t=t'$. The equivalence classes of $Q$ will be called \textit{levels}. We remark that, in~\cite{jardon2023applications}, the levels equivalence relation is actually defined not for tree decompositions, but rather in terms of separation systems satisfying a certain property called $(*)$. However, Proposition 2.6 in that same paper implies that these definitions carry over to any Borel collection of tree decompositions with bounded adhesion. 

\subsection{Relevant results}
Let us now go over some results from~\cite{jardon2023applications} and~\cite{esperet2024structure} regarding treeability, tree decompostions, and the structure of quasi-transitive graphs excluding some minor.

Recall the definition of accessibility given in Section~\ref{subsec:ends}. The first result, due to Esperet, Giocanti, and Legrand-Duchesne, tells us that every quasi-transitive, minor-excluded graph is accessible.

\begin{theorem}[Theorem 1.4\cite{esperet2024structure}]\label{thm:accessible_minor-free}
    Every locally finite quasi-transitive graph which does not have the infinite clique $K_\infty$ as a minor is accessible.
\end{theorem}

In that same paper, it was also proven that quasi-transitive graphs excluding some minor admit highly structured tree decompositions. This is the content of the following theorem. 

\begin{theorem}[Theorem 4.3~\cite{esperet2024structure}]\label{thm:planar_torsos}
    Let $G$ be a quasi-transitive graph which does not have $K_{\infty}$ as a minor. Then, there exists some positive integer $m$ such that $G$ admits a canonical tree decomposition $(T,\mathcal V)$ with $\mathcal V=\{V_t\}_{t\in V(T)}$, of adhesion at most $3$, and whose torsos $G[V_t]$ are quasi-transitive, connected, and are either planar or have treewidth at most $m$.
\end{theorem}

\noindent\textit{Remark.} In~\cite{esperet2024structure}, this theorem is stated without the connectedness property on the torsos. However, upon inspection of the proof, the result can be strengthened in this way. We can actually assume that the tree decomposition satisfies yet another property, namely, that each vertex belongs only to finitely many bags $V_t$. If this property were not satisfied, then we could modify the decomposition by selecting one of the orbits of vertices in $G$, simultaneously removing each vertex in this orbit from infinitely many bags, and then repeating this process for each of the orbits.

\medskip

Theorem~\ref{thm:planar_torsos} suggests that quasi-transitive minor-excluded graphs are similar in structure to planar graphs. This intuition can be made precise, and it has been shown in~\cite{esperet2023coarse} that every quasi-transitive minor-excluded graph is quasi-isometric to a planar graph of bounded degree; we will not require this result, but consider it worth mentioning. 

Next, we recall two results by Jardón-Sánchez. The first one concerns the treeability of accessible planar graphs.

\begin{theorem}[Theorem 2~\cite{jardon2023applications}]\label{thm:planar_accessible_treeable}
    Suppose that $\mathcal G$ is a locally finite Borel graph whose connected components are planar and accessible. Then, $\operatorname{Rel}(\mathcal G)$ is measure treeable. 
\end{theorem}

The next theorem will allow us to pass from studying an entire Borel graph to studying only the levels equivalence relation of some Borel collection of tree decompositions of its connected components.

\begin{theorem}[Theorem 3.4 $+$ Proposition 2.6~\cite{jardon2023applications}]\label{tmh:jardon}
    Let $\mathcal G$ be a locally finite Borel graph on some standard Borel space $X$, and let $\{(T_i,\mathcal V_i)\}_{i\in \mathcal I}$ be a Borel collection of tree decompositions of the connected components of $\mathcal G$. Suppose there exists some positive integer $m$ such that, for every $i\in \mathcal I$, $(T_i,\mathcal V_i)$ has adhesion at most $m$. Then, if the levels equivalence relation $Q$ of $\{(T_i,\mathcal V_i)\}_{i\in \mathcal I}$ is measure treeable, so is $\operatorname{Rel}(\mathcal G)$.
\end{theorem}

\noindent\textit{Remark.} Once again, in~\cite{jardon2023applications}, this result is stated not for tree decompositions, but for of separation systems satisfying a property called $(*)$. However, Proposition 2.6 in the same paper tells us that the result holds for any separation system induced by a Borel collection of tree decompositions with bounded adhesion.

\subsection{Transitive minor-excluded graphs}

We are ready to prove Theorem~\ref{thm:main-transitive}.

\begin{proof}[Proof of Theorem~\ref{thm:main-transitive}]
    We begin with a unimodular random rooted graph $(G,\rho)$ as in the statement of the theorem. Since the law of any unimodular random graph can be written as a mixture of ergodic laws of unimodular random graphs, and because soficity and treeability are preserved under taking mixtures, we may and will assume that the law of $(G,\rho)$ is ergodic (see Section~\ref{subsec:unimodular-sofic}). For every quasi-transitive graph $G$, there exists a positive integer $k$ such that any two vertices $u$ and $v$ of $G$ belong to the same orbit under the action of $\operatorname{Aut}(G)$ if and only if $B_G(u,r)$ and $B_G(v,r)$ are isomorphic as rooted graphs; let $\operatorname{rad}(G)$ denote the smallest positive integer with this property. The function $\operatorname{rad}(G)$ is measurable with respect to the topology of $\mathcal G_*$ so, by ergodicity, we may assume that there exists a single integer $k$ such that a.s.\ $\operatorname{rad}(G)\leq k$---note that here we have used the fact that there are only countably many values that $\operatorname{rad}(G)$ can take. Also by ergodicity, and because quasi-transitive graphs have finite maximum degree, there exists a positive integer $D$ such that $G$ a.s.\ has maximum degree at most $D$.
    
    We will show that there is a marking of $(G,\rho)$ which encodes a tree decomposition $(T,\mathcal V)$ of $G$ as in the statement of Theorem~\ref{thm:planar_torsos}. As was already discussed in Section~\ref{subsec:transitive_definitions}, this marking will store at every vertex $v$ the separations induced by $(T,\mathcal V)$ whose separator contains $v$. For every positive integer $r$, let $\mathcal G_*^{r,D}\subset\mathcal G_*$ be the set of (isomorphism classes of) finite rooted graphs with maximum degree at most $D$ and with all vertices at distance at most $r$ from the root; note that $\mathcal G_*^{r,D}$ is finite. For every positive integer $r$, consider the family $\mathcal F_{r}$ of all functions with domain $\mathcal G_*^{r,D}$ which, for every $(H,o)\in\mathcal G_*^{r,D}$, output a collection $\{(S_i,B_i)\}_{i\in \mathcal I}$ of distinct separations of $H$ such that $v\in S_i$ for all $i\in \mathcal I$. Since $\mathcal G_*^{r,D}$ is finite and all of its elements are finite graphs, the family $\mathcal F_r$ has finitely many elements too. We can think of each element of $\mathcal F_r$ as a local algorithm that can be used to assign a collection of separations to each vertex of a graph.
    
    Let $G$ be a graph drawn from our ergodic, unimodular measure. Then, a.s.\ $G$ is a quasi-transitive graph of maximum degree at most $D$, which does not have $K_\infty$ as a minor, and satisfies $r(G)\leq k$. Consider a tree decomposition $(T,\mathcal V)$ of $G$ which satisfies the properties in the statement of Theorem~\ref{thm:planar_torsos} for some parameter $m$; any such decomposition will be called $m$-\textit{pristine}. This tree decomposition can be encoded by recording, for each vertex $v$ of $G$, the collection $S_v$ of separations induced by $(T,\mathcal V)$ that have $v$ in their separator. Furthermore, as we discussed near the end of Section~\ref{subsec:transitive_definitions}, if $u$ and $v$ are vertices in the same orbit of $G$ and $\gamma\in \operatorname{Aut}(G)$ maps $u$ to $v$, then $S_v$ can be obtained by letting $\gamma$ act on $S_u$. As every vertex belongs to finitely many bags, each $S_v$ consists of finitely many separations. Now, since $(T,\mathcal V)$ has finite adhesion, there must exist a positive integer $d$ such that any two vertices in the separator of a separation induced by $(T,\mathcal V)$ are at distance at most $d$. Together with the fact we can discover which orbit a vertex of $G$ lies in by observing a sufficiently large ball around it, we conclude that there exists some positive integer $R$ and some element $\mathfrak M_G$ of $\mathcal F_R$ such that, for every vertex $v$ of $G$, $S_v$ can be obtained by applying $\mathfrak M_G$ to $B_G(v,R)$. 
    
    For every $\mathfrak M\in\bigcup_{n=1}^{\infty}\mathcal F_n$ and every positive integer $m$, the event that $\mathfrak M$ can be used at every vertex of $G$ to produce a collection of separations inducing an $m$-pristine tree decomposition is measurable with respect to the local topology on $\mathcal G_*$. Indeed, if $\mathfrak M$ fails to produce such a collection of separations, then this must become evident after observing a sufficiently large but finite ball around the root. Since the set $\mathbb N\times\bigcup_{n=1}^{\infty}\mathcal F_n$ is countable, the ergodicity of $(G,\rho)$ tells us that there exist a single positive integer $m$ and a single element $\mathfrak M\in\bigcup_{n=1}^\infty\mathcal F_n$ such that---a.s.\ over the randomness of $(G,\rho)$---$\mathfrak M$ produces a collection of separations inducing an $m$-pristine tree decomposition of $G$. We can think of the local algorithm $\mathfrak M$ as a marking.

    Now that we know that such a marking exists, we switch to the setting of Borel graphs so as to be able to apply the machinery from~\cite{jardon2023applications}. Let us denote by $\mathcal G$ and $(X,\lambda)$ a graphing corresponding to $(G,\rho)$ and the probability space on which it is defined, respectively (see Section~\ref{subsec:graphings}). By (possibly) passing to a subset of $X$ which has measure $1$ with respect to $\lambda$, may assume that every connected component of $\mathcal G$ is a quasi-transitive $K_\infty$-minor-free graph on which $\mathfrak M$ encodes an $m$-pristine tree decomposition. The marking described above translates into a Borel collection $\{(T_i,\mathcal V_i)\}_{i\in \mathcal I}$ of tree decompositions of the connected components of $\mathcal G$, and each decomposition in this collection is $m$-pristine.

    With $\{(T_i,\mathcal V_i)\}_{i\in \mathcal I}$ in hand, we can now attempt to use Theorem~\ref{tmh:jardon}. Towards this goal, we must understand the levels equivalence relation $Q$ that is derived from $\{(T_i,\mathcal V_i)\}_{i\in \mathcal I}$. As before, we denote by $Z$ the standard Borel space of pairs $(x,t)$ with $x\in X$, $t\in \bigcup_{i\in \mathcal I}V(T_i)$, and $x\in V_t$. Also, we let $\mathcal G_Q$ be the Borel graph on $Z$ where two vertices $(x,t)$ and $(y,t')$ form an edge if and only if $t=t'$ and $x$ and $y$ are adjacent in the torso corresponding to $t$. Since each tree decomposition $(T_i,\mathcal V_i)$ has connected torsos, $Q=\operatorname{Rel}(\mathcal G_Q)$. Furthermore, the torsos of each $(T_i,\mathcal V_i)$ are all quasi-transitive, and each of them is either planar or has treewidth at most $m$. It follows that $\mathcal G_Q$ can be written as a disjoint union of two Borel graphs: one whose connected components have treewidth at most $m$, and one whose connected components are quasi-transitive and planar. By the results in~\cite{jardon2023applications, chen2023tree} (see Theorem 1.4 in~\cite{chen2023tree}, as well as Theorem~\ref{thm:treewidth-treeable}), the first of these Borel graphs induces a Borel treeable equivalence relation. By Theorem~\ref{thm:accessible_minor-free}, the connected components in the second of these Borel graphs are not only planar but also accessible. Thus, Theorem~\ref{thm:planar_accessible_treeable} tells us that the second Borel graph induces a measure treeable equivalence relation. It follows that $Z=\operatorname{Rel}(\mathcal G_Q)$ is itself measure treeable and, finally, we can use Theorem~\ref{tmh:jardon} to deduce that $\operatorname{Rel(G)}$ is measure treeable. This concludes the proof, as it yields that $(G,\rho)$ is treeable and sofic.
\end{proof}

\bibliographystyle{plain}
\bibliography{refs}

\newpage
\section*{Appendix}\label{sec:Appendix}
\subsection*{Minors within crossed grids}

\begin{proof}[Proof of Proposition~\ref{prop:crossed-grid_minors}] 
    Take some drawing of $K_n$ on the plane so that no three edges cross at a single point, no two edges cross more than once, and no two edges sharing an endpoint cross. By adding a \textit{dummy} vertex at each crossing between two edges, we can transform this drawing into an embedded planar graph $K_n'$ which has at most $n+\binom{n}{2}<n^2$ vertices. Let $V_1$ and $V_2$ denote the sets of non-dummy and dummy vertices of $K_n'$, respectively, and write $V:=V_1\cup V_2$. It is well known that, for any sufficiently large integer $g$, the $g\times g$-grid $G_{g\times g}$ will contain $K_n'$ as a minor~\cite{robertson1986graph}. Consider any such $g$ and let $\{U_v\}_{v\in V(K_n')}$ be a family of pairwise disjoint sets of vertices of $G_{g\times g}$ which act as witnesses to the fact that $K_n'\preceq G_{g\times g}$ (see the definition of graph minors in Section~\ref{subsec:structural}). An example of this for $n=5$ is shown in Figure~\ref{fig:dummy_vertex}.
\begin{figure}[ht]
    \centering
    \includegraphics[width=.6\linewidth]{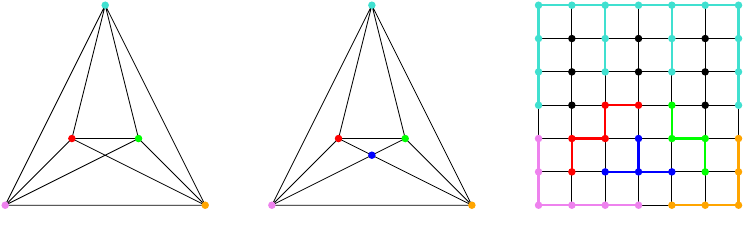}
    \caption{The picture on the left shows a drawing of the complete graph $K_5$ on the plane. Note that this drawing contains a pair of crossing edges. In the middle, we can see that the crossing point has been substituted by a dummy vertex (blue) to obtain the graph $K_5'$. On the left, we have highlighted some some sets of vertices on the grid $G_{7\times 7}$ (or, more accurately, the subgraphs they induce). These sets act as witnesses to the fact that $K_5'\preceq G_{7\times 7}$.}
    \label{fig:dummy_vertex}
\end{figure}

    Now, we consider a finer $(3g-2)\times (3g-2)$-grid, $G'$, which is obtained by subdividing each cell of $G_{g\times g}$ into a $4\times 4$-grid. By subdividing all the edges in each of the graphs $\{G_{g\times g}[U_v]\}_{v\in V(K_n')}$ into paths of length $3$, we obtain new sets of vertices $\{U_v'\}_{v\in V(K_n')}$. Next, for each $v\in V(K_n')$, we define $U''_v$ as the set of vertices of $G'$ which are at distance at most $1$ from at least one vertex in $U_v'$. Note that the sets $\{U_v''\}_{v\in V(K_n')}$ are pairwise disjoint, and that they act as witnesses to the fact that $K_n'\preceq G'$. An example of this is shown in Figure~\ref{fig:subdividing_grid}.
\begin{figure}[ht]
    \centering
    \includegraphics[width=.6\linewidth]{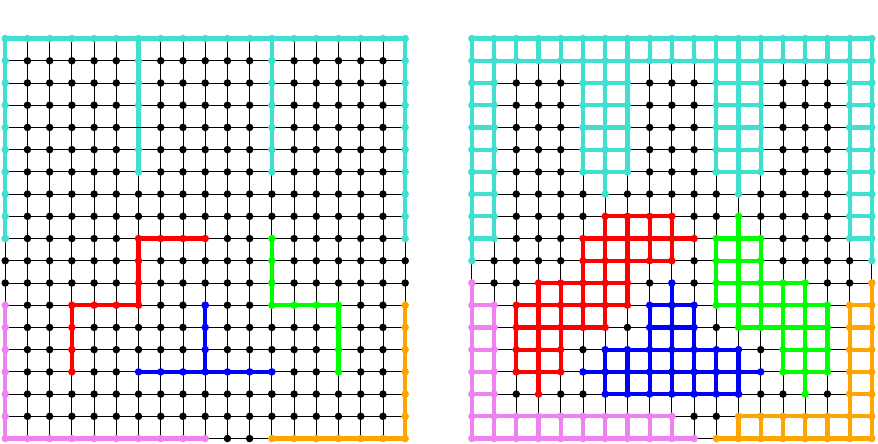}
    \caption{This is a continuation of the example in the previous figure. On the left, we see the result of subdividing the grid $G_{7\times 7}$---along with the highlighted sets---to obtain a $19\times 19$-grid $G'$. On the right, we can observe the result of adding to each highlighted set all those vertices which lie at distance at most $1$. The new highlighted sets act as witnesses to the fact that $K_5'\preceq G'$.}
    \label{fig:subdividing_grid}
\end{figure}
    
    Let $G''$ denote the $(3g-2)\times (3g-2)$-crossed-grid that arises from $G'$ by adding the two diagonal edges within each cell. We will prove that $K_n\preceq G''$. Each dummy vertex $u$ of $K_n'$ can be traced back to a crossing between two edges $(v_1,v_2)$ and $(v_3,v_4)$ of $K_n$. Inside the graph $G''[U''_u]$, we can find four distinct vertices $x_1,\dots,x_4$ which are adjacent to $U''_{v_1},\dots,U''_{v_4}$ in $G''$, in that order. By (possibly) using a pair of crossing diagonal edges, we can find two vertex-disjoint paths $P_1$ and $P_2$ within $G''[U_u'']$ which connect $x_1$ to $x_2$ and $x_3$ to $x_4$, respectively. Next, we forget about the set $U_u''$, and instead add the vertices in $P_1$ to $U_{v_1}''$, and the vertices of $P_2$ to $U_{v_3}''$. Note that this operation preserves the disjointness of the vertex sets, and makes it so that there is an edge in $G''$ joining $U_{v_1}''$ to $U_{v_2}''$, as well as an edge joining $U_{v_3}''$ to $U_{v_4}''$. After applying this operation for every dummy vertex $u$, we will have produced a family of sets of vertices of $G''$ which act as witnesses to the fact that $K_n\preceq G''$, as promised. See Figure~\ref{fig:fixing_dummies}.
\begin{figure}[ht]
    \centering
    \includegraphics[width=.6\linewidth]{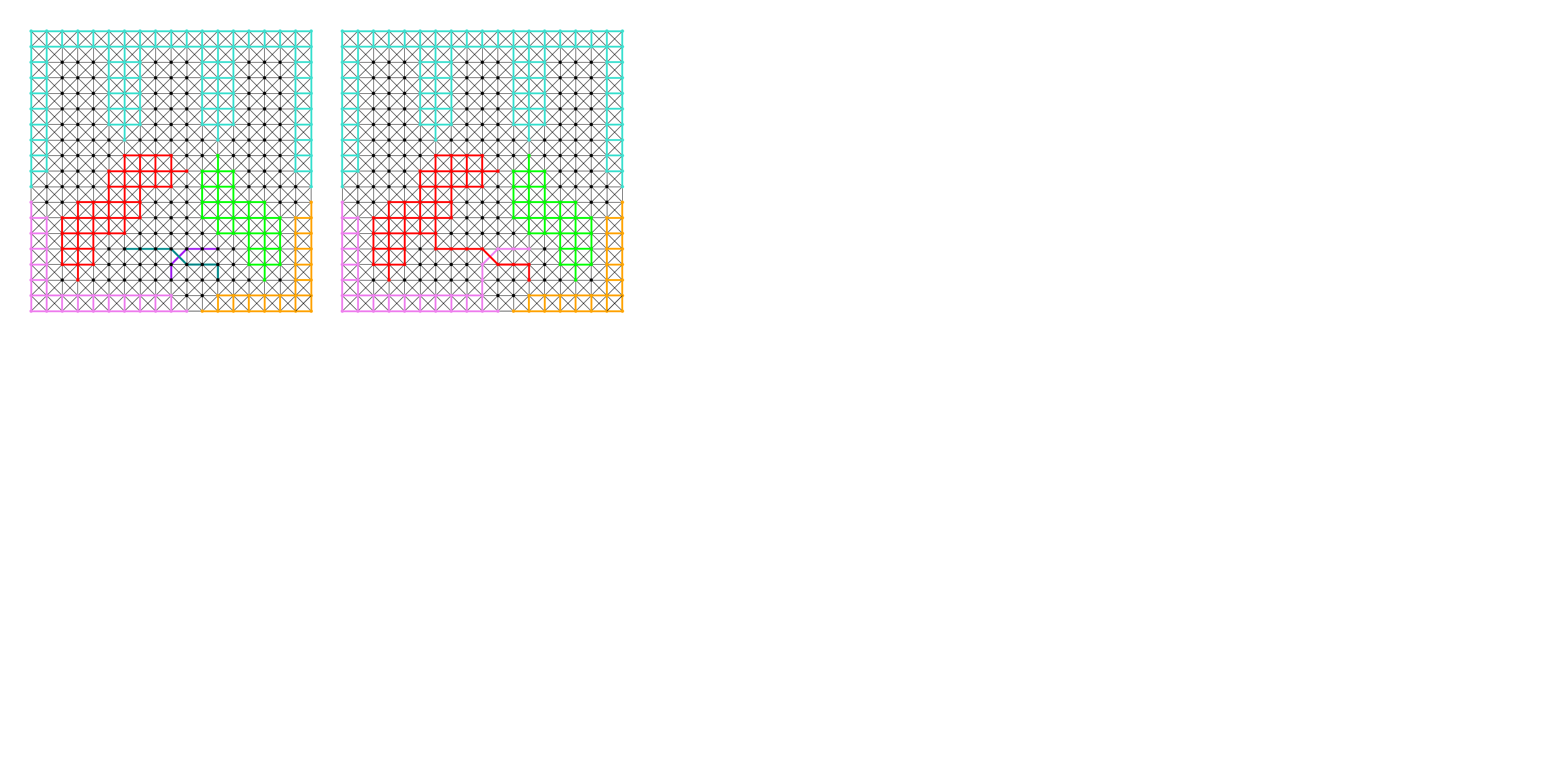}
    \caption{On the left, we see the graph $G''$. The set of vertices that was originally representing the blue dummy vertex has been replaced by two paths, which connect the red and orange sets, and the violet and green sets. On the right, these two paths have been appended to the red and violet vertex sets, respectively. The new highlighted sets act as witnesses to the fact that $K_5\preceq G''$.}
    \label{fig:fixing_dummies}
\end{figure}
\end{proof}

\subsection*{Szemeredi's Theorem}\label{subsec:Szemeredi-Furstenberg}

Here, we provide some background regarding the multi-dimensional version of Szemeredi's theorem due to Furstenberg and Katznelson. The original, one-dimensional version, first proven by Szemeredi~\cite{szemeredi1975sets}, states that for every constant $\delta\in(0,1]$ and every positive integer $k$, there exists an integer $N=N(\delta,k)$ so that every set $A\subseteq [N]$ of size at least $\delta n$ contains an arithmetic progression of length $k$. Here, by an \textit{arithmetic progression of length }$k$ we mean a set of $k$ integers which can be written as $\{a,a+d,\dots,a+(k-1)\}$. Of course, the conclusion still holds if we substitute $N$ by any other larger integer. This result is one of the cornerstones of the area which is nowadays often referred to as \textit{additive combinatorics}. A few years after the publication of Szemeredi's result, Furstenberg and Katznelson~\cite{furstenberg1978ergodic} generalized it as follows.

\begin{theorem}\label{thm:generalized-szemeredi}
    Fix a positive integer $d$. Then, for every positive integer $k$, every configuration $F=\{v_1,v_2,\dots,v_k\}$ consisting of $k$ distinct elements in $\mathbb Z^d$, and every $\delta\in(0,1]$, there exists an integer $N=N(k,F,\delta)$ such that the following holds:

    For every set $A\subseteq [N]^d$ with $|A|\geq \delta N^d$, we can find a point $u\in [N]^d$ and a positive integer $q$ such that the points $u+qv_1,u+qv_2,\dots,u+qv_k$ all belong to $A$.
\end{theorem}

In some sense, one can think of the points $u+qv_1,u+qv_2,\dots,u+qv_k$ as forming a generalized arithmetic progression. We remark that the original proof of the above theorem does not yield any effective bounds on the size of $N$, but such bounds were eventually obtained through very different methods (see, for example,~\cite{gowers2007hypergraph}).

The way Theorem~\ref{thm:generalized-szemeredi} is applied during the proof of Theorem~\ref{thm:locally_treelike} is by setting $k=K^2$, $F=[K]^2$, $\delta=1/(4D)$, and then taking any $M$ with $M\geq N(k,F)$. If $A$ corresponds to the set of pairs $(i,j)$ such that $\overline{W_{i,j}}$ is $\mathcal L'$-connected, the above theorem guarantees the existence of a $K\times K$-grid-like structure within $A$.
\end{document}